\documentclass[reqno,12pt]{amsart}

\numberwithin{equation}{section}

\usepackage{amsmath}

\usepackage{esint} 

\usepackage{amsthm}

\usepackage{verbatim}

\usepackage{calligra}

\usepackage{epsfig}

\usepackage{psfrag}

\usepackage{graphicx}

\usepackage{graphpap,latexsym,epsf}

\usepackage{color}

\usepackage{amssymb,mathrsfs,enumerate}

\usepackage{endnotes}

\usepackage{calligra}

\usepackage{mathtools}

\usepackage{pgfplots}

\usepackage{pgfplots}
\pgfplotsset{compat=1.18}

\usepackage[colorinlistoftodos]{todonotes} 

\definecolor{citegreen}{rgb}{0,0.6,0}

\definecolor{refred}{rgb}{0.8,0,0}

\usepackage[colorlinks, citecolor=citegreen, linkcolor=refred]{hyperref}

\usepackage{enumitem}

\usepackage{mathtools}
\mathtoolsset{showonlyrefs}

\newcommand{\R}{\mathbb{R}}

\newcommand{\N}{\mathbb{N}}

\newcommand{\Z}{\mathbb{Z}}

\newcommand{\SSS}{\mathbb{S}}

\def\HHH{{\rm H}}

\def\RRR{{\mathrm R}}

\def\a{\alpha}

\newcommand{\pa}{\partial}

\newcommand{\ep}{\varepsilon}

\newcommand{\mmp}{\mathfrak{m}^{\!(p)}}

\newcommand{\na}{\nabla}

\newcommand{\hhh}{{\rm h}}

\mathchardef\emptyset="001F

\DeclareMathOperator{\divg}{div}
\DeclareMathOperator{\Per}{Per}
\newcommand{\dmu}{\,d\mu_g}
\newcommand{\dsigma}{\,d\sigma}
\newcommand{\loc}{\mathrm{loc}}
\newcommand{\Lip}{\mathrm{Lip}}
\newcommand{\dist}{\mathrm{dist}}
\newcommand{\calH}{\mathcal H}
\newcommand{\eps}{\varepsilon}

\definecolor{vgreen}{rgb}{0.1,0.5,0.2}

\definecolor{viola}{RGB}{85,26,139}

\newtheorem{theorem}{Theorem}[section]

\newtheorem{corollary}[theorem]{Corollary}

\newtheorem{definition}[theorem]{Definition}

\newtheorem*{conjecture}{Conjecture}

\newtheorem{proposition}[theorem]{Proposition}

\newtheorem{remark}[theorem]{Remark}

\newtheorem{lemma}[theorem]{Lemma}

\newtheorem{ackn}{Acknowledgement
\hspace{-.4cm}}

\definecolor{byzantium}{rgb}{0.44, 0.16, 0.39}

\definecolor{amber}{rgb}{1.0, 0.75, 0.0}

\definecolor{darkmagenta}{rgb}{0.55, 0.0, 0.55}

\definecolor{fuzzywuzzy}{rgb}{0.8, 0.4, 0.4}

\definecolor{brown}{rgb}{0.2, 0.08, 0.08}

\definecolor{arancio}{rgb}{1.0, 0.13, 0.0}

\begin{document}
\title[Mass-type invariants with  a cosmological constant ]{Mass-type invariants\\ in the presence of a cosmological constant}
\author[V.~Agostiniani]{Virginia Agostiniani}
\address{V.~Agostiniani, Universit\`a degli Studi di Trento,
via Sommarive 14, 38123 Povo (TN), Italy}
\email{virginia.agostiniani@unitn.it}
\author[S.~Borghini]{Stefano Borghini}
\address{S.~Borghini, Universit\`a degli Studi di Napoli Federico II,
Via Cintia, Monte S. Angelo, 80126 Napoli (NA), Italy}
\email{stefano.borghini@unina.it}
\author[L.~Mazzieri]{Lorenzo Mazzieri}
\address{L.~Mazzieri, Universit\`a degli Studi di Trento,
via Sommarive 14, 38123 Povo (TN), Italy}
\email{lorenzo.mazzieri@unitn.it}

\begin{abstract}
In this paper, we introduce a new family of mass-type invariants for time-symmetric initial data in space-times satisfying the Dominant Energy Condition. For positive cosmological constant, these invariants,
unlike the total Hawking mass, turn out to be genuinely effective in providing new characterizations of the de Sitter solution. From a theoretical standpoint, this opens a new perspective on how one might refine the rigidity statement originally proposed by Min-Oo in his well known conjecture, later refuted by the counterexamples of Brendle, Marques, and Neves. Via a formal limiting procedure, we also define another invariant, the $1$-harmonic Mass, for which we independently prove a positive mass theorem and a Penrose-type inequality, thereby extending tools for probing space-time geometries in the presence of a positive cosmological constant.
\end{abstract}

\maketitle

\bigskip



\section{Introduction and statement of the main results.}
\label{effmon}

\subsection{The state of the art.}

To introduce the foundational setting of our investigation and to establish the notation and conventions adopted throughout, we recall that, in the Cauchy formulation of Einstein’s equations, an initial data set for the Einstein Field Equations is given by a triple \((M,g,k)\), where \((M,g)\) is a three-dimensional Riemannian manifold and \(k\) is a symmetric \((2,0)\)-tensor field satisfying the Einstein constraint equations
\begin{equation}
\begin{cases}
\RRR - 2 \Lambda - |k|^2 + {\rm K}^2   = \, 16 \pi \rho \, ,\\[0.4em]
\qquad \qquad {\rm div} k - d {\rm K}  = \,  8 \pi J \, .
\end{cases}
\end{equation}
Here \(\Lambda \in \mathbb{R}\) denotes the cosmological constant, \(\RRR\) is the scalar curvature of \(g\), and \({\rm K}\) is the trace of \(k\) with respect to \(g\). Interpreting \(k\) as the second fundamental form of the initial hypersurface \(M\) embedded in its maximal Cauchy development, the quantity \({\rm K}\) is then the corresponding mean curvature. Furthermore, if \(T\) denotes the \((2,0)\) stress–energy tensor appearing on the right-hand side of the Einstein Field Equations, the energy and momentum densities measured by an observer normal to \(M\) are given by
\[
\rho = T_{\nu\nu}, 
\qquad 
J_i = T_{\nu i}, \quad i \in \{1,2,3\},
\]
where \(\nu\) is the future-directed unit normal vector field to \(M\) within its Cauchy development. The quantities \(\rho\) and \(J\) represent, respectively, the local energy density and the local momentum density. In this formulation, the Dominant Energy Condition is equivalent to the requirement
\begin{equation*}
\rho \,\geq\, |J| \, .
\end{equation*}
In what follows, we restrict attention to the class of {\em time-symmetric initial data}, which constitute the fundamental model for the analysis of mass-type invariants. For such data, the second fundamental form \(k\) vanishes identically, and hence the constraint equations, coupled with the Dominant Energy Condition, reduce to
\begin{equation}
\label{eq:DEC}
\RRR - 2 \Lambda \,=\,  16 \pi \rho \,\geq \,0 \, ,
\end{equation}
This, in turn, justifies our standard geometric assumption that the scalar curvature is bounded below by the constant \(2 \Lambda\).

Despite the extensive body of work developed by leading physicists and mathematicians on the notion of mass in space-times with vanishing or negative cosmological constant, comparatively little is known about the corresponding concept in space-times with positive cosmological constant -- precisely the regime that, according to astronomical observations and measurements~\cite{Carrol,Carroll_Press_Turner,Riess1998,Perl99,Wein01}, is expected to describe the actual large-scale structure of our universe.
In this latter setting, not only are general theorems lacking, but even the formulation of an appropriate and robust definition of mass remains a {\em vexata quaestio}. The present work is devoted to the analysis of this challenging and  fascinating problem.

This difficulty is not merely technical. In the relativistic literature on asymptotics with a positive cosmological constant, it is stressed that the standard framework underlying ADM/Bondi energy, positive energy theorems, and radiation theory at null infinity does not extend from $\Lambda=0$ to $\Lambda>0$ in a straightforward physically useful way; see, for instance,~\cite{ABK15,Ashtekar17,Penrose11,SzabadosTod19}. Partial positive energy statements in asymptotically de Sitter settings are available under specific choices of slicing and asymptotic structure~\cite{AvaLinPiu,KastorTraschen02,LuoXieZhang07}, but they do not provide a canonical intrinsic mass for compact time-symmetric data with a cosmological horizon. Hence, in this latter 
regime one has to reconsider the very meaning of mass and energy.

To put our results in a perspective, let us recall that the cornerstone result and guiding principle in this area of investigation is the renowned Positive Mass Theorem, originally proven by Schoen and Yau for the ADM mass in the case of time-symmetric initial data in asymptotically flat spacetimes with a zero cosmological constant.

\begin{theorem}[Positive Mass Theorem]
\label{pmt}
$(M,g)$ 
be a complete, asymptotically flat, three-dimensional Riemannian manifold, with nonnegative scalar curvature $\RRR_g \geq 0$. Then the total ADM-mass of the manifold is nonnegative
\[
m_{ADM}(M,g)
\geq 0\,. 
\]
Moreover, it vanishes if and only if $(M,g)$ is isometric to the flat Euclidean three-space $(\mathbb{R}^3, g_{\mathbb{R}^3})$.
\end{theorem}

Over the years, several different proofs of this theorem have been provided, starting from the original proof by Schoen and Yau \cite{SchoenYau1,SchoenYau2}, based on minimal surface techniques. A few years later, a conceptually different argument was discovered by Witten \cite{Witten}, using spin geometry and harmonic spinors. Subsequent geometric approaches deepened the understanding of the ADM mass.  
Huisken and Ilmanen \cite{HI} established monotonicity of the Hawking mass along the weak inverse mean curvature flow, yielding both the Positive Mass Theorem and the Riemannian Penrose Inequality for a single black hole, while Bray \cite{bray1} proved the general case via a conformal flow of metrics.
More recent alternative proofs of the Positive Mass Theorem exploit monotonicity of the ADM mass along the Ricci flow \cite{LiRicci}, Bochner-type identities for harmonic and spacetime harmonic functions \cite{Bra_Kaz_Khu_Ste,Hir_Kaz_Khu},
and monotonicity formulas based on linear and nonlinear potential theory~\cite{AMO,AMMO}.
This potential-theoretic approach, together with the one developed by Huisken and Ilmanen through the weak IMCF, will be the most relevant for the subsequent discussion. 

In stating the Positive Mass Theorem, we have intentionally been somewhat imprecise regarding the assumption of asymptotic flatness of the manifold. The optimal conditions under which the theorem can be proven are the same as those that allow for a well-defined ADM-mass~\cite{BAR86, CHR86}, namely
\[
g_{ij} (x) = \delta_{ij} + \mathcal{O}(|x|^{-\tau}) \qquad \hbox{and} \qquad \pa_kg_{ij} (x) =  \mathcal{O}(|x|^{-\tau -1}) \, , \qquad \hbox{with $\tau > 1/2$} \, ,
\]
in a given coordinate chart at infinity $(x^1,x^2, x^3)$. Together with $\RRR_g\in L^1$ at infinity, these decay assumptions ensure that the ADM-mass
is well defined as 
\[
m_{ADM}(M,g) 
\,= \, \lim_{r\to +\infty}\, \frac{1}{16\pi} \int\limits_{\{\vert x\vert=r\}} 
 \sum_{i,j=1}^3 \left( 
\frac{\partial g_{ij}}{\partial x^j}  
- \frac{\partial g_{jj}}{\partial x^i} \!
\right) \!
\frac{x^i}{|x|} \,\, d\sigma \,.
\]
It follows from this definition that, if the decay rate of the metric coefficients is large enough, then the ADM-mass is vanishing and one can invoke the Positive Mass Theorem to conclude that the manifold is isometric to the Euclidean three-space. This fact can be phrased, saying that 

\medskip

\noindent {\em A  strongly asymptotically flat Riemannian three-manifold with nonnegative scalar curvature is necessarily flat.}

\medskip

This fundamental result was subsequently extended to the case of negative cosmological constants through successive efforts by various authors, increasingly clarifying the appropriate concept of mass for this setting~\cite{WangAH,Chr_Nag_2002, Chr_Her_2003,Kro_Oro_Pin_2025}. It is noteworthy that the first significant result in this context was Min-Oo's proof of the rigidity statement, made before any proper concept of mass had been identified or defined. Specifically, the theorem proven by Min-Oo in~\cite{MINOO_1} states that 

\begin{theorem}[Scalar Curvature Rigidity for AH Manifolds]
\label{thm:AH_Min-Oo}
Let $(M, g)$ be a complete three-dimensional Riemannian manifold, with $\RRR_g \geq - 6$, as it follows by setting $\Lambda = - 3$ in~\eqref{eq:DEC}. If $(M, g)$ is strongly asymptotically hyperbolic in the sense of~\cite{MINOO_1}, then it is isometric to hyperbolic space.
\end{theorem}

The geometric intuition behind this statement is quite natural: if one postulates the validity of a positive mass theorem and assumes a strong asymptotic decay of the metric toward the model one  -- so that any reasonable asymptotic geometric invariant, including the yet undefined notion of mass, is forced to vanish -- then it would be natural to conclude that the given manifold is isometric to the model space. 
In this sense, Theorem~\ref{thm:AH_Min-Oo} should be regarded as a pre-Positive Mass Theorem, as it represents a preliminary condition for the validity of a full Positive Mass statement. Following a similar rationale, Min-Oo proposed in~\cite{MINOO2} his famous conjecture for initial data with a positive cosmological constant, aiming to characterize the de Sitter solution as the only one with zero mass. The precise statement of the conjecture is as follows:

\begin{conjecture}[Min--Oo conjecture]
Suppose that $g$ is a smooth Riemannian metric on the closed hemisphere $\mathbb{S}^3_+$ with the following properties:
\begin{enumerate}
  \item[(i)] The scalar curvature of $g$ satisfies $\RRR_g\geq 6$, , as it follows by setting $\Lambda =  3$ in~\eqref{eq:DEC};
  \item[(ii)] The induced metric on the boundary agrees with the standard one, i.e., $g|_{\partial \mathbb{S}^3_+}=g_{\mathbb{S}^2}$;
  \item[(iii)] The boundary $\partial \mathbb{S}^3_+$ is totally geodesic with respect to $g$.
\end{enumerate}
Then $(\mathbb{S}^3_+,g)$ is isometric to the standard round hemisphere $(\mathbb{S}^3_+,g_{\mathbb{S}^3})$.
\end{conjecture}

As in the previous case, the conditions imposed on the boundary of the manifold -- sometimes referred to as the cosmological horizon, in the framework of General Relativity -- are intended to ensure that the total mass vanishes, even in the absence of an explicit definition of mass. Over the past two decades, investigations surrounding Min-Oo conjecture have produced a bifurcated picture: a corpus of positive rigidity results under strengthened hypotheses, and a decisive set of counterexamples showing the conjecture fails in full generality. 

On the positive side, several authors proved rigidity statements when the scalar-curvature condition is supplemented by additional geometric hypotheses. Notably, Hang and Wang established rigidity theorems in settings where one replaces the scalar-curvature lower bound by suitable Ricci lower bounds~\cite{Han_Wan_2009} or restricts attention to 
the class of metrics which are conformal to the round metric ~\cite{Han_Wan_2006}.
Later,
Spiegel~\cite{Spi}
extended Hang-Wang's argument to the locally conformally flat setting
(see also~\cite{BCE18} 
for a further extension where the bound on the scalar curvature is replaced by a corresponding bound for a more general function of the eigenvalues of the Schouten tensor). 
In~\cite{Bre_Mar_2011}, a version of the Min-Oo conjecture is proved
for spherical caps with radius sufficienly smaller than $\pi/2$, under the assumption that the metric is sufficiently close to the round metric in the $\mathscr C^2$-norm.
Another related result is the one obtained by Eichmair~\cite{EIC09}, where he proves a sharp isoperimetric characterization of the round three-sphere and derives as a corollary a Min-Oo-type rigidity statement for the upper hemisphere. More precisely, he shows that a closed, Ricci-positive 3-manifold with scalar curvature 
$\RRR_g \geq 6$ whose isoperimetric surfaces attain the maximal possible area $4\pi$ must be isometric to the standard $\mathbb{S}^3$. 
We return to this result with Theorem~\ref{thm_eich} below.

The general conjecture, however, does not hold under the mere natural assumptions under which it was originally formulated: Brendle, Marques and Neves~\cite{BMN11} constructed smooth metrics on the upper hemisphere that coincide with the standard round metric in a tubular neighborhood of the equatorial boundary, having scalar curvature $\RRR_g \geq 2 \Lambda >0$ everywhere and strictly grater than $2\Lambda$ on an open domain (supported far away from the round tubular neighborhood of the equator). 
The particular features of these  counterexamples not only disprove the conjecture but also invalidate the heuristic behind it: no pre-positive mass theorem in the style of Theorem~\ref{thm:AH_Min-Oo} can be obtained by imposing any order of contact with the model solution -- here, the hemisphere or space-like de Sitter metric -- in a neighborhood of the cosmological horizon.
In physical terms, one may assert that whatever the mass associated with these initial data may be, it cannot be measured by an observer on the cosmological horizon, even under the idealized assumption that this observer possesses complete knowledge of the surroundings.
A complementary obstruction appears in the de Sitter--Schwarzschild setting: the deformations constructed by Cruz, de Lima, and Montenegro~\cite{CruzLimaMontenegro18} show that naive positive-$\Lambda$ analogues of the Penrose inequality and its rigidity cannot be transplanted without modifying the underlying notion of mass.

\subsection{Hawking Mass and Inverse Mean Curvature Flow: failure of a natural strategy.}

In view of the rather unsatisfactory status of the art that we have outlined so far, the need to develop new heuristics and novel perspectives for addressing the Positive Mass problem in the presence of a positive cosmological constant $\Lambda > 0$ becomes apparent.  
Besides the interesting, though conceptually rather distant approach recently proposed in~\cite{AvaLinPiu}, a particularly natural strategy is to base the analysis on an appropriate notion of quasi-local mass and to investigate its large-scale behavior, in analogy with the Geroch monotonicity of the Hawking mass along the inverse mean curvature flow (IMCF, for short). Initiating the flow from a point in the manifold and making 
Geroch’s ~\cite{Ger} and Jang-Wald’s~\cite{JW77}
smooth computations
mathematically rigorous through the introduction of the weak IMCF, Huisken and Ilmanen, in their celebrated work~\cite{HI}, obtained an alternative proof of Theorem~\ref{pmt} in the classical case $\Lambda = 0$. 
Guided by these considerations, one would like to use the same pattern to address the positive mass problem in the presence of a positive cosmological constant.
In fact, one readily sees that the concept of Hawking mass can be straightforwardly extended to the setting of a nonvanishing cosmological constant (see, for example,~\cite[equation~(15)]{Gib_1999}) by defining
\begin{equation}
\label{mahw}
\mathfrak{m}_{\rm Haw} (\Sigma) \, = \, \sqrt{\frac{|\Sigma|}{16 \pi}}  \, \left\{ 1 - \frac{\Lambda}{12 \pi} |\Sigma| - \frac{1}{16 \pi} \int_\Sigma \HHH^2 \, d \sigma \right\},
\end{equation}
where $\Sigma \hookrightarrow (M,g)$ is a closed, smooth surface embedded in the initial data set $(M,g)$ with $\RRR \geq 2 \Lambda$ and $\HHH$ is its mean curvature. It is also easy to verify that if $\{\Sigma_t\}_{t\in I}$ is a family of closed, connected hypersurfaces evolving by inverse mean curvature flow, then, wherever all quantities are well defined, the map
\[
t \longmapsto \mathfrak{m}_{\rm Haw}(\Sigma_t)
\]
is non-decreasing. In fact, in this setting Geroch’s smooth computation yields
\begin{equation}
\label{eq:geroch}
\frac{d}{dt}
\mathfrak{m}_{\rm Haw}(\Sigma_t)
= \frac{1}{8 \pi}
\sqrt{\frac{|\Sigma_t|}{16 \pi}}\!
\left[\!
\left(4\pi - \int_{\Sigma_t} \frac{\RRR^{\Sigma}}{2} \, d\sigma \right)
+ \int_{\Sigma_t}
\left|\frac{\nabla^{\Sigma}{\rm H}}{{\rm H}}\right|^2
+ \frac{|\mathring{\rm h}|^2}{2
}
+ \frac{\RRR - 2\Lambda}{2}
\, d\sigma
\right],
\end{equation}
where $\RRR^\Sigma$ denotes the scalar curvature of the metric induced on the evolving hypersurfaces, $\mathring{\rm h}$ is their trace-free second fundamental form, and $\nabla^{\Sigma}$ is the Levi-Civita connection associated with the induced metric.

We now ask whether this strategy can be implemented on a compact initial data set \((M,g)\) with connected minimal boundary and scalar curvature \(\RRR \ge 2\Lambda > 0\), in line with condition~\eqref{eq:DEC}. Assume the above monotonicity formula extend rigorously to a weak setting that allows flow singularities, in the spirit of Huisken-Ilmanen~\cite{HI}. Further assume there exists a weak IMCF starting from a point of the initial data set -- more precisely, an ancient solution whose level sets collapse to a chosen pole as \(t \to -\infty\) -- and that this flow expires by landing on the cosmological horizon at a finite time \(T^*\). In this thought experiment, the Hawking mass at time \(T^*\), 
can be viewed as the total mass of the manifold. Monotonicity, together with \(\mathfrak{m}_{\rm Haw}(\Sigma_t) \to 0\) as \(t \to -\infty\), then strongly suggests a Positive Mass–type statement for this total mass. This appealing line of reasoning is nevertheless ruled out by the counterexamples of Brendle, Marques, and Neves. Indeed, once the flow enters a region where \(\RRR > 2\Lambda\), the last term in~\eqref{eq:geroch} forces the derivative of the Hawking mass to be strictly positive. Thus, the Hawking mass, which vanishes at the pole and is nondecreasing along the flow, becomes strictly positive and remains so up to time \(T^*\), when the flow expires on the cosmological horizon. In particular,
\[
0 \, <  \, \mathfrak{m}_{\rm Haw} (\Sigma_{T^*}) \, = \, \mathfrak{m}_{\rm Haw}(\pa M) \, .
\]
However, in these counterexample-manifolds a whole neighborhood of the boundary is isometric to a tubular neighborhood of the equator in the standard hemisphere of radius \(R_\Lambda = \sqrt{3/\Lambda}\) (equivalently, with constant sectional curvature \(\Lambda/3\)). A direct computation shows that the Hawking mass of this standard equator vanishes, since it is a minimal surface of area \(4\pi R_\Lambda^2 = 12\pi/\Lambda\). Consequently, any attempt to obtain a Positive Mass-type result via such an inverse mean curvature flow faces fundamental obstructions.

Upon closer examination, it becomes clear that the underlying cause of this speculative contradiction lies in the initial assumption that a flow of the specified type actually exists. Recent progress, including the important work of Xu~\cite{Xu}, shows that the Dirichlet problem is not the most appropriate setting in which to study the weak IMCF on a compact manifold with boundary. Instead, the natural condition appears to be that of the so‑called outer obstacle condition, in which the level sets of the arrival time function of the IMCF, namely, the weak solution of
\begin{equation}
{\rm div} \left( \frac{\nabla w}{|\nabla w|}\right) = |\nabla w| \, ,
\end{equation}
adheres tangentially to the boundary upon reaching it. Informally, this condition may be interpreted as the alignment of the unit normal to the level sets with the outward unit normal to the boundary, as the point approaches the boundary:
\[
\frac{\nabla w}{|\nabla w|}(y) \,\longrightarrow\, \nu(y_0) \qquad \text{as } y \to y_0 \in \partial M \,,
\]
where $\nu$ denotes the outward unit normal to $\partial M$ (we refer the reader to~\cite[Summary of Definitions/Theorem 1.5]{Xu} for the precise definition).
The pole-normalized version of this outer-obstacle problem, which will be used in the definition of the $1$-harmonic Mass, is stated and proved in Appendix~\ref{app:pole-imcf}; see in particular Definition~\ref{def:pole-normalized-solution} and Theorem~\ref{thm:pole-existence}.
In general, however, the Hawking mass is not guaranteed to be nondecreasing along such a flow. A concrete counterexample is provided by the region inside a Clifford torus in the round sphere of radius $R_\Lambda$. This manifold has constant scalar curvature $2\Lambda>0$, so it satisfies~\eqref{eq:DEC}, and its boundary is the minimal Clifford torus $\mathbb{T}_{\Lambda}$ 
with area $6\pi^2/\Lambda >12\pi/\Lambda$. 
A straightforward calculation then yields that the Hawking mass of this boundary is given by
\[
\mathfrak m_{\rm Haw}(\mathbb{T}_{\Lambda})
= \sqrt{\frac{3\pi}{8 \Lambda}}\left(1-\frac\pi2\right) \,< 0.
\]
Now initiate the IMCF with outer obstacle $\mathbb{T}_\Lambda$ from a point in this manifold (whose existence and uniqueness are provided by Theorem~\ref{thm:pole-existence}). As already mentioned, we have $\mathfrak m_{\rm Haw}(\pa \Omega_t) \to 0$ as $t\to -\infty$, where $\Omega_t =\{w <t\}$ denotes the sublevel set.
In the context of IMCF with an outer obstacle, it is more convenient to think in terms of the perimeter of the sublevels rather than the area of the level sets $\Sigma_t = \{ w=t\}$, so as to be able to extend our considerations even beyond the first time the flow comes into contact with the boundary. As time increases, the Hawking mass of $\pa \Omega_t$ initially remains nonnegative by monotonicity, up to reaching a certain time $T_0$, smaller than the first contact time. Now, following~\cite{Xu}, one can check that the IMCF with outer obstacle sweeps out the whole domain $M\setminus \Omega_{T_0}$ in finite time, say $T^*-T_0$. Moreover, for large times, the perimeter of the sublevel sets $\Omega_t =\{w <t\}$ approaches the full area of the boundary. 
In our case, this implies that for times $t$ sufficiently close to $T^*$, the perimeter of $\Omega_t$ becomes strictly greater than $12\pi/\Lambda$. As the contribution of the mean curvature term is always nonpositive, it follows that for such times the Hawking mass becomes strictly negative
\[
\mathfrak m_{\rm Haw} (\partial \Omega_t) \leq \sqrt{\frac{{\rm Per}(\Omega_t)}{16 \pi}}  \, \left\{ 1 - \frac{\Lambda}{12 \pi} {\rm Per}(\Omega_t) \right\} 
< 0.
\]
Thus, in this example $t \mapsto \mathfrak m_{\rm Haw}(\pa \Omega_t)$ starts near $0$ as $t\to -\infty$ and becomes strictly negative as $t\to T^*$, so it cannot be nondecreasing for all times.
This failure is consistent with analogous difficulties observed in the asymptotic and radiative theory for $\Lambda>0$: natural Bondi-type or quasi-local quantities may lose expected positivity or localization properties,
see~\cite{SzabadosTod15,Senovilla22,Saw18,ChruscielIfsits16}.

\subsection{A potential-theoretic approach toward the Positive Mass Theorem.}

In many recent contributions it has been demonstrated that the level-set flows associated with harmonic and $p$‑harmonic functions exhibit a highly nontrivial geometric structure, with far-reaching implications (see, e.g., ~\cite{Colding_Acta,Col_Min_3,Moser2007,Ago_Fog_Maz_1,Ago_Fog_Maz_2,Bra_Kaz_Khu_Ste, Mun_Wan_IMRN,Mun_Wan_AJM, AMO,AMMO}). 
In particular, these flows, when combined with appropriately designed monotonicity formulas, can serve as effective substitutes for geometric evolution equations that are equally natural but technically more delicate. This substitution can be exploited to establish comparison principles and derive geometric inequalities. Among the latter, the Positive Mass Theorem and the Riemannian Penrose Inequality are of primary relevance for the present study.

In the setting of vanishing cosmological constant and time-symmetric, asymptotically flat initial data, the potential-theoretic approach has, so far, been able to reproduce essentially all geometric results obtainable from the combined use of Huisken-Ilmanen’s weak IMCF and Geroch’s monotonicity formulas. Compared with this more geometric approach, the potential-theoretic method demands substantially more work in identifying and constructing suitable monotonicity formulas, since these necessarily involve terms 
that are not purely geometric (see for example equations~\eqref{mp_gen} and~\eqref{dermp} in Section~\ref{sec:MF}). Yet this extra analytical effort is offset by a decisive advantage on the side of the existence theory, which reduces to the classical question of constructing appropriate harmonic or $p$‑harmonic functions on exterior or punctured domains, with natural Dirichlet conditions imposed on any boundary components. 
The examples discussed in the previous subsection -- where the intrinsic difficulties of coupling IMCF with the Hawking mass came to light -- indicate, through the analysis of their pathologies, that shifting from geometric-flow to potential-theoretic existence theory could prove decisive in the framework of a positive cosmological constant and compact initial data with minimal boundary. 
The results that we are going to describe will confirm this expectation.

Starting from Green’s functions for the $p$-Laplacian, for exponents in the range $1<p<3$, subject to homogeneous Dirichlet boundary conditions, 
we will construct and analyze suitable monotone quantities whose constancy identifies the zero-mass model configuration, which, in the present setting, coincides with the hemisphere of radius $R_\Lambda = \sqrt{3/\Lambda}$, or, in physical terms, with the spatial component of the de Sitter spacetime with cosmological constant $\Lambda>0$, in its static slicing representation.

By evaluating these monotone quantities on the minimal boundary of the manifold -- which coincides with the level set of the \(p\)-Green’s function farthest from the pole -- we will identify, in Definition~\ref{def:polmass} the following mass-type invariant, that we name the Polarized \(p\)-harmonic Mass of $(M, g)$ with pole at $x$:
\begin{equation}
\label{eq:polarized_total_mass_Intro}
\mmp_\Lambda(M,g,x)\,= \int_{-\infty}^{+\infty} \!\!\!\!\!e^{\lambda_p(\tau)}  \, \big(4\pi - \Lambda \, {\rm Per} (\Omega_{\tau})\big)  \, d\tau 
\,+ \, \frac{R_\Lambda^{\frac{5-p}{p-1}}}{16\pi}\, 
K_p
\int_{\pa M} \!\!\!|\na u|^2 d\sigma \, ,
\end{equation}
where $u=u_x^{(p)}$ is the $p$-Green's function with pole at $x$ and null Dirichlet boundary conditions, 
\begin{equation*}
\begin{dcases}
\Delta_p u \,=\, - 4\pi \delta_x & \text{in } M,
\\
\quad u \,=\, 0 & \text{on } \partial M,
\end{dcases}
\end{equation*}
the set $\Omega_\tau = \Omega_{x, \tau}^{(p)}$ is given by 
\[
\Omega_\tau = \left\{ u > e^{-\tau/(p-1)} \right\} \, ,
\]
the function $\tau \mapsto \lambda_p(\tau)$ is a structural coefficient selected in Theorem~\ref{thm:global}, and the constant $K_p>0$ is given by  
\begin{equation}
\label{eq:keypi}
K_p \, = \, \frac{\Gamma(\frac12) \, \Gamma(c_p)}{\Gamma(a_p+\frac{3}{2})\Gamma(b_p+\frac{3}{2})} \, ,
\end{equation}
the coefficients $a_p, b_p$ and $c_p$ being the ones defined in~\eqref{p-parameters} and $\Gamma(\cdot)$ being the Euler's Gamma function. For this invariant, we establish 
the following positive mass statement:
\begin{theorem}[Positive Mass Theorem for the Polarized $p$-harmonic Mass]
\label{thm:pos_pass_pol_Intro}
Let $(M,g)$ be a compact three-dimensional Riemannian manifold with smooth connected boundary $\pa M$,
whose scalar curvature $\RRR$ satisfies 
\[
\RRR \, \geq \, 2 \Lambda \, , 
\]
for some $\Lambda>0$. Assume that $H_2(M, \pa M; \Z) = \{0\}$. 
Then, for every $1<p<3$ and every $x\in M\setminus\pa M$, the Polarized $p$-harmonic Mass with pole at $x$ satisfies
\[
\mmp_{\Lambda}(M,g,x)\geq 0\,.
\]
Furthermore, if $\mmp_{\Lambda}(M,g,x)=0$, then $(M,g)$ is isometric to a round hemisphere with constant sectional curvatures equal to $\Lambda/3$ and $x$ coincides with the north pole. 
\end{theorem}
The dependence of our newly introduced invariants on the choice of pole marks a clear difference from the classical scenarios with vanishing or negative cosmological constant. In those better studied regimes, the reference geometries used to define mass are noncompact, and this feature underlies their natural pole-independence. By contrast, for time-symmetric initial data with positive cosmological constant, the natural reference space is the standard round hemisphere. Here, compactness and the finite geodesic distance to the cosmological horizon make the issue of polarization unavoidable at first glance. This fact leads us to introduce the $p$-harmonic Total Mass, defined in Definition~\ref{def:pmass} as
\begin{equation*}
\mmp_\Lambda(M,g)\,=\,\inf_{x\in{M \setminus \pa M}} \mmp_\Lambda(M,g,x)\,.
\end{equation*}
By Theorem~\ref{thm:pos_pass_pol_Intro}, this mass functional is automatically nonnegative, but we will prove that it also satisfies an associated rigidity statement. In fact, we will show that the infimum in the above definition is actually a minimum. 

\begin{theorem}[Positive Mass Theorem for the $p$-harmonic Total  Mass]
\label{thm:pos_pass_total_Intro}
Let $(M,g)$ be a compact three-dimensional Riemannian manifold with smooth connected minimal boundary $\pa M$,
whose scalar curvature $\RRR$ satisfies 
\[
\RRR \, \geq \, 2 \Lambda \, , 
\]
for some $\Lambda>0$. Assume that $H_2(M, \pa M; \Z) = \{0\}$. 
Then, for every $1<p<3$, the $p$-harmonic Total Mass of $(M,g)$ satisfies
\[
\mmp_{\Lambda} (M,g) \, \geq  \, 0\,.
\]
Furthermore, if $\mmp_{\Lambda}(M,g)=0$, then $(M,g)$ is isometric to a round hemisphere with constant sectional curvatures equal to $\Lambda/3$.
\end{theorem}
Theorem~\ref{thm:pos_pass_pol_Intro} and Theorem~\ref{thm:pos_pass_total_Intro} represent our main contributions to the study of the concept of mass in the presence of a positive cosmological constant. They are restated as Theorem~\ref{thm:pos_pass_pol} and Theorem~\ref{thm:pos_pass_total} later on for the reader's convenience.

\subsection{Formal limits as \texorpdfstring{$p \to 1^+$}{p→1}: conjectures and further perspectives.}
\label{sub:pto1}

We find it both appropriate and interesting to include some remarks on the connection between our monotone quantities and the more classical, albeit non-effective, Hawking mass. Our aim is to clarify this relationship here in a preliminary way and to pursue a more systematic treatment in a subsequent work.
To this end, recall that $p$-harmonic functions $u$ can be employed, according to the scheme introduced by Moser~\cite{Moser2007}, to approximate the weak IMCF, in the limit as $p \to 1^+$. This is achieved via the change of variable
\[
w = - (p-1) \log u\,,
\]
under which $w$ solves the so-called $p$-IMCF equation
\[
\operatorname{div} (|\nabla w|^{p-2} \nabla w) = |\nabla w|^p \, .
\]
In particular, as noted in Appendix~\ref{app:pharmonic}, a $p$-Green’s function with pole at $x \in M\setminus \partial M$ gives rise to a solution of the $p$-IMCF emanating from $x$:
\begin{equation*}
\begin{dcases}
\mathrm{div}\big(e^{-w}|\nabla w|^{p-2}\nabla w\big) \,=\, 4\pi (p-1)^{p-1} \delta_x & \text{in } M\,,
\\[0.3em]
\qquad \qquad \qquad \quad w(y)\longrightarrow +\infty & \text{as } y\to\partial M \, .
\end{dcases}
\end{equation*}
Applying the recent approximation theorem~\cite[Theorem 5.1]{Ben_Mar_Rig_Set_Xu} to an exhaustion of $M\setminus\{x\}$ by domains of the form $M \setminus B_{\rho}(x)$, with $\rho \to 0$, one can show that the 
solutions to the corresponding problem
subconverge uniformly on compact subsets of $M\setminus\{x\}$ to the IMCF with outer obstacle $w_1=w_x^{(1)}$ emanating from $x$.
We do not pursue this derivation here, as it will be addressed in future work. Instead, we provide a direct proof of existence in  Appendix~\ref{app:pole-imcf}. 
Indeed, Theorem~\ref{thm:pole-existence} produces a unique pole-normalized weak obstacle solution $w_1=w_x^{(1)}$ of 
\begin{equation}
\label{eq:obs-imcf}
\begin{dcases}
\mathrm{div}\left(e^{-w} \frac{\nabla w}{|\nabla w|}\right) \,=\, 4\pi  \delta_x & \text{in } M,
\\
\quad \frac{\nabla w}{|\nabla w|}(y)\longrightarrow \nu(y_0) & \text{as } y\to y_0 \in \partial M,
\end{dcases}
\end{equation}
in the sense of Definition~\ref{def:pole-normalized-solution}, as the locally uniform limit, as $\rho\to^+0$, of Xu's maximal solutions starting from the small balls $B_\rho(x)$ and normalized as in~\eqref{eq:def-w-rho}.
The normalization constant in front of the Dirac delta 
is reflected in the fact that, for every $t \in \R$,
\begin{equation*}
\int_{\Sigma_{x,t}^{(p)}} |\na w_p|^{p-1}\, d \sigma \, = \,
4\pi(p-1)^{p-1} e^t , \qquad \Sigma_{x,\tau}^{(p)} = \{ w_x^{(p)} = \tau \} \, ,
\end{equation*}
(see the proof of Theorem~\ref{thm:ssl}).
Indeed, in the limit as $p\to1^+$, this identity 
induces the following normalization for the area growth law of the IMCF
\begin{equation*}
|\Sigma_{x,\tau}^{(1)}| \, = \, 4 \pi e^\tau, \qquad \Sigma_{x,\tau}^{(1)} = \{ w_x^{(1)} = \tau \} ,
\end{equation*}
for all $t$ smaller than the first contact time~\eqref{eq:fcc} (more precisely, for a.e. $\tau<T_*(x)$ in the weak setting; see \eqref{eq:area-growth-pole} and \eqref{eq:smooth-area-growth-pole}). The $p$-IMCF $w_p = w_x^{(p)}$ emanating from $x$ is also very useful for expressing the monotone quantities that arise in the proof of Theorem~\ref{thm:pos_pass_pol_Intro} as
\begin{equation*}
\mmp_\Lambda (x,t) \, = \int_{-\infty}^t \!\!\!\! e^{\lambda_p(\tau)} \bigl(4\pi - \Lambda|\Sigma_{x,\tau}^{(p)}|\bigr)  \, d\tau
\, - \,\,  e^{\lambda_p(t)} \!\int_{\Sigma_{x,t}^{(p)}} \!\! |\nabla w_p| \,  \bigl(  \HHH - \mu_p(t)|\nabla w_p|  \bigr) \, d\sigma \,, \qquad t \in \mathbb{R}\, , 
\end{equation*}
where $\Sigma_{x,\tau}^{(p)} = \{ w_x^{(p)}=\tau \}$, $\HHH$ denotes its mean curvature, and $\lambda_p$ and $\mu_p$ are the structural coefficients selected in Theorem~\ref{thm:global}.
Although these structural coefficients are not as explicit as in the $\Lambda = 0$ case (see formula~\eqref{eq:lomo}), a finer analysis, carried out in Section~\ref{sec:further}, reveals very effective insights about our monotone quantities. Specifically, their asymptotic behavior as $t \to \pm \infty$, summarized in Corollary~\ref{cor:behaviour_with_t}, delivers two key outcomes. On the one hand, as $t \to -\infty$, it gives a Small Sphere Limit result (Theorem~\ref{thm:ssl}) for $1<p<2$,
\[
\lim_{t\to-\infty}\frac{\mmp_\Lambda(x,t)}{|\Omega_{x,t}^{(p)}|}\,=\,
\frac{\RRR(x)-2\Lambda}{16 \pi}\,, \qquad \hbox{where} \,\,\, \Omega_{x,t}^{(p)} = \{ w_x^{(p)} \!< \, t \} \, ,
\]
which substantiates their interpretation as quasi-local mass functionals. On the other hand, as $t \to +\infty$, it leads to a clean expression for the Polarized $p$-harmonic Mass (see formula~\eqref{eq:polarized_total_mass_Intro}),
which underpins Definition~\ref{def:polmass} and provides a concrete global invariant distilled from the underlying monotonicity properties.

It is also interesting to study the behavior of the structural coefficients $\lambda_p$ and $\mu_p$ in the limit as $p \to 1^+$. Indeed, in  Lemmas~\ref{le:eallalambda} and~\ref{mupi}, we will prove that
\[
e^{\lambda_p (t)}\,\longrightarrow \,
\begin{dcases}
\frac{e^{t/2}}{16 \pi} & \hbox{if $t< T_\Lambda$}
\\
\\
0 & \hbox{if $t > T_\Lambda$}
\end{dcases}
\qquad \hbox{and}\qquad 
\mu_p(t) \, e^{\lambda_p(t)}\,\longrightarrow \,
\begin{dcases}
\frac{e^{t/2}}{32\pi} & \hbox{if $t<  T_\Lambda$}
\\
\\
0 & \hbox{if $t >  T_\Lambda$}
\end{dcases}
\]
as $p\to 1^+$, where we set 
\[
T_\Lambda = 2 \log (R_\Lambda) = \log (3/\Lambda) \, .
\]
Note that $T_\Lambda$ coincides both with the maximal existence time and with the first contact time of the model IMCF, namely the IMCF starting from the north pole $N$ of the standard sphere of radius $R_\Lambda$ and having the equator as an outer obstacle. Indeed, the model solution is explicitly given by $w_N= 2 \log r$ in this framework. The behavior of the structural coefficients, as $p \to 1^+$, makes it possible to identify a formal limit of the monotone quantities $\mmp_\Lambda (x,t)$ toward the Hawking mass.
In fact, assuming that the $p$‑IMCF $w_x^{(p)}$ converge to an IMCF with obstacle $w_x^{(1)}$ in a way that is strong enough to justify the convergence of the areas of the level sets and of suitable integral norms of the gradients, we would obtain the following limit for small enough $t < T_\Lambda$:
\begin{align*}
&\int_{-\infty}^t
\left(\frac{e^{\tau/2}}{16\pi}\right)
\left(4\pi-\Lambda|\Sigma_{x,\tau}^{(1)}|\right)
d\tau
-\bigg(\frac{e^{t/2}}{16\pi}\bigg)
\int_{\Sigma_{x,t}^{(1)}} \! |\nabla w_1| \!\left(  \HHH - \frac{1}{2}|\nabla w_1|  \right) \, d\sigma  \\
=\,&\int_{-\infty}^t
\left(\frac{e^{\tau/2}}{16\pi}\right)
\left(4\pi-\Lambda \,  4\pi e^\tau\right)
d\tau
-\bigg(\frac{e^{t/2}}{16\pi}\bigg)
\int_{\Sigma_{x,t}^{(1)}} \!  \frac{\HHH^2}{2}  \, d\sigma\\
=\,& \,\mathfrak{m}_{\rm Haw}\big(\Sigma_{x,t}^{(1)}\big)\, ,
\end{align*}
as $|\na w_1| =\HHH$ for the IMCF and, according to our normalization  \eqref{eq:smooth-area-growth-pole}, $|\Sigma_{x,t}^{(1)}| = 4 \pi e^t$, if $t$ is not too large. We conjecture that this formal limit can be rigorously justified for almost every time in the range $(-\infty, T_\Lambda)$ smaller than the first contact time 
\begin{equation}
\label{eq:fcc}
T_*(x) \, = \, \min_{\partial M} w_x^{(1)}
=\inf\, \Bigl\{t\in\mathbb R:\, \bigl|\partial\Omega^{(1)}_{x,t}\cap\partial M\bigr|\neq 0\Bigr\}
\end{equation}
of the weak IMCF with the obstacle $\pa M$ emanating from $x$. In this range, in fact, the standard area growth $|\Sigma_{x,t}^{(1)}| = 4 \pi e^t$ is satisfied (see \eqref{eq:area-growth-pole}--\eqref{eq:smooth-area-growth-pole} 
and also ~\cite[Remark 3.5-(iii)]{Xu})
and it is very plausible that one can adapt the theory of improved convergence, recently developed by Benatti, Pluda, and Pozzetta in~\cite[Theorem 1.2]{Ben_Plu_Poz}.

On one side, the formal convergence to the Hawking mass for fixed times clarifies the  geometric nature of our monotone quantities. On the other side, it is already clear that the limit of the Hawking mass as $t \to T^*(x)$, where
\begin{equation}
\label{eq:tistar}
T^*(x) \, =\, \max_{\pa M } \, w_x^{(1)}
\end{equation}
denotes the maximal existence time of the IMCF with obstacle emanating from $x$, does not yield truly satisfactory information. By contrast, a far more intriguing approach is to reverse the order of limits: first considering $\lim_{t\to +\infty}\mmp_\Lambda(x,t)$, and only then pass to the limit as $p \to 1^+$ in the Polarized $p$-harmonic Mass. According to formula~\eqref{eq:polarized_total_mass_Intro}, in the setting of manifolds with minimal boundary, this limit is given by
\[
\lim_{p \to 1^+} \mmp_\Lambda(M,g,x) \, = \, \lim_{p \to 1^+} \left(\int_{-\infty}^{+\infty} \!\!\!\!\!e^{\lambda_p(\tau)}  \, \big(4\pi - \Lambda \, {\rm Per} (\Omega_{x,\tau}^{(p)})\big)  \, d\tau 
\,
+ \, \frac{R_\Lambda^{\frac{5-p}{p-1}}}{16\pi}\, 
K_p
\, \int_{\pa M} \!\!\! |\na u_p|^2 d\sigma  \right)\,. 
\]
As already noted, a full analysis of this quantity will be deferred to future work; at present we limit ourselves to a few preliminary remarks. First, using expansion~\eqref{eq:behaviour_rp_to_one_phi} in Proposition~\ref{pro:behaviour_near_one}, we obtain
\[
K_p 
\, = \, (p-1) \, (1 + o(1)) \, , \qquad \hbox{as} \,\, p \to 1^+ \, ,
\]
where $K_p$ is the constant defined in~\eqref{eq:keypi}.
If we assume an upper bound on the boundary energy of the form
\begin{equation*}
\int_{\pa M} \!\!\! |\na u_p|^2 d\sigma \, \leq \,  B_p \int_{\SSS^2_\Lambda} \!\! \big|\na u_\Lambda^{(p)}\big|^2 d\sigma \, ,
\end{equation*}
where $u_\Lambda^{(p)}$ is the model solution with pole at the north pole and the constant $B_p>0$ satisfies $B_p=o(1/(p-1))$ as $p \to 1^+$, then, thanks to~\eqref{u_L_dot}, we deduce that the second integral can be estimated as
\[
\frac{R_\Lambda^{\frac{5-p}{p-1}}}{16\pi}\, 
K_p
\, \int_{\pa M} \!\!\! |\na u_p|^2 d\sigma \,  \leq \, K_p   B_p \,\frac{R_\Lambda^{\frac{5-p}{p-1}}}{16\pi}\, 
\, 4\pi R_\Lambda^2 \, R_{\Lambda}^{-\frac{4}{p-1}} \, = \, \frac{R_\Lambda}{4} \, K_p  B_p \,= \, o(1)\,, \qquad \hbox{as} \,\, p \to 1^+ \, .
\]
It follows that, within this framework, only the first integral -- the so-called bulk term -- contributes in the limit as $p \to 1^+$. Assuming the $L^1$-convergence of the perimeters of the sublevel sets of the $p$-IMCF to the perimeters of the sublevel sets of the IMCF with obstacle with the same pole, 
the bulk term should converge, as $p \to 1^+$, to the quantity 
\begin{equation}
\label{eq:1harmass}
\mathfrak{m}_\Lambda^{(1)} (M, g, x) \, = \, \int_{-\infty}^{T^+_x}
\!\!\!\left(\frac{e^{\tau/2}}{16\pi}\right) \big(4\pi - \Lambda \, {\rm Per} (\Omega_{x,\tau}^{(1)})\big)  \, d\tau  \, ,
\end{equation}
where $T^+_x = \min \{ T_\Lambda ,T^*(x)\}$. 
This quantity, that we name Polarized $1$-harmonic Mass as it has been obtained in the limit as $p\to1^+$ of the Polarized \(p\)-harmonic Masses under a set of reasonable simplifying assumptions, is interesting in its own right and allows us to introduce a second group of main results related to it. Notably, these  
include a Positive Mass Theorem and a Riemannian Penrose Inequality for the $1$-harmonic Mass.

\subsection{The \texorpdfstring{$1$}{1}-harmonic Mass}
\label{sub:1harm}

We first observe that, on a compact manifold with smooth connected boundary, the quantity~\eqref{eq:1harmass} is well defined for every $x \in M \setminus \partial M$. Indeed, its definition only requires the existence of an IMCF with outer obstacle $\partial M$ emanating from the pole $x$, subject to a suitable normalization, that is, a solution of~\eqref{eq:obs-imcf}; this is exactly the content of Theorem~\ref{thm:pole-existence}, together with the pole normalization in Definition~\ref{def:pole-normalized-solution}. 
Once such a flow is available, one can indeed compute the areas of its sublevel sets and the last contact time~\eqref{eq:tistar} of the flow with the obstacle, denoted by $T^*(x)$. This time, when compared with the standard existence time $T_\Lambda$, uniquely determines the upper limit of integration $T^+_x$.
From this viewpoint, the $1$-harmonic Mass can be regarded as an intrinsic substitute for an absent asymptotic charge, in line with the general observation that the de Sitter cosmological horizon obstructs genuinely global notions of energy and suggests quasi-local or intrinsically normalized alternatives~\cite{AvaLinPiu}.
We also note that the polarized $1$-harmonic Mass~\eqref{eq:1harmass}
vanishes in the correspondence of the de Sitter metric
\[
g_\Lambda\,=\,\frac{dr\otimes dr}{1-\frac{\Lambda}{3}r^2}+r^2g_{\mathbb{S}^2} \, , \qquad 0<r \leq R_\Lambda = \sqrt{3/\Lambda} \, ,
\]
when computed with respect to the north pole $N$. Indeed, in this case we have that
$T_\Lambda = T^*(N)$, the canonical IMCF is given by $w_N^{(1)} = 2 \log r$, and the area of its level sets is $4 \pi e^\tau$ for every $-\infty<\tau \leq T_\Lambda = 2 \log R_\Lambda$. A direct evaluation of the defining integral then yields
\[
\mathfrak{m}_\Lambda^{(1)} (\SSS^3_+, g_\Lambda
, N) = \frac{R_\Lambda}{2} \left( 1 - \frac{\Lambda}{3} (R_\Lambda)^2 \right) = 0 \, .
\]

In the next theorem, we prove that, for a suitable class of time-symmetric initial data, the polarized \(1\)-harmonic Mass is always nonnegative, and that its vanishing characterizes the de Sitter solution.
\begin{theorem}[Positive Mass Theorem for the Polarized $1$-harmonic Mass]
\label{thm_eich}
Let $(M,g)$ be a compact three-dimensional Riemannian manifold with smooth minimal boundary $\pa M$, whose scalar curvature satisfies
\[
\RRR \geq 2 \Lambda,
\]
for some $\Lambda>0$. Assume that $\pa M$ is simply connected and unstable and that there are no closed minimal surfaces in $M\setminus \pa M$. Then, for every $x \in M\setminus \pa M$, we have
\[
\mathfrak{m}_\Lambda^{(1)} (M, g, x) \, \geq \, 0,
\]
with equality if and only if $(M,g)$ is isometric to a round hemisphere of constant sectional curvatures equal to $\Lambda/3$ and $x$ coincides with the north pole. 
\end{theorem}
\begin{remark}
As it will be clear from Lemma~\ref{lem:topol} below, 
the absence of closed minimal surfaces in $M \setminus \pa M$, in the presence of a 
simply connected and unstable minimal boundary,
implies that the manifold is topologically simple and in fact diffeomorphic to $\mathbb{B}^3$. 
This property is important for the rigidity statement, but it is irrelevant for the mere inequality, whose proof does not rely on any topological assumption.
\end{remark}

\begin{remark}
We conjecture that following the ideas developed in Section~\ref{sec:blowup} a similar positive mass statement can be proven for the unpolarized version of the $1$-harmonic Mass, namely
\[
\mathfrak{m}_\Lambda^{(1)} (M, g) \, = \, \inf_{x\in M \setminus \pa M} \, \mathfrak{m}_\Lambda^{(1)} (M, g, x) \, .
\]
In fact, a closer inspection of formula~\eqref{eq:1harmass} suggests 
that when the pole approaches the boundary the value of the Polarized $1$-harmonic Mass would increase. 
\end{remark}

To introduce our last main result, we first note that the $1$-harmonic Mass defined in~\eqref{def:1harm} is compatible with the real parameter $m$ appearing in the two-parameter family of Schwarzschild-de Sitter metrics, whose spatial component we now describe.
These Riemannian metrics are given by
\begin{equation}
\label{eq:SDS}
g_{\Lambda, m}\,=\,\frac{dr\otimes dr}{1-\frac{\Lambda}{3}r^2 - \frac{2m}{r}}+r^2g_{\mathbb{S}^2} \, , \qquad R_\Lambda^- (m)\leq r \leq R_\Lambda^+(m) \, ,
\end{equation}
with $\Lambda > 0$ and $0 \leq m < 1/(3\sqrt{\Lambda})$.
Here $R_\Lambda^{\pm}(m)$ denote, for $m \neq 0$, the two positive roots of the equation $p_\Lambda(r) = m$, where $r \mapsto p_\Lambda(r)$ is the cubic polynomial
\[
p_\Lambda(r) \, = \, \frac{r}{2}\left(1-\frac{\Lambda}{3}r^2 \right)  .
\]
For future convenience, we also define $T_\Lambda^\pm(m) = 2 \log (R_\Lambda^\pm(m))$ and observe that these two numbers represent the endpoints of the standard existence interval of the model IMCF $w = 2 \log r$. In other words, this flow
originates from the black hole horizon $\{ r = R_\Lambda^-(m)\}$ at time $T_\Lambda^-(m)$ and reaches the cosmological horizon $\{ r = R_\Lambda^+(m)\}$ at time $T_\Lambda^+(m)$. Furthermore, note that for $\Lambda > 0$ the polynomial $r \mapsto p_\Lambda (r)$ is increasing on the interval $R_\Lambda^-(0) = 0 \leq r \leq  \sqrt{1/\Lambda}$ and decreasing on the interval $\sqrt{1/\Lambda} \leq r \leq \sqrt{3/\Lambda} = R_\Lambda^+(0)$. In particular, for every $\Lambda >0$ and every $0\leq m < 1/(3 \sqrt{\Lambda})$, we have
\[
0 \leq R_\Lambda^-(m) < \sqrt{{1}/{\Lambda}} < R_\Lambda^+(m) \leq R_\Lambda \, . 
\]
\def\Lam{1.0}    
\def\mm{0.22}     
\begin{figure}[t]
\centering
\begingroup
\providecommand{\Lam}{1.0}
\providecommand{\mm}{0.22}

\pgfmathsetmacro{\Rmax}{1/sqrt(\Lam)}
\pgfmathsetmacro{\Ymax}{1/(3*sqrt(\Lam))}
\pgfmathsetmacro{\Rzero}{sqrt(3/\Lam)}
\pgfmathsetmacro{\mu}{\mm*sqrt(\Lam)}
\pgfmathsetmacro{\theta}{acos(-3*\mu)}
\pgfmathsetmacro{\Rplus}{2*cos(\theta/3)/sqrt(\Lam)}
\pgfmathsetmacro{\Rminus}{2*cos(\theta/3 - 120)/sqrt(\Lam)}
\pgfmathsetmacro{\xmin}{-0.12*\Rzero}
\pgfmathsetmacro{\xmax}{1.08*\Rzero}
\pgfmathsetmacro{\ymin}{-0.18*\Ymax}
\pgfmathsetmacro{\ymax}{1.17*\Ymax}

\begin{tikzpicture}
\begin{axis}[width=0.88\linewidth,
height=0.50\linewidth,
axis lines=middle,
xmin=\xmin, xmax=\xmax,
ymin=\ymin, ymax=\ymax,
samples=250, domain=\xmin:\xmax,
clip=false,
xlabel={$r$},
ylabel={$p_{\Lambda}(r)$},
xtick=\empty,
ytick=\empty,
axis line style={thick},
xlabel style={at={(ticklabel* cs:1)},anchor=west},
ylabel style={at={(ticklabel* cs:1)},anchor=south},
]
\addplot[very thick] {x/2*(1-\Lam*x^2/3)};
\addplot[dashed, thick] coordinates {(\Rmax,0) (\Rmax,\Ymax)};
\addplot[dashed, thick] coordinates {(0,\Ymax) (\Rmax,\Ymax)};
\addplot[only marks, mark=*, mark size=1.9pt] coordinates {(\Rmax,\Ymax) (\Rmax,0) (0,\Ymax)};
\addplot[blue!70!black, dashed, thick] coordinates {(\Rminus,0) (\Rminus,\mm)};
\addplot[blue!70!black, dashed, thick] coordinates {(\Rplus,0) (\Rplus,\mm)};
\addplot[blue!70!black, dashed, thick] coordinates {(0,\mm) (\Rplus,\mm)};
\addplot[blue!70!black, only marks, mark=*, mark size=1.8pt]
coordinates {(\Rminus,0) (\Rplus,0) (\Rminus,\mm) (\Rplus,\mm) (0,\mm)};
\addplot[only marks, mark=*, mark size=1.8pt] coordinates {(\Rzero,0)};
\node[below left=1pt, fill=white, inner sep=1.2pt] at (axis cs:0,0) {$0$};
\node[below=2pt, fill=white, inner sep=1.2pt] at (axis cs:\Rmax,0) {$\sqrt{1/\Lambda}$};
\node[below=2pt, fill=white, inner sep=1.2pt] at (axis cs:\Rzero,0) {$\sqrt{3/\Lambda}$};
\node[left=2pt, fill=white, inner sep=1.2pt] at (axis cs:0,\Ymax) {$\frac{1}{3\sqrt{\Lambda}}$};
\node[blue!70!black, left=3pt, fill=white, inner sep=1.2pt] at (axis cs:0,\mm) {$m$};
\node[blue!70!black, below=4pt, fill=white, inner sep=1.2pt] at (axis cs:\Rminus,0) {$R_{\Lambda}^{-}(m)$};
\node[blue!70!black, below=4pt, fill=white, inner sep=1.2pt] at (axis cs:\Rplus,0) {$R_{\Lambda}^{+}(m)$};
\end{axis}
\end{tikzpicture}
\endgroup
\caption{Graph of the function $p_{\Lambda}(r)=\frac{r}{2}\left(1-\frac{\Lambda}{3}r^2\right)$ over the  interval $[0,\sqrt{3/\Lambda}]$. The black dashed lines highlight the maximum point and the maximal value of the function; the blue dashed lines represent the level $m$ and its two positive preimages $R_{\Lambda}^{-}(m)$ and $R_{\Lambda}^{+}(m)$.}
\label{fig:p-lambda}
\end{figure}

\noindent
Having introduced the appropriate notation, we can now confirm the claimed consistency check for the $1$‑harmonic mass of the Schwarzschild--de Sitter solutions. For this purpose, it is convenient to attach a spherical cap (hemisphere) of radius $R_\Lambda^{-}(m)$ to the black‑hole horizon of $g_{\Lambda,m}$ -- that is, to the boundary component corresponding to $r = R_\Lambda^{-}(m)$ -- thereby producing a manifold with connected boundary and a distinguished north pole. In this configuration, one verifies that the IMCF starting from the north pole $N$ of the spherical cap and terminating at the cosmological horizon $\{r = R_\Lambda^{+}(m)\}$ is still described by 
\[
w_N^{(1)} = 2 \log r,
\]
as in the case $m = 0$. The only modification concerns the maximal existence time of the flow, which is now given by
\[
T^*(N) = T_\Lambda^+(m) = 2 \log R_\Lambda^+(m).
\]
Moreover, for $-\infty < \tau \leq T_\Lambda^+(m)$, the areas of the corresponding level sets remain equal to $4\pi e^\tau$. Consequently, performing in the present setting the very same integration that we performed in the massless case yields that the value of the $1$-harmonic Mass is now given by
\[
\frac{R^+_\Lambda(m)}{2} \left( 1 - \frac{\Lambda}{3} \bigl(R^+_\Lambda(m)\bigr)^2 \right) = m \, ,
\]
as wished. Beyond this consistency verification, we observe that the notion of $1$‑harmonic mass admits a coherent and natural extension to the setting of three‑dimensional Riemannian bands $(M,g)$ with scalar curvature satisfying $\mathrm{R}_g \geq 2\Lambda > 0$ and boundary
\[
\partial M = \partial M^- \sqcup \partial M^+ \, .
\]
In this more general framework, we will demonstrate that the $1$‑harmonic mass satisfies a Riemannian Penrose‑type inequality, in agreement with the conjectural inequalities proposed by Bray and Chruściel in~\cite{Bra_Chr_2004} (see in particular their Penrose conjecture, formula (4.7)).
The primary formal distinction between
the definition given in~\eqref{eq:1harmass},
for the case of connected boundary, 
and the one we are about to introduce lies in the specification of the standard existence interval and the consequent normalization of the initial mass content. To clarify these concepts, we assume that both boundary components are smooth closed minimal surfaces 
and we suppose without loss of generality that
\[
|\partial M^-| \leq |\partial M^+| \, ,
\]
thus interpreting the boundary component $\partial M^-$ 
as a black-hole–type horizon and the boundary component $\partial M^+$ as a cosmological horizon. Under this hypothesis, we define the real numbers $R^-\!= R^-(M,g)$ and $T^-\!=T^-(M, g)$ as
\begin{equation*}
R^- \! =  \sqrt{\frac{|\pa M^-|}{4 \pi}}    \qquad \hbox{and} \qquad T^- \! = \, 2 \log R^- \! = \, \log  \frac{|\pa M^-|}{4 \pi} \, .
\end{equation*}
These quantities will be instrumental in conveniently renormalizing the initial time of the IMCF that originates from $\partial M^-$ and for which $\partial M^+$ acts as an outer obstacle. Let now $w$ denote the unique solution of the problem
\begin{equation}
\label{eq:ivp-obs-imcf}
\begin{dcases}
\mathrm{div}\left(e^{-w} \frac{\nabla w}{|\nabla w|}\right) \,=\, 0 & \text{in } M\,,
\\
\,\,\, \qquad \qquad \qquad w \,= \, T^- & \text{on } \pa M^-\,,\\
\quad \,\,\,\, \frac{\nabla w}{|\nabla w|}(y)\longrightarrow \nu(y_0) & \text{as } y\to y_0 \in \partial M^+\, ,
\end{dcases}
\end{equation}
in the weak sense 
described by Xu in~\cite[Summary of Definitions/Theorem 1.5]{Xu}, where $\nu(y_0)$ is the outward unit normal at $y_0\in \partial M^+$. 
As in the polar case, we set
\begin{equation}
T^* \, =\, \max_{\pa M^+ } \, w \qquad \hbox{and} \qquad T_* \, =\, \min_{\pa M^+} \,  w \, ,
\end{equation}
and, correspondingly, $R^* = e^{T^*/2}$ and $R_* = e^{T_*/2}$. 
At this stage, mimiking the definition of $T^+_x$ in the connected boundary case, we define the time $T^+ = T^+(M, g)$ and the associated radius $R^+ = R^+(M,g)$ by
\begin{equation*}
T^+ = \min \, \{ T^+_\Lambda(p_\Lambda(R^-)) , T^* \} \qquad \hbox{and} \qquad R^+ = \min \, \{R^+_\Lambda(p_\Lambda(R^-)), R^*\} \, ,
\end{equation*}
and we note that they satisfy the relation $T^+ = 2 \log R^+$. 
\begin{remark}
\label{rmk:stabledef}
We first observe that, in order for $T^+$ and $R^+$ to be well defined, the argument of $p_\Lambda(\cdot)$ must lie in the interval $[0,\sqrt{3/\Lambda}]$. Furthermore, for the interpretation of $\pa M^-$ as a black-hole horizon to be genuinely compelling from a physical and geometric standpoint, it is in fact preferable that 
this argument belongs
to the more restrictive interval $[0,\sqrt{1/\Lambda}]$. 
We note that this would be the case, under the hypotheses that will appear in Theorem~\ref{thm:RPI-1harm}. Indeed, strict stability of $\pa M^-$ implies $|\pa M^-| < 4\pi/\Lambda$, or equivalently 
\begin{equation*}
0 < R^- < \sqrt{{1}/{\Lambda}} 
\qquad \text{and} \qquad 
-\infty < T^- < - \log \Lambda \, .
\end{equation*}
Therefore $p_\Lambda(R^-) < 1/(3\sqrt{\Lambda})$ and the outer root $R_\Lambda^+(p_\Lambda(R^-))$ is well defined and strictly 
larger
than $\sqrt{1/\Lambda}$.
\end{remark}

Having introduced all the notation needed, we can proceed in analogy with~\eqref{eq:1harmass} 
and  define the $1$-harmonic Mass of the Riemannian band $(M,g)$ as
\begin{equation}
\label{def:1harm}
\mathfrak{m}_\Lambda^{(1)} (M,g) \, := \,  \mathfrak{m}_{\rm Haw} (\partial M^-) \, + \, \int_{T^-}^{T^+} \left(\frac{e^{\tau/2}}{16\pi}\right)  \big(4\pi - \Lambda \, {\rm Per} (\Omega_{\tau})\big)  \, d\tau ,
\end{equation}
where $\Omega_\tau = \{ w< \tau \}$. We observe that this invariant can be realized as the limit, as $p\to1$, of the $p$-harmonic masses associated with the solutions of the Dirichlet problems
\begin{equation*}
\begin{dcases}
\Delta_p u \,=\, 0 & \text{in } M,\\
u \,=\, e^{-T^-/(p-1)} & \text{on } \partial M^-,\\
u \,=\, 0 & \text{on } \partial M^+,
\end{dcases}
\end{equation*}
for $1<p<3$, via a formal limiting procedure analogous to that described in Subsection~\ref{sub:1harm}. Furthermore, one readily verifies that the 
$1$-harmonic Mass 
of the model solutions coincides with the parameter $m$ appearing in formula~\eqref{eq:SDS}. Finally, when the boundary portion $\pa M^-$ collapses to a point $x$, both its area and its Hawking mass converge to zero, whereas $T^-$ and $T_\Lambda^+(p_\Lambda(R^-))$ converge to $-\infty$ and $T_\Lambda = \log (3/\Lambda)$, respectively. This shows that Definition~\eqref{def:1harm} is consistent with the one 
of the Polarized $1$-harmonic Mass
introduced in~\eqref{eq:1harmass}. With this new notion at hand, we are now in a position 
to state
the following theorem.

\begin{theorem}[Riemannian Penrose Inequality for the $1$-harmonic Mass]
\label{thm:RPI-1harm}
Let $(M,g)$ be a compact, three-dimensional Riemannian band, with scalar curvature satisfying
\[
\RRR\geq 2 \Lambda, 
\]
for some $\Lambda >0$, and with compact minimal boundary $\partial M$ decomposing into two connected components
\[
\partial M \, = \, \partial M^- \sqcup \partial M^+,
\]
such that $|\partial M^-| < |\partial M^+|$. 
Assume that $\pa M^-$ is strictly stable, that $\pa M^+$ is simply connected and unstable and that there are no closed minimal surfaces in $M\setminus \pa M$. Then, the $1$-harmonic Mass of $(M,g)$ defined in~\eqref{def:1harm} satisfies
\begin{equation}
\label{eq:RPI_ABM}
\mathfrak{m}_\Lambda^{(1)} (M,g) \, \geq \, \sqrt{\frac{|\partial M^-|}{16 \pi}} \left( 1 - \frac{\Lambda}{12 \pi} 
|\partial M^-|\right),
\end{equation}
with equality if and only if $(M, g)$ is isometric to the Schwarzschild–de Sitter solution with mass parameter $m = \mathfrak{m}_\Lambda^{(1)} (M,g)$.
\end{theorem}
The inequality~\eqref{eq:RPI_ABM} is formally compatible with the classical Riemannian Penrose inequality for time-symmetric, asymptotically flat initial data with vanishing cosmological constant $\Lambda = 0$. It is likewise consistent with the Penrose-type inequalities formulated by Bray and Chruściel in~\cite{Bra_Chr_2004} for time-symmetric, asymptotically hyperbolic initial data with negative cosmological constant. In the case where the outermost minimal boundary is connected, these inequalities take the form
\begin{equation*}
\mathfrak{m}_{\rm Haw}(M,g) \,\geq\, \sqrt{\frac{|\partial M|}{16\pi}} \left( \frac{1}{2}\,\chi(\partial M) \;-\; \frac{\Lambda}{12\pi}\,|\partial M| \right) \, ,
\end{equation*}
where $\chi(\partial M)$ denotes the Euler characteristic of the boundary and $\mathfrak{m}_{\rm Haw}(M,g)$ denotes the total Hawking mass of the manifold. As already discussed, this concept is unfit to provide a reliable concept of mass when $\Lambda>0$. The above theorem suggest that the $1$-harmonic Mass could provide a valid substitute.  

\subsection{Organization of the paper.}

In Section~\ref{sec:MF}, we establish monotonicity formulas on Riemannian bands under the assumption of a lower bound on the scalar curvature. These formulas are substantially more general than those required for the arguments developed in the rest of the paper. Nonetheless, they serve a dual purpose: on the one hand, they situate our contribution within a broader research perspective; on the other hand, they are expected to play a central role in forthcoming extensions of the theory in several natural directions, including a comprehensive analysis of Penrose-type inequalities in the presence of a nonzero cosmological constant. In Section~\ref{sec:selection}, we specialize to the case of globally defined $p$-Green’s functions with $1<p<3$. Starting from the study of rotationally symmetric solutions, we then identify a suitable choice of the structural coefficients that appear in the general formulation of the monotonicity formulas, in such a way as to characterize static zero-mass metrics (space forms). These metrics constitute the natural reference configurations for any proposed notion of mass and are therefore required to arise as the rigidity case in inequalities intended to yield positive mass theorems.
In Section~\ref{sec:further}, we begin our detailed analysis of the case of a positive cosmological constant, which is the main focus of the present work, and we refine the description of the structural coefficients. These will be crucial for evaluating the limits of the monotone quantities at both the pole of the Green’s function and the boundary of the manifold.
This, in turn, prepares the ground for a first group of main results, namely Theorems~\ref{thm:pos_pass_pol_Intro} and~\ref{thm:pos_pass_total_Intro}, corresponding, respectively, to Theorem~\ref{thm:pos_pass_pol} in Section~\ref{sec:polar_PMT} and Theorem~\ref{thm:pos_pass_total} in Section~\ref{sec:blowup}. 
In Section~\ref{sec:further}, we also provide additional remarks on the behavior of the structural coefficients in the limit as $p\to 1^+$, which support some of the more speculative considerations presented in Subsection~\ref{sub:pto1}. 
These considerations underlie the notion of $1$-harmonic Mass, introduced in formula~\eqref{eq:1harmass} and discussed in Subsection~\ref{sub:1harm}. In this subsection we state a positive mass theorem and a Penrose inequality for the $1$-harmonic Mass,
both including a rigidity statement. These results are then proved in Section~\ref{sec:harm_1} and constitute the second group of main results of the present work. Appendix~\ref{app:ODE} contains auxiliary material needed for the treatment of the case $\Lambda>0$ in Theorem~\ref{thm:global}, as well as for the analysis carried out in Section~\ref{sec:further}. The material in Appendix~\ref{app:pharmonic} underpins the arguments developed in Sections~\ref{sec:polar_PMT} and~\ref{sec:blowup}.
Appendix~\ref{app:pole-imcf} establishes existence and uniqueness for the weak inverse mean curvature flow with an outer obstacle emanating from a point and provides the proof of this result, which underlies the definition of the $1$-harmonic Mass.

\section{Monotonicity formulas for \texorpdfstring{$p$}{p}-harmonic functions on Riemannian bands}
\label{sec:MF}

In this section, we are going to establish monotonicity formulas holding along the level set flow of $p$-harmonic functions defined on Riemannian $3$-manifolds, whose scalar curvature admits a uniform lower bound. In the case of manifolds with nonnegative scalar curvature, these formulas are closely related -- and in some cases coincide -- with the ones previously obtained in~\cite{AMO,AMMO}. We note that a substantial number of novel monotonicity formulas and integral identities, derived through the use of harmonic functions or their generalizations have recently emerged in the literature. These developments have significantly advanced the study of the geometry of scalar curvature, with notable implications for geometric inequalities as well as for the mathematical aspects of general relativity. Without attempting to provide an exhaustive list, we mention here the following contributions~\cite{Ago_Maz_CMP,Ago_Maz_CV,Hir_Kaz_Khu,hirsch,Stern,Mun_Wan_JFA,ChaChuLeeTsa,Miao_PMJ,Bor_Ced_Cog,Ced_Cog_Feh}.

\smallskip

Let $(M,g)$ be a smooth $3$-dimensional Riemannian manifold with scalar curvature larger than $2 \Lambda \in \mathbb{R}$ and let  
$\Omega \subseteq M\setminus\pa M$ 
be a {\em Riemannian band}, that is a smooth domain, whose boundary has precisely two connected components, $\pa \Omega = \pa\Omega^+ \sqcup \pa\Omega^-$.
For $1<p$ and for $S<T$, let $u$ be a $p$-harmonic function satisfying
\begin{equation}
\label{eq:up_band}
\begin{dcases}
\Delta_p u\,=\,0 & \hbox{in } \Omega\,,
\\
\quad \, u \,  = \, e^{-S/(p-1)} & \hbox{on } \pa\Omega^-,
\\
\quad \, u \,  = \, e^{-T/(p-1)} & \hbox{on } \pa\Omega^+,
\end{dcases}
\end{equation}
where the $p$-Laplacian is defined as $\Delta_p \phi = {\rm div} (|\nabla \phi|^{p-2} \nabla \phi)$. The fact that the above problem admits a unique solution $u \in \mathcal{W}^{1,p}(\Omega) \cap \mathscr{C}^0(\overline \Omega)$ is standard, see for example by~\cite[Theorem 2.16]{LIN17} for a proof based on the direct method in the calculus of variations. Following~\cite{Moser2007}, we set $w=-(p-1)\log u$, so that $w$ satisfies
\begin{equation}
\label{eq:wp_band}
\begin{dcases}
\Delta_p w =\, |\na w|^p & \hbox{in }\Omega\\
\quad \,\, w = S & \hbox{on } \pa\Omega^-,\\
\quad \,\, w = T & \hbox{on } \pa\Omega^+.
\end{dcases}
\end{equation}
To fix the notation, we set 
\[
\Omega_t \, = \,  \{x \in \Omega  :  w(x) < t\}    \quad \hbox{and} \quad \Sigma_t \,= \, \{x \in \Omega  :  w(x) = t\}\,,
\]
and, for every regular value
$t \in [S,T]$, we denote by $\HHH$ the mean curvature of $\Sigma_t$, computed with respect to the unit normal $\na w/|\na w|$. For $1<p$ and $\Lambda \in \R$, we introduce the function 
$\mmp_\Lambda : [S,T] \rightarrow \R$, setting
\begin{equation}
\label{mp_gen}
\mmp_\Lambda (t) \, = \int_{S}^t e^{\lambda(\tau)} \big(4\pi - \Lambda|\Sigma_\tau|\big)  \, d\tau
\, - \,\,  e^{\lambda(t)}  \int_{\Sigma_t} |\nabla w| \,  \big(  \HHH - \mu(t)|\nabla w|  \big) \, d\sigma \, ,
\end{equation}
where $\lambda, \mu \in \mathscr{C}^1([S,T])$ are functions satisfying the following semi-decoupled 
ODE system
of {\em structural relationships}:
\begin{equation}
\label{structural}
\begin{cases}
\dot{\mu} \,  = \, \left(\frac{5-p}{p-1}\right)\alpha^2 - \left[\left(\frac{5-p}{p-1}\right)\alpha+\left(\frac{1}{p-1}\right)\right]\mu + \left(\frac{3-p}{p-1}\right)\mu^2 \\[2ex]
\dot{\lambda} \,  =  \, \left[\left(\frac{5-p}{p-1}\right)\alpha-\left(\frac{1}{p-1}\right)\right] - \left(\frac{3-p}{p-1}\right)\mu 
\end{cases}
\end{equation}
for some given $\alpha \in \mathscr{C}^0([S,T])$. In order to deduce meaningful geometric applications, we are going to discuss in the next section how the functions $\alpha, \lambda$ and $\mu$ should be 
further specified. Before 
that, we find it convenient to state and prove a general monotonicity formula for the functions $t \mapsto \mmp_\Lambda (t)$.

\begin{theorem}
[General Monotonicity Formula] 
\label{GMF}
Let $(M,g)$ be a smooth $3$-dimensional Riemannian manifold, whose scalar curvature $\RRR$ satisfies 
\[
\RRR \, \geq \, 2 \Lambda \, , 
\]
for some $\Lambda \in \R$. Let 
$\Omega \subseteq M\setminus\pa M$
be a smooth bounded domain with two connected boundary components, so that 
\[
\pa \Omega = \pa\Omega^+ \sqcup \pa\Omega^- \, ,
\]
and assume that either $H_2(\Omega, \pa \Omega^-; \Z) = \{0\}$ or $H_2(\Omega, \pa \Omega^+; \Z) = \{0\}$. For
$1<p\leq5$ and $S<T$, let 
$w$ be the unique solution to problem~\eqref{eq:wp_band}. Then, the function $t \mapsto \mmp_\Lambda (t)$ defined in~\eqref{mp_gen}, whose coefficients
$\mu$ and $\lambda$
obey the structural relationships~\eqref{structural}
for some $\alpha\in\mathscr C^0([S,T])$,
satisfies 
\begin{equation}
\label{monogen}
s \leq t \qquad \Longrightarrow \qquad \mmp_\Lambda (s) \leq \mmp_\Lambda (t) \, ,
\end{equation}
for every $s,t \in [S,T ]$ regular values of $w$.
\end{theorem}

Some remarks are in order regarding this general statement, to motivate
some of the assumptions and clarify the range of applicability of the result.

\begin{remark}[Geroch-type Formula]
If the critical values of $w$ are negligible, 
one can prove that 
the function $t \mapsto \mmp_\Lambda(t)$ is absolutely continuous and monotone on its domain. Its weak derivative is given almost everywhere by
\begin{align}
\label{dermp}
\frac{d \mmp_\Lambda}{d t} \, = \, 
e^{\lambda(t)}
& \left\{  
\left(4\pi - \int_{\Sigma_t} \frac{\RRR^{\Sigma}}{2} d\sigma \right) + \int_{\Sigma_t} \left|\frac{\nabla^{\Sigma}|\nabla w|}{|\nabla w|}\right|^2 + \frac{|\mathring{\rm h}|^2}{2} d\sigma \right. \nonumber \\
& \left. + \int_{\Sigma_t} \left(\frac{\RRR-2\Lambda}{2}\right) d\sigma + \left(\frac{5-p}{p-1}\right) \int_{\Sigma_t} \left(\frac{\HHH}{2} - 
\alpha(t) 
|\nabla w|\right)^{\!2} d\sigma \right\} \geq 0 \, ,
\end{align}
where, for any regular value $t \in [S,T]$, $\Sigma_t = \{x \in \Omega : w(x) = t\}$, and $\na^\Sigma$, $\RRR^{\Sigma}$, $\HHH$, and $\mathring{\rm h}$ denote respectively its covariant derivative, scalar curvature, mean curvature, and traceless second fundamental form, computed with respect to the unit normal $\nabla w / |\nabla w|$. 
This formula is analogous to Geroch’s expression~\cite{Ger} for the derivative of the Hawking mass along an inverse mean curvature flow, except for the last term, which, on one hand dictates the restriction $p\leq5$ and on the other can be employed to determine the choice of the structural coefficient $\alpha$, tuning the formula to the chosen comparison model.
\end{remark}

\begin{remark}
[Preserving Connectedness]
\label{connessione}
The topological assumption about the vanishing of the second relative homology group (i.e., $H_2(\Omega, \pa \Omega^-; \Z) = \{0\}$ or $H_2(\Omega, \pa \Omega^+; \Z) = \{0\}$) guarantees that all the regular level sets of $w$ are connected, so that 
$$
\int_{\Sigma_t} \frac{\RRR^{\Sigma}}{2} d\sigma  \, = \, 2\pi \, \chi (\Sigma_t) \, \leq \, 4\pi \, ,
$$
by the Gauss-Bonnet Theorem and the first summand on the right hand side of~\eqref{dermp} is nonnegative. 
\end{remark}

The monotonicity depends only on a few fundamental ingredients: the existence of a 
$p$-harmonic function -- indeed, no boundary conditions are required, but solely the validity of the equation itself --, the connectedness of its level sets (see Remark~\ref{connessione}), and, for a fixed choice of the structural coefficient $\alpha$, the existence of a global solution 
$(\mu,\lambda)$
to the ODE system~\eqref{structural}. Note that the first equation in~\eqref{structural} is of Riccati type, and hence some of its solutions may exist only up to a finite blow-up time.
In Section~\ref{sec:selection}, we will make a suitable choice of $\alpha$ for $1<p<3$, which guarantees the global existence of solutions to~\eqref{structural}. In addition to global existence, further selection criteria will be introduced in order to single out, among all global solutions, those that are effective in deriving geometrically meaningful statements. Such a selection must be tuned to the reference model case: for this reason, and since the qualitative behaviour of the model $p$-Green's function changes across the regimes $1<p<3$, $p=3$, and $p>3$, from Section~\ref{sec:selection} onward we will restrict our attention to the case $1<p<3$.
Regarding the range $3\leq p\leq 5$ in the statement of Theorem~\ref{GMF}, we limit ourselves to the following considerations. For $p=3$ the system \eqref{structural} becomes linear and, consequently, for any given $\alpha \in \mathscr{C}^0([S,T])$, the corresponding solutions $\mu$ and $\lambda$ are globally defined.
Finally, for $p=5$, which represents the borderline case for the validity of the monotonicity, both the system of structural relations~\eqref{structural} and the first derivative formula~\eqref{dermp} no longer depend on $\alpha$. Moreover, since in this case the system has constant coefficients, one can explicitly compute closed-form expressions for the globally defined solutions without difficulty.

\begin{proof}[Proof of Theorem~\ref{GMF}] Let us justify the smooth computation leading to~\eqref{dermp}, starting by the case where the $p$-harmonic function $w$ has no critical point in $\Omega$. In this case, we can use the function $w$ itself as a coordinate. By standard computations, using~\eqref{eq:wp_band} and
the traced Gauss Equation, one can easily deduce the identities 
\begin{align*}
\frac{\pa |\na w|}{\pa w} = & \,\, \frac{|\na w| - \HHH}{p-1} \, , \\
|\na w| \, \frac{\pa \HHH}{\pa w}  = & \,\, {\rm div}^\Sigma \! \left( \frac{\na^\Sigma |\na w|}{|\na w|}\right) - 
\left|\frac{\nabla^{\Sigma}|\nabla w|}{|\nabla w|}\right|^2 - \frac{|\mathring{\rm h}|^2}{2} - \left(\frac{\RRR-2\Lambda}{2}\right)  
+ \left(\frac{\RRR^\Sigma-2\Lambda}{2}\right)  
- \, \frac34 \, \HHH^2  \,, \\
\frac{\pa}{\pa w} d\sigma = & \,\, 
\frac{\HHH}{|\na w|} 
\,  d \sigma \, ,
\end{align*}
recalling that ${\rm h}$ and $\HHH$ are computed with respect to $\na w/|\na w|$. Differentiating~\eqref{mp_gen} with the help of the previous expressions, one gets
\begin{align*}
\frac{d \mmp_\Lambda}{d t} \, = \, e^{\lambda}& \left\{  
  \left(4\pi - \int_{\Sigma_t} \frac{\RRR^{\Sigma}}{2} d\sigma \right) + \int_{\Sigma_t} \left|\frac{\nabla^{\Sigma}|\nabla w|}{|\nabla w|}\right|^2 + \frac{|\mathring{\rm h}|^2}{2} d\sigma \right. \nonumber \\
& \left. + \int_{\Sigma_t} \left(\frac{\RRR-2\Lambda}{2}\right) d\sigma + \left(\frac{5-p}{p-1}\right) \int_{\Sigma_t} \left(\frac{\HHH}{2} - \alpha |\nabla w|\right)^{\!2} d\sigma \right. \\
& \left. -  \left[ \, \dot\lambda  - \left(\frac{5-p}{p-1}\right) \alpha + \left(\frac{1}{p-1}\right) + \left(\frac{3-p}{p-1}\right) \mu \, \right]    \, \int_{\Sigma_t} |\na w| \HHH\,  d\sigma  \right.      \\
&  +  \left[ \, \dot\mu  - \left(\frac{5-p}{p-1}\right) \alpha^2 + \dot\lambda \, \mu + \left(\frac{2}{p-1}\right) \mu \, \right]    \, \int_{\Sigma_t} |\na w|^2 \,  d\sigma 
\Bigg\} \, .
\end{align*}
Imposing the {\em structural relationships}~\eqref{structural}, one can easily check that the last two terms on right hand side 
vanish, so that~\eqref{dermp} holds true. In particular, it is evident that, as long as $1<p\leq 5$, the right hand side is nonnegative.

To complete the proof, one needs to discuss the case where critical points of $w$ are present. However, the desired conclusion follows 
by a straightforward adaptation of the methods and techniques developed in~\cite[Sections 1.2--1.5]{AMMO} to treat the $\Lambda = 0$ case, and we leave the details to the interested reader (notice that for $p\in(1,2]$ the same computation can also be justified exploiting~\cite[Theorem~B.1]{Ben_Plu_Poz}). Hereafter, we only indicate the key observations and recall the main features of the proof's strategy. First of all, we observe that the monotone quantity $t \mapsto\mmp_\Lambda (t)$ can be re-written as
\[
\mmp_\Lambda (t) \,\, = \,\, \int_{\Sigma_t} \left\langle \! X  \,  \bigg|  \, \frac{\nabla w}{|\nabla w|}  \right\rangle \, d \sigma \, , 
\]
where the vector field $X$ is given by
\begin{align*}
X & =  e^{\lambda(w)}  
\left\{  
\left(\int_{S}^w \!\! \frac{4 \pi - \Lambda |\Sigma_\tau|}{e^{\lambda(w)- \lambda(\tau)}} \, d\tau \right)\cdot
\frac{|\na w|^{p-2} \na w}{e^{w} \, {K_p}} + \na |\na w| - \frac{\Delta w}{|\na w|} \na w  \, + \,\mu(w) \cdot  |\na w| \na w
\right\} \, .
\end{align*}
The constant $K_p$ appearing in the above expression is the one matching the constant value of the function 
\begin{equation}
\label{int_costante}
t \, \longmapsto \, e^{-t}\int_{\Sigma_t} |\na w|^{p-1}\, d \sigma.
\end{equation}
A direct computation, using the structural relationships~\eqref{structural}, shows that, far away from the critical points of $w$, it holds
for $s < t$
\begin{align}
\notag
\mmp_\Lambda (t) - \mmp_\Lambda (s)\,\,& = \,\, \int_{\Sigma_t} \left\langle \! X  \,  \bigg|  \, \frac{\nabla w}{|\nabla w|}  \right\rangle \, d \sigma \, - \int_{\Sigma_s} \left\langle \! X  \,  \bigg|  \, \frac{\nabla w}{|\nabla w|}  \right\rangle \, d \sigma \, \\
\notag
& = \,\, \int_{\{s<w<t\}} \!\!\!\!\!\!\!\!{\rm div} X \,\, d \mu \\
\notag
& = \int_s^t \!d\tau \cdot
e^{\lambda(\tau)}
\left\{  
  \left(4\pi - \int_{\Sigma_\tau} \!\!\frac{\RRR^{\Sigma}}{2} \,d\sigma \right) + \int_{\Sigma_\tau} \left|\frac{\nabla^{\Sigma}|\nabla w|}{|\nabla w|}\right|^2 + \frac{|\mathring{\rm h}|^2}{2} \, d\sigma \right. \nonumber \\
\label{eq:m_Lambda_monotonicity}
&  \qquad\qquad\  + \!\int_{\Sigma_\tau} \!\!\left(\frac{\RRR-2\Lambda}{2}\right) d\sigma + \left(\frac{5-p}{p-1}\right) \int_{\Sigma_\tau} \!\!\left(\frac{\HHH}{2} - \alpha |\nabla w|\right)^{\!\!2} \!d\sigma  
\Bigg\} \, , 
\end{align}
where 
we have also used
the Divergence Theorem and the Coarea Formula. According to~\cite{AMMO}, we now divide the discussion into two subcases, depending on whether the set of critical values is negligible or nonnegligible. If the set of critical values $\mathcal T$ is negligible, the monotonicity property of $\mmp_\Lambda$ can be established, following 
step-by-step the smoothing procedure adopted in~\cite[Section 1.3]{AMMO}.
More precisely, with the help of  suitably chosen cut-off functions, one can show that for $s < t$
\[
\mmp_\Lambda (t) - \mmp_\Lambda (s)
\,\geq\,
\int_{(s,t)\setminus\mathcal T}
d\tau\int_{\Sigma_\tau}\frac{{\rm div}X}{|\na w|}\,d\sigma
\,\geq\,0,
\]
where the last inequality follows from \eqref{eq:m_Lambda_monotonicity}. For the sake of clarity, and without compromising the heuristic nature of the argument, we shall say that the cut-off functions are of the form $\chi(f(w)|\nabla w|)$. They are chosen so as to satisfy the dual requirement of vanishing in a neighborhood of the critical points and having a gradient of the form $\dot\chi(f(w)|\nabla w|)\, \big[ f(w)\nabla|\nabla w| + \dot{f}(w) |\na w| \, \na w \big]$, whose scalar product with the vector field $X$ has the desired sign. Under these conditions, the divergence of the smoothed-off vector field $\chi(f(w)|\nabla w|)\,X$ retains the structural properties that are needed for our purposes.

Finally, to accomodate for the case where the set of critical values is possibly nonnegligible, one has to consider the standard smooth approximations of the $p$-harmonic 
functions, introduced by DiBenedetto in~\cite{DiBenedetto1}, 
and derive
a family of approximate monotonicity formulas
associated with them, 
eventually leading to the desired monotonicity~\eqref{monogen}. Again, this follows with minor changes from exactly the same arguments employed in~\cite[Section 1.4 and Section 1.5]{AMMO}.
\end{proof}

\subsection{Riemannian bands vs punctured domains.}

Beginning with the next section and continuing throughout the remainder of this work, our analysis will concentrate on applying the formulas obtained in Theorem~\ref{GMF} to geometrically meaningful configurations. In particular, by means of Theorems~\ref{thm:pos_pass_pol} and~\ref{thm:pos_pass_total}, we aim to formulate results in the spirit of the Positive Mass Theorem for time-symmetric initial data with positive cosmological constant, thereby addressing several of the questions that have emerged from the study of Min-Oo’s conjecture about the concept of mass. As indicated by the preliminary investigations in~\cite{AMO}, the most appropriate analytic tools to use for these purposes are the Green’s functions associated with the $p$-Laplacian -- or $p$-Green's functions, for short. These may be defined either on a compact manifold with smooth boundary under homogeneous Dirichlet boundary conditions -- an assumption that is particularly natural in the presence of a positive cosmological constant -- or on a complete, non-compact manifold. In the latter setting, it is often appropriate to impose that the manifold admits either an asymptotically flat end ($\Lambda=0$) or an asymptotically hyperbolic end ($\Lambda<0$), according to the sign of the cosmological constant. In all cases under consideration, we will be dealing with punctured domains rather than Riemannian bands; consequently, we must justify the seemingly less natural decision to formulate the monotonicity formulas in this second framework rather than in the first.
To begin with, we observe that, from a technical standpoint, the monotonicity formulas established for Riemannian bands automatically yield the corresponding formulas on punctured domains of the form $\Omega \setminus \{x_0\}$. Indeed, given a $p$‑Green's function $u_p$ with zero boundary data and pole at $x_0$, one can construct a natural exhaustion $\{\Omega^{(\varepsilon)}\}_{\varepsilon>0}$ of the punctured domain by setting
\[
\Omega^{(\varepsilon)} \, = \, \{ x \in \Omega \, : \, \varepsilon < u_p(x) < 1/\varepsilon\} \, .
\]
At this stage, it is enough to apply Theorem~\ref{GMF} repeatedly, with $S=(p-1)\log \varepsilon$ and $T=-(p-1)\log \varepsilon$, in order to derive the desired monotonicity properties.

An additional advantage of working on Riemannian Bands is that the existence of a solution to problem~\eqref{eq:up_band} is always guaranteed, a fact that is not generally obvious in more general settings. For instance, in the case of noncompact domains, one must analyze on a case‑by‑case basis the conditions ensuring the existence of a minimal positive $p$‑Green's function, which may require restricting the admissible range of the exponent $p$ or imposing volume growth assumptions so as to verify the requirements of the so‑called $p$‑nonparabolicity.

Finally, Riemannian bands provide the natural geometric setting for a variety of problems that constitute natural continuations of the present work and that we intend to investigate in future works. A primary example is the study of inequalities of Riemannian Penrose type (see~\cite{HI, bray1, bray3, AMMO}), with particular emphasis on the situations in which the cosmological constant is nonzero (see~\cite{Ambrozio,LeeNev} for the case $\Lambda < 0$ and~\cite{Bor_Maz_2} for the case of static metrics with $\Lambda > 0$).

\section{A selection principle for the structural coefficients}
\label{sec:selection}

In this section, we show how to select in a geometrically meaningful way the structural coefficients $\alpha$, $\mu$, and $\lambda$ appearing in the monotonicity formulas~\eqref{mp_gen}-\eqref{dermp} and obeying the ODE system
\begin{equation*}
\begin{cases}
\dot{\mu} \,  = \, \left(\frac{5-p}{p-1}\right)\alpha^2 - \left[\left(\frac{5-p}{p-1}\right)\alpha+\left(\frac{1}{p-1}\right)\right]\mu + \left(\frac{3-p}{p-1}\right)\mu^2  , \\[2ex]
\dot{\lambda} \,  =  \, \left[\left(\frac{5-p}{p-1}\right)\alpha-\left(\frac{1}{p-1}\right)\right] - \left(\frac{3-p}{p-1}\right)\mu \,  .
\end{cases}
\end{equation*}  
More precisely, we want to design the monotone quantities $\mmp_\Lambda$ is such a way that their constancy characterizes the underlying manifold as a space form. This perspective is motivated by the observation that, within the framework of positive mass theorems, it is natural to adopt space forms as reference geometries, in which the mass identically vanishes and rigidity holds. It is, however, important to emphasize that there exist other mathematically and physically significant settings in which the reference geometries are given by warped product manifolds with non-constant sectional curvature, as exemplified by the Schwarzschild–Kottler metrics, which realize equality in the Riemannian Penrose inequalities. The monotonicity formulas obtained in Theorem~\ref{GMF} are sufficiently robust and versatile to be extended to this setting as well, providing, through their constancy, a characterization of the rigidity of the underlying manifold and thereby implying that it is isometric to the corresponding model space. For the sake of clarity and coherence of the present exposition, we defer a detailed analysis of this generalization to future work, and, for now, restrict our attention to the problem of identifying monotone quantities specifically designed for the geometric characterization of space forms.

To this aim, we consider 
the situation where the $p$-harmonic function, that gives rise to the monotonic quantities $\mmp_\Lambda$ described in~\eqref{mp_gen}, is actually a $p$-Green's function solving
\begin{equation*}
\Delta_p u \, = \, - 4 \pi \delta_x,
\end{equation*}
where $x$ is a point in the interior of $M$. Here, as before, we let $(M,g)$ be a $3$-dimensional Riemannian manifold  with scalar curvature bounded from below (i.e., $\RRR \geq 2 \Lambda$), and from now on we let $p$ ranging between $1$ and $3$:
\[
1<p<3.
\]
Depending on the sign of the cosmological constant $\Lambda \in \R$, it will be natural to consider compact manifolds with boundary ($\Lambda>0$), asymptotically flat manifolds ($\Lambda=0$), or asymptotically hyperbolic manifolds ($\Lambda<0$).
Correspondingly, we will impose the natural Dirichlet boundary conditions, namely 
$u=0$ on $\pa M$
or $u\to0$ at $\infty$.
To see how this setting fits into the framework of the previous section, we 
let $w$ to be the function fulfilling the problem
\begin{equation}
\label{wdistr}
{\rm div} \big( e^{-w} |\nabla w|^{(p-2)} \nabla w \big) = 4 \pi (p-1)^{(p-1)} \delta_x,
\end{equation}
and recall that it is related to $u$ via the formula $w=-(p-1)\log u$. 
Ideally, this would correspond in problem \eqref{eq:wp_band} to consider boundary components $\pa \Omega^-$ and $\pa \Omega^+$ of $\Omega$ that are arbitrarily small ($S \to - \infty$) and an arbitrarily large ($T \to + \infty$) level sets of $w$.

\subsection{The guideline case \texorpdfstring{$\Lambda=0$}{Λ=0}.}

Note that the case $\Lambda=0$ has been fully analyzed in~\cite{AMO,AMMO}. Nevertheless, we find it convenient to illustrate the selection process of the structural coefficients through this illuminating example. 
First of all, let $(M,g)$ be the flat Euclidean $3$-space endowed with a rotationally symmetric $p$-Green's function (and thus a rotationaly symmetric $w$) whose pole is located at the origin. As we are going to tune the function $t \mapsto \mmp_0 (t)$ so that it is constantly vanishing in this special framewok, the last summand on the right-hand side of \eqref{dermp} has to vanish as well.
This dictates the relationship
\begin{equation}
\label{defalfa}
\frac{\HHH}{2} \, = \, \alpha |\na w|, 
\end{equation}
where $\HHH$ is the mean curvature of the level sets $\Sigma_t= \{ w= t \}$, $t \in\R$, computed with respect to the unit normal $\na w/|\na w|$. Since, in this favorable context, everything depends on one single variable, we compute 
\[
\frac\HHH2\,=\,
\frac{e^{-\frac t{3-p}}}
{\Big(\frac{p-1}{3-p}\Big)^\frac{p-1}{3-p}},
\qquad
\hbox{and} \qquad
|\na w|\,=\,
(3-p)\frac{e^{-\frac t{3-p}}}
{\Big(\frac{p-1}{3-p}\Big)^\frac{p-1}{3-p}},
\]
on each $\Sigma_t= \{ w= t \}$, $t \in \R$.
In turn, equation~\eqref{defalfa} leads us  to 
\begin{equation}
\label{astd}
\alpha(t) \, \equiv \,\frac1{3-p} \, , \, \qquad t \in \R
\end{equation}
Given this specific choice of the function $\alpha$, one can then focus on the ODE system~\eqref{structural}. First, observe that the system is semi-decoupled in the sense that, once $\mu$ is prescribed, the function $\lambda$ is determined up to an additive constant $\kappa$. Regarding the first equation in the aforementioned ODE system, we note that the globally defined solutions are precisely those trajectories that remain confined between the two equilibrium points
\[
\frac{5-p}{(3-p)^2}
\qquad\text{and}\qquad
\frac{1}{3-p}.
\]
Furthermore, there exists precisely one solution that converges to \(1/(3-p)\) as \(t \to -\infty\), namely the constant equilibrium solution \(\mu(t) \equiv 1/(3-p)\). This uniqueness of the asymptotic behavior as $t\to -\infty$ -- a phenomenon that in fact occurs in all other configurations as well, see Theorem~\ref{thm:selection_mu} -- naturally justifies the selection of the solution \(\mu\) that exhibits this distinguished property. In particular, this choice is the physically relevant one, provided that the integration constant \(\kappa\) appearing in the consequent explicit representation of \(\lambda\) 
\[
\lambda \, = \, \frac{t}{3-p} \, + \, \kappa
\]
is fixed so that
\[
\lim_{t\to -\infty}\frac{\mmp_0(x,t)}{|\Omega_t|}
\,=\,
\frac{\RRR(x)}{16\pi},
\]
where $\Omega_t := \{ w < t \}$. All in all, the right choice for $\lambda$ and $\mu$ in the case $\Lambda=0$ is
\begin{equation}
\label{eq:lomo}
e^{\lambda(t)} 
\,=\, 
\Big(\frac{p-1}{3-p}\Big)^\frac{p-1}{3-p}
\frac{e^{\frac t{3-p}}}{8\pi (3-p)}, 
\qquad\qquad\qquad
\mu(t)= \frac1{3-p}.  
\end{equation}
In turn, the monotonicity formulas specialize to
\begin{align*}
\mmp_0 (x,t) &
\, = \,
\Big(\frac{p-1}{3-p}\Big)^{\!\!\frac{p-1}{3-p}}\frac{e^{\frac{t}{3-p}}}{8\pi(3-p)}  
\left\{ 4 \pi(3-p) 
\, -     \int_{\Sigma_t} |\nabla w|\left(  \HHH - \frac{|\nabla w|}{3-p}  \right)  d\sigma   \right\} \, , \\
\frac{d \mmp_0}{d t} (x,t)
&
\, = \,
\Big(\frac{p-1}{3-p}\Big)^{\!\!\frac{p-1}{3-p}}\frac{e^{\frac{t}{3-p}}}{8\pi (3-p)} \left\{  
\left(4\pi - \int_{\Sigma_t} \!\frac{\RRR^{\Sigma}}{2} d\sigma \right) + \int_{\Sigma_t} \left|\frac{\nabla^{\Sigma}|\nabla w|}{|\nabla w|}\right|^2 + \frac{|\mathring{\rm h}|^2}{2} d\sigma \right. \nonumber \\
& \qquad\qquad\qquad\qquad
\qquad\qquad\left. + \int_{\Sigma_t} \!\frac{\RRR}{2} \,d\sigma + \left(\frac{5-p}{p-1}\right) \int_{\Sigma_t} \left(\frac{\HHH}{2} - \frac{|\nabla w|}{3-p}\right)^{\!2} \!d\sigma \right\},
\end{align*}
which are exactly those derived in~\cite{AMO,AMMO}, up to a multiplicative constant.

\medskip

The remainder of this section is structured as follows. In Subsection~\ref{sub:selection_alpha}, we first establish the notation for the model solutions and subsequently introduce the structural coefficient $\alpha$, analyzing its asymptotic behavior when $t \to -\infty$.
In the next Subsection~\ref{subsec:selection}, we show how to select two special ancient solutions $\mu$ and $\lambda$ of system~\ref{structural} and verify that the choice made is compatible with the validity of a small sphere limit for the corresponding monotone quantities. Finally, in Subsection~\ref{sub:global}, we show that the solutions selected up to this point actually exist for all times. In the case $\Lambda > 0$, we will also study the asymptotic behavior of these solutions as $t \to \infty$, which will allow us, in the subsequent Section~\ref{sec:polar_PMT}, to provide a more detailed and explicit description of our mass-type invariants.

\subsection{Green's functions for the \texorpdfstring{$p$}{p}-Laplacian on 3-dimensional spaceforms and the selection of \texorpdfstring{$\alpha$}{α}}
\label{sub:selection_alpha}

We begin by expressing the metric of the space forms in so‑called static coordinates, after normalizing the scalar curvature to be constantly equal to \(2\Lambda\). Although these coordinates are less commonly employed, they facilitate, in many instances, a uniform notation and, consequently, a more streamlined and coherent treatment. Furthermore, in the case of a positive cosmological constant, they exhibit -- as is well known -- an artificial coordinate singularity at the equator. This feature aligns well with our principal objective, which, from a purely geometric perspective, is the characterization of the hemisphere.
For $\Lambda \in \R$, we denote by $(M_\Lambda, g_\Lambda)$ the $3$-dimensional Riemannian manifold defined by
\[
M_\Lambda\,=\,
\begin{dcases}
\SSS_+^3 \, , & \hbox{if } \Lambda >0
\\
\R^3 \, ,  & \hbox{if }    \Lambda = 0 \qquad \qquad \hbox{and} \qquad \qquad  g_\Lambda\,=\,\frac{dr\otimes dr}{1-\frac{\Lambda}{3}r^2}+r^2g_{\mathbb{S}^2}\,,
\\
\mathbb{H}^3\, ,  & \hbox{if }    \Lambda < 0
\end{dcases}
\]
where $\SSS^3_+$ is the upper half sphere (together with its equatorial boundary), whereas $g_{\SSS^2}$ is the standard round metric on $\SSS^2$, with Gaussian curvature equal to $1$. The above expression for $g_\Lambda$ is naturally well defined for $0<r<R_\Lambda$, with 
\[
R_\Lambda\,=\,
\begin{dcases}
\sqrt{3/\Lambda} \, , & \hbox{if } \Lambda >0
\\
+ \infty \, ,  & \hbox{if }    \Lambda \leq 0 \, .
\end{dcases}
\]
It can be readily verified that all these manifolds have constant sectional curvature equal to $\Lambda/3$ and constant scalar curvature equal to $2\Lambda$.  
Furthermore, one easily observes that, for $1<p<3$, the function 
\begin{equation}
\label{eq:up_dS}
u_\Lambda^{\!(p)} : (0,R_\Lambda) \rightarrow (0,+\infty) \, , \qquad r \mapsto u_\Lambda^{\!(p)}(r)\,=\,\int_{r}^{R_\Lambda} \!\!\!\! \frac{1}{\rho^{\frac{2}{p-1}}\sqrt{1-\frac{\Lambda}{3}\rho^2}} \,\, d\rho\, ,
\end{equation}
solves, in the sense of distributions, the equation
\[
\Delta_p u \, = \, -4\pi \delta_o \, ,
\]
where $o$ denotes the point corresponding to $r=0$; moreover, notice that $u_\Lambda^{(p)}(r) \to 0$ as $r \to R_\Lambda$. In other words, $u_\Lambda^{(p)}$ is the Green's function of the $p$-Laplacian associated with the metric $g_\Lambda$, with homogeneous Dirichlet boundary conditions. In this context, we remark that for $\Lambda > 0$ the hypersurface $\{r = R_\Lambda\}$ represents the genuine smooth boundary of $M_\Lambda$, whereas for $\Lambda \leq 0$ the Dirichlet condition must be interpreted as a decay condition, requiring that the function tend to zero along the asymptotically flat end (when $\Lambda = 0$) or along the asymptotically hyperbolic end (when $\Lambda < 0$) of the manifold. For future convenience, observe that 
\begin{equation}
\label{u_L_dot}
\frac{d u_\Lambda^{(p)}}{dr} \, = \, - \,  \frac{r^{-\frac{2}{p-1}}}{\sqrt{1-\frac{\Lambda}{3} r^2}} \qquad \hbox{and} \qquad
|\na u_\Lambda^{(p)}| \, = \, r^{-\frac{2}{p-1}}
\end{equation}
Next we introduce the function 
\begin{equation}
\label{eq:wp_dS}
w_\Lambda^{\!(p)} : (0,R_\Lambda) \rightarrow (-\infty,+\infty) \, , \qquad r \mapsto w_\Lambda^{\!(p)} (r) = - (p-1) \log(u_\Lambda^{\!(p)}(r))\, ,
\end{equation}
and we observe that 
\begin{equation}
\label{cose_modello}
\frac{d w_\Lambda^{(p)}}{dr} \, = \, - \frac{p-1}{u_\Lambda^{(p)}} \, \frac{d u_\Lambda^{(p)} }{ dr } \qquad \hbox{and} \qquad
|\na w_\Lambda^{(p)}| \, = \, (p-1) \frac{r^{-\frac{2}{p-1}}}{u_\Lambda^{(p)}(r)} \, .
\end{equation}
Note in particular that the functions $u_\Lambda^{\!(p)}$ and $w_\Lambda^{\!(p)}$ are monotone and hence invertible. This feature is especially useful when it is convenient to change from the radial coordinate $r$, in terms of which the model is most naturally specified, to the coordinate $t$, which instead constitutes the natural variable for expressing the monotone quantities $\mmp_\Lambda$ introduced in the previous section. 
To this end, we define the inverse function of $w_\Lambda^{\!(p)}$ by
\begin{equation}
\label{eq:rp_dS}
r_\Lambda^{\!(p)} : (-\infty,+\infty) \rightarrow (0,R_\Lambda) \, , \qquad t \longmapsto r_\Lambda^{\!(p)}(t) = (w_\Lambda^{\!(p)})^{-1}(t) \, ,
\end{equation}
and note that it satisfies the relations
\[
w_\Lambda^{\!(p)} (r_\Lambda^{\!(p)} (t)) = t \,  \qquad \hbox{and} \qquad u_\Lambda^{\!(p)}(r_\Lambda^{\!(p)}(t)) = e^{-\frac{t}{p-1}} \, , \quad t \in \R \, .
\]
Back to the function $w_\Lambda^{\!(p)}$, as already observed, this function satisfies a relation of the form~\eqref{wdistr} in the sense of distributions and, in particular, away from the pole $o$, it satisfies the equation
\[
\Delta_p w = |\nabla w|^p \, ,\qquad \hbox{in } M_\Lambda \setminus\{o\}\, .
\]
It has been shown by various authors, starting from R. Moser~\cite{Moser2007}, then in~\cite{Kot_Ni} and~\cite{MRS}, up to the recent~\cite{Ben_Mar_Rig_Set_Xu} that this equation can be effectively used to approximate -- in the sense of uniform convergence on compact sets -- the weak inverse mean curvature flow of Huisken and Ilmanen~\cite{HI}, or, in the case of compact manifolds with boundary, the weak inverse mean curvature flow with outer obstacle of Xu~\cite{Xu}. Indeed, by letting $p \to 1$, one formally obtains the limiting PDE
\[
{\rm div} \left( \frac{\nabla w}{|\nabla w|} \right) \, = \, |\nabla w|
\]
which corresponds to the level-set formulation of the inverse mean curvature flow, in the sense that the level sets of the solution $w$, as the value of the solution changes, move in the normal direction with speed equal to the inverse of the mean curvature. Later on we will also benefit from these considerations when, in Subsection~\ref{sub:pto1}, we will discuss the formal limits of our monotone quantities as $p\to 1$. However, for the moment, the exponent $1<p<3$ can be thought of as a fixed parameter and, where it is not necessary, it will be systematically omitted from the notation, allowing us to write
\[
u_\Lambda = u_\Lambda^{\!(p)} \qquad \hbox{and} \qquad w_\Lambda = w_\Lambda^{\!(p)} 
\]
Recalling the relation~\eqref{defalfa} and noting that $|\nabla w_\Lambda|$ depends solely on the radial coordinate $r$, we compute the mean curvature of the level hypersurfaces determined by this coordinate, obtaining
\[
\HHH_\Lambda(r)\,=\,\frac{2}{r}\sqrt{1-\frac{\Lambda}{3}r^2} \,.
\]
We now have all the elements at hand to introduce the function $A^{\!(p)}_\Lambda : (0, R_\Lambda) \rightarrow \R$, together with the structural coefficient $\alpha^{\!(p)}_\Lambda : (-\infty, + \infty) \rightarrow \R$. We therefore set
\begin{align}
\label{eq:alpha_r}
A^{\!(p)}_\Lambda(r) \,&= \, \frac{\HHH_\Lambda}{2 |\nabla w_\Lambda|} (r) \, = \,
\frac{u_\Lambda(r)}{p-1} \,\, \sqrt{1-\frac{\Lambda}{3}r^2}\,\,\, r^{\frac{3-p}{p-1}} \, , \quad r \in (0,R_\Lambda)
\\
\label{eq:alpha_t}
\alpha^{\!(p)}_\Lambda (t)\,&= \, A^{\!(p)}_\Lambda (r^{\!(p)}_\Lambda (t)) \, = \,
\frac{e^{-\frac{t}{p-1}}}{p-1} \,\, \sqrt{1-\frac{\Lambda}{3}r_\Lambda^2(t)}\,\,\, r_\Lambda^{\frac{3-p}{p-1}}(t) \, , \quad t \in (-\infty,+ \infty) \, .
\end{align}
In order to get a little more acquainted with the quantities introduced above, it is instructive to observe that, in the particular case $p = 2$, they admit a more explicit representation. Indeed, by exploiting the fact that
\[
u_\Lambda^{\!(2)} (r)\,=\, 
\begin{dcases}
\frac{1}{r} \sqrt{ 1- \frac{\Lambda}{3}r^2} \, , & \hbox{if } \Lambda \geq0
\\
\frac{1}{r} \sqrt{ 1- \frac{\Lambda}{3}r^2} - \sqrt{\frac{|\Lambda|}{3}} \, ,  & \hbox{if }    \Lambda < 0 \, ,
\end{dcases}
\]
one immediately gets
\[
A_\Lambda^{\!(2)} (r)\,=\, 
\begin{dcases}
 1- \frac{\Lambda}{3}r^2 \, , & \hbox{if } \Lambda \geq0
\\
  1- \frac{\Lambda}{3}r^2 - \sqrt{ 1- \frac{\Lambda}{3}r^2}  \,\sqrt{\frac{|\Lambda|}{3}}  \,  r  \, ,  & \hbox{if }    \Lambda < 0 \, .
\end{dcases}
\]
For future convenience, we observe that, differentiating the identity $u_\Lambda^{\!(p)}(r_\Lambda^{\!(p)}(t)) = e^{-\frac{t}{p-1}}$ with the help of~\eqref{u_L_dot} and using the definition of $\alpha_\Lambda^{\!(p)}$, one obtains the differential relationship
\begin{equation}
\label{eq:dotr}
\frac{d r_\Lambda^{(p)}}{dt}\, = \, \alpha_\Lambda^{\!(p)} \cdot r_\Lambda^{\!(p)} \, .
\end{equation}
At the same time, combining~\eqref{cose_modello} with~\eqref{u_L_dot}, one gets the inverse relationship
\begin{equation}
\label{eq:dotrinv}
\frac{d w_\Lambda^{(p)}}{dr}\, = \, \frac{1}{ A_\Lambda^{\!(p)} \cdot r } \qquad \hbox{and} \qquad \frac{d w_\Lambda^{(p)}}{dr} (r_\Lambda^{(p)} (t)) \, = \, \frac{1}{  \alpha_\Lambda^{\!(p)} \!(t)\cdot r_\Lambda^{(p)} \!(t) }
\end{equation}
As $1<p<3$ will be most of the times fixed throughout this section, we are going to drop it whenever it is possible and use the short hand notation 
\[
A = A_\Lambda^{\!(p)} \qquad \hbox{and} \qquad \alpha = \alpha_\Lambda^{\!(p)} \, ,
\]
so that the above differential relationship will be simply phrased as $\dot{r}=\alpha r$.
To establish in Theorem~\ref{thm:selection_mu} the existence of a unique ancient solution $\mu$ to the Riccati equation
\[
\dot{\mu} \,  = \, \left(\frac{5-p}{p-1}\right)\alpha^2 - \left[\left(\frac{5-p}{p-1}\right)\alpha+\left(\frac{1}{p-1}\right)\right]\mu + \left(\frac{3-p}{p-1}\right)\mu^2 
\]
that converges to $1/(3-p)$ as $t \to -\infty$, it is essential to characterize the behavior of $\alpha$ as $t \to -\infty$, or, equivalently, the behavior of 
$A_\Lambda$ as $r \to 0$. To this end, 
note that from expression~\eqref{eq:up_dS}  one gets, via a simple computation,
the asymptotic expansion
\begin{equation}
\label{uexp}
u_\Lambda (r) \, = \, \left( \frac{p-1}{3-p}\right) \, r^{-\frac{3-p}{p-1}} \, (1 + {o}(1)) \, , \, \quad \text{as } r \to 0 \, .
\end{equation}
Combining this asymptotic profile with formula~\eqref{eq:alpha_r}, one immediately obtains
\[
\lim_{r \to 0} A(r) = \frac{1}{3-p} \,,
\]
and, correspondingly,
\begin{equation}
\label{lim_alpha_meno}
\lim_{t\to-\infty}\alpha(t) = \frac{1}{3-p} \,.
\end{equation}
In order to perform the precise selection of the last structural coefficient $\lambda$, it is then necessary to integrate the differential equation
\[
\dot{\lambda} \,  =  \, \left[\left(\frac{5-p}{p-1}\right)\alpha-\left(\frac{1}{p-1}\right)\right] - \left(\frac{3-p}{p-1}\right)\mu \, ,
\]
determining the constant of integration with full accuracy. This constant must be chosen so that a small-sphere-limit condition is in force for the associated monotone quantity -- see, in this connection, Theorem~\ref{thm:ssl} below. To prove that theorem, one needs a more refined description of the behaviour of $A(r)$ as $r \to 0$. This refinement is provided by the following proposition, together with its subsequent corollary where the refined asymptotics are expressed in terms of $\alpha(t)$ as $t \to -\infty$.

\begin{proposition}
\label{prop:Ar}
For $1<p\leq2$ and $\Lambda \in \R$, let $A_\Lambda^{\!(p)}$ be the function defined in~\eqref{eq:alpha_r}. Then, the following expansions hold, as $r\to 0$:
\[
A_\Lambda^{\!(p)}(r)-\frac{1}{3-p}\,=\,
\begin{dcases}
\frac{1}{5-3p} \left(\frac{p-1}{3-p}\right)\, \, \frac{\Lambda}{3} \,\, r^2 \, + \, \mathcal{O}(r^4) & \hbox{if }1<p<5/3
\\
\frac{\Lambda}{4} \,\, r^2 \, \log (1/r) \, +\, \mathcal{O}(r^2) &\hbox{if }p=5/3
\\
\frac{2(2-p)}{p-1}
\left(\frac{p-1}{3-p}\right)\int_0^{R_\Lambda} \!\!\!\! \frac{\rho^{-\frac{2(2-p)}{p-1}} }{\sqrt{1-\frac{\Lambda}{3}\rho^2}} \,d\rho\, \,\, \frac{\Lambda}{3} \,\, r^{\frac{3-p}{p-1}} \, + \, \mathcal{O}(r^2) & \hbox{if }5/3<p< 2 \, .
\end{dcases}
\]
In the case $p=2$, as $r\to 0$, it holds
\[
A_\Lambda^{\!(2)} (r) - 1\,=\, 
\begin{dcases}
- \frac{\Lambda}{3}r^2 \, , & \hbox{if } \Lambda \geq 0
\\
- \sqrt{\frac{|\Lambda|}{3}}  \,  r  \, + \, \mathcal{O}(r^2) \, ,& \hbox{if } \Lambda <0 \, .
\end{dcases}
\]
\end{proposition}

To obtain similar asymptotic expansions for the structural coefficient $\alpha$, it is sufficient to observe that 
\[
r_\Lambda^{\!(p)} (t) \, = \, \left( \frac{p-1}{3-p}\right)^{\!\frac{p-1}{3-p}} e^{\frac{t}{3-p}} (1 + {o}(1)) \, , \quad \hbox{as} \,\,\,  t \to - \infty \, ,
\]
as it easily follows from~\eqref{uexp}, coupled with the very definition of $r_\Lambda^{\!(p)}$.
\begin{corollary} 
\label{cor:at}
For $1<p\leq2$ and $\Lambda \in \R$, let $\alpha_\Lambda^{\!(p)}$ be the function defined in~\eqref{eq:alpha_t}. Then, the following expansions hold, as $t \to - \infty$:
\[
\alpha_\Lambda^{\!(p)}(t)-\frac{1}{3-p}=
\begin{dcases}
\mathcal{O}(e^{\frac{2t}{3-p}}) & \hbox{if }1<p<5/3
\\
\mathcal{O}(t e^{\frac{3t}{2}})&\hbox{if }p=5/3
\\
\mathcal{O}(e^{\frac{t}{p-1}}) & \hbox{if }5/3<p< 2 \,.
\end{dcases}
\]
In the case $p=2$, as $t\to -\infty$, it holds
\[
\alpha_\Lambda^{\!(2)} (t) - 1\,=\, 
\begin{dcases}
\mathcal{O} (e^{2t}) \, , & \hbox{if } \Lambda > 0
\\
0\, , & \hbox{if } \Lambda = 0
\\
\mathcal{O} (e^{t}) \, ,& \hbox{if } \Lambda <0 \, .
\end{dcases}
\]
\end{corollary}

\begin{remark}
\label{rem:extra_p}
It can be verified that analogous estimates may be derived for the range $2 < p < 3$.
In particular, one can prove that the expansion
\begin{equation*}
\alpha_\Lambda^{\!(p)}(t) \, = \, \frac{1}{3-p} \, + \mathcal{O} (e^{\frac{t}{p-1}}) \, , \qquad  \hbox{as} \,\, t \to -\infty \, 
\end{equation*}
holds true also in the range $2<p<3$.
Nevertheless, these asymptotic expansions are ultimately of limited relevance for the purposes of the present study, since the validity of the small-sphere limit for the quantities $\mmp_\Lambda$ can be established only for $p \leq 2$, as will be demonstrated in the proof of Theorem~\ref{thm:ssl}.
\end{remark}

\begin{proof}[Proof of Proposition~\ref{prop:Ar}]
To begin with, note that when $\Lambda = 0$ there is nothing to prove, since in this case
\[
u_0^{\!(p)}(r) \, = \, \left(\frac{p-1}{3-p}\right) r^{-\frac{3-p}{p-1}},
\]
and, consequently, $A_0^{\!(p)}(r) \equiv {1}/({3-p})$, so that all the required asymptotic expansions follow immediately. We therefore restrict our attention to the case $\Lambda > 0$, and defer to the reader the (straightforward) modifications needed to derive the corresponding estimates when $\Lambda < 0$. Fix $\Lambda > 0$. We first examine the distinguished cases $p=2$ and $p = 5/3$, for which the solution can be written in a more explicit form. For $p = 2$ we have already established that
\[
A_\Lambda^{\!(2)}(r) \, = \, 1 \, - \, \frac{\Lambda}{3} \, r^{2}.
\]
For $p = 5/3$, by means of a slightly more involved computation, one obtains
\begin{align*}
A_\Lambda^{\!(5/3)}(r)& \, = \, \frac{3}{2}r^2 \, \sqrt{1-\frac{\Lambda}{3}r^2} \, \left[\, - \, \frac{\sqrt{1-\frac{\Lambda}{3}\rho^2}}{2\rho^2}+\frac{\Lambda}{12} \, \log\left|\frac{1-\sqrt{1-\frac{\Lambda}{3}\rho^2}}{1+\sqrt{1-\frac{\Lambda}{3}\rho^2}}\right| \, \,\right]_{\rho=r}^{\rho=R_\Lambda}
\\
&=\,
\frac{3}{4}\left(1-\frac{\Lambda}{3}r^2\right)-\frac{\Lambda}{8}\, r^2 \, \sqrt{1-\frac{\Lambda}{3}r^2} \,\, \log\left|\frac{1-\sqrt{1-\frac{\Lambda}{3}r^2}}{1+\sqrt{1-\frac{\Lambda}{3}r^2}}\right|
\\
&=\,\frac{3}{4} \, + \, \frac{\Lambda}{4} \, r^2 \, \log (1/r) \, +\, \mathcal{O}(r^2).
\end{align*}
To discuss the remaining cases, it is convenient to set $q=1/(p-1)$, so that 
\[
A_\Lambda^{\!(p)}(r)\,=\,q \,  r^{2q-1}\sqrt{1-\frac{\Lambda}{3}r^2}\, \int_r^{R_\Lambda} \! \!\!\!\! \frac{ d\rho}{\rho^{2q}\sqrt{1-\frac{\Lambda}{3}\rho^2}} \,=\,\left(1-\frac{\Lambda}{3}r^2\right)\,q \,  \phi_q(r)\,,
\]
where we have defined
\[
\phi_q(r)=\frac{r^{2q-1}}{\sqrt{1-\frac{\Lambda}{3}r^2}}\,\,\int_r^{R_\Lambda} \!\!\!\!\! \frac{d\rho}{\rho^{2q}\sqrt{1-\frac{\Lambda}{3}\rho^2}}\,.
\]
Notice that, integrating by parts, we have
\begin{multline*}
\int_r^{R_\Lambda} \!\!\!\!\frac{d\rho}{\rho^{2q}\sqrt{1-\frac{\Lambda}{3}\rho^2}}\,=\,
-\frac{3}{\Lambda}\left[\frac{\sqrt{1-\frac{\Lambda}{3}\rho^2}}{\rho^{2q+1}}\right]_{\rho=r}^{\rho=R_\Lambda} \!\!\!\!\!\!\!\!
-\, (2q+1)\, \frac{3}{\Lambda} \, \int_r^{R_\Lambda} \frac{\sqrt{1-\frac{\Lambda}{3}\rho^2}}{\rho^{2q+2}} \, d\rho \,= 
\\
=\,
\frac{3}{\Lambda}\frac{\sqrt{1-\frac{\Lambda}{3}r^2}}{r^{2q+1}}
\, - \, (2q+1)\, \frac{3}{\Lambda}\int_r^{R_\Lambda}\!\!\!\!\frac{d\rho}{\rho^{2q+2}\sqrt{1-\frac{\Lambda}{3}\rho^2}}
+(2q+1)\int_r^{R_\Lambda}\!\!\!\!\frac{d\rho}{\rho^{2q}\sqrt{1-\frac{\Lambda}{3}\rho^2}}\,.
\end{multline*}
As a consequence, we conclude that $\phi_q(r)$ satisfies the following recursive relation
\begin{equation}
\label{eq:phi_recursive}
(2q+1) \phi_{q+1}(r) \, = \, 
2q  \,\phi_q(r) \, \frac{\Lambda}{3} \, r^2 \, +\, 1\,.
\end{equation}
Thus, we compute
\begin{align}
\notag
A_\Lambda^{\!(p)}(r)\,&=\,\left(1-\frac{\Lambda}{3}r^2\right) \, q \, \phi_q(r)
\\
\notag
&=  \left(\frac{q}{2q-1}\right) \left(1-\frac{\Lambda}{3}r^2\right)\left[ \, {(2q-2)} \phi_{q-1}(r) \, \frac{\Lambda}{3}r^2 + {1} \, \right]
\\
\label{eq:Ar_expansion}
&=\,\frac{1}{3-p} \left\{ \, 1 \,+ \,\left[\frac{2(2-p)}{p-1}\phi_{q-1}(r)-1\right]  \left(\frac{\Lambda}{3}r^2 \right) -\frac{2(2-p)}{p-1}\phi_{q-1}(r)\left(\frac{\Lambda}{3}r^2\right)^{\!\!2}
\!\right\}.
\end{align}
Recall that $\phi_{q-1}$ is defined as
\[
\phi_{q-1}(r)
=\frac{r^{2q-3}}{\sqrt{1-\frac{\Lambda}{3} r^2}}
\, \int_r^{R_\Lambda} \!\!\!\!\frac{d\rho}{\rho^{2q-2}\sqrt{1-\frac{\Lambda}{3} \rho^2}}.
\]
For $1<p<5/3$, we have that $3/2 <q<+\infty$, and we set
\[
a_{q-1}(r):=\int_r^{R_\Lambda}\!\!\!\!\frac{d\rho}{\rho^{2q-2}\sqrt{1-\frac{\Lambda}{3} \rho^2}},
\qquad
b_{q-1}(r):=\frac{\sqrt{1-\frac{\Lambda}{3} r^2}}{r^{2q-3}} \, ,
\]
so that $\phi_{q-1}(r) = {a_{q-1}(r)} /{b_{q-1}(r)}$. As \(r\to 0^+\),  l'H\^{o}pital's rule yields
\[
\lim_{r\to0^+}\phi_{q-1}(r)=\lim_{r\to0^+}\frac{a_{q-1}(r)}{_{q-1}(r)}
=\lim_{r\to0^+}\frac{a_{q-1}'(r)}{b_{q-1}'(r)}.
\]
By the Fundamental Theorem of Calculus,
\[
a_{q-1}'(r)=-\frac{1}{r^{2q-2}\sqrt{1-\frac{\Lambda}{3} r^2}}.
\]
Moreover, by a direct computation we get
\[
\begin{aligned}
b_{q-1}'(r)
&=-\frac{\frac{\Lambda}{3} r}{\sqrt{1-\frac{\Lambda}{3} r^2}}\,r^{-(2q-3)}
-(2q-3)\sqrt{1-\frac{\Lambda}{3} r^2}\,r^{-(2q-2)}.
\end{aligned}
\]
Hence, after some obvious algebraic simplifications, we get
\[
\lim_{r\to0^+}\phi_{q-1}(r)
=\lim_{r\to0^+}\frac{1}{(2q-3)+2(2-q)\frac{\Lambda}{3} r^2}
=\frac{1}{2q-3} = \frac{p-1}{5-3p}.
\]
As a consequence, for $1<p<5/3$ we find
\[
A_\Lambda^{\!(p)}(r)\,=\,\frac{1}{3-p}+
\frac{1}{5-3p} \left(\frac{p-1}{3-p}\right)
\,\frac{\Lambda}{3}r^2+O(r^4)
\]
as $r\to 0$. 
If instead $5/3<p<2$, then $1<q<3/2$ and we have that the integral appearing in the definition of $\phi_{q-1}(r)$ is finite, whereas $r^{2q-3}$ goes to infinity as $r\to 0$. Hence, from~\eqref{eq:Ar_expansion} we deduce
\[
A_\Lambda^{\!(p)}(r) \, = \, \frac{1}{3-p}+\frac{2(2-p)}{3-p} \, \frac{\Lambda}{3} \, \phi_{q-1}(r) \, r^2+O(r^2)
\]
This concludes the proof.    
\end{proof}

\subsection{
A selection principle for the structural coefficients \texorpdfstring{$\mu$ and $\lambda$}{ μ and λ}: ancient solutions and the small sphere limit.}
\label{subsec:selection}

In this subsection, building upon the prescribed choice of the structural coefficient $\alpha$ and upon the analysis of its behavior as $t \to -\infty$, we investigate the ancient solutions of system~\eqref{structural}, with the aim of selecting the two solutions $\mu$ and $\lambda$ that exhibit the highest degree of compatibility with the infinitesimally Euclidean structure of the underlying manifold. Recall that, as $t \to -\infty$, we are, in a heuristic sense, approaching the pole of the $p$-Green's function. Consequently, it is natural to expect that all the relevant analytic and geometric quantities display an asymptotic behavior analogous to that of the corresponding Euclidean model. In accordance with formulas~\eqref{eq:lomo}, we therefore single out the unique pair of solutions $\mu$ and $\lambda$ characterized by the asymptotic conditions
\begin{equation}
\label{eq:euclidean_as}
 \mu(t) = \frac{1}{3-p} \, + \, \mathcal{O}(e^{t/2}) \quad \text{and} \quad  
 \lambda(t) = \frac{t}{3-p}
 \, + \,  \log \, \left[ \frac{\left(\frac{p-1}{3-p}\right)^{\!\!\frac{p-1}{3-p}} }{8\pi(3-p)} \right]
 \, + \,  \mathcal{O}(e^{t/2})\,,
\end{equation}
as $t \to -\infty$. In particular, the precise choice of the constant term in the asymptotic expansion of $\lambda$ will be instrumental in establishing the small sphere limit Theorem~\ref{thm:ssl} for the quantities $\mmp_\Lambda$, a property that is typically regarded as a fundamental requirement for mass-type invariants. As the structural system~\eqref{structural} is semi-decoupled, we start with the analysis of the independent equation for $\mu$.
\begin{theorem}
\label{thm:selection_mu}
For every $1<p<3$ and  $\Lambda\in\R$, let $\alpha = \alpha_{\Lambda}^{\!(p)}$ be the function defined in~\eqref{eq:alpha_t}.  Then, there exists a unique solution $\mu$ to the equation 
\begin{equation}
\label{eq:ODE_mu}
\dot{\mu} \,  = \, \left(\frac{5-p}{p-1}\right)\alpha^2 - \left[\left(\frac{5-p}{p-1}\right)\alpha+\left(\frac{1}{p-1}\right)\right]\mu + \left(\frac{3-p}{p-1}\right)\mu^2
\end{equation}
such that $\lim_{t\to -\infty} \mu(t) = 1/(3-p)$.
\end{theorem}

\begin{remark}
Observe that, {\em a priori}, the ancient solution $\mu$ obtained in the above theorem is only defined up to some maximal time $T\in\R$. Later in this section, we will show that $\mu$ is in fact defined for all times.
\end{remark} 

\begin{proof}
The right hand side of the differential equation \eqref{eq:ODE_mu} induces a second-order equation for $\mu$, whose zeros determine the regions where the solutions of the differential equation either increase or decrease. More precisely, $\dot{\mu}(t)=0$ occurs only at times $t$ when $\mu(t)=\mu_{\pm}(t)$, where
\begin{equation}
\label{eq:mu_plus_minus}
\mu_{\pm}(t)\,=\,\frac{ 1+  (5-p)\alpha \,\pm\sqrt{1+2(5-p)\alpha    -(7-3p)(5-p)\alpha^2}}{2(3-p)}.
\end{equation}
Note that, from \eqref{lim_alpha_meno}, the function under the square root approaches the positive constant $4/(3-p)^2$ as $t\to -\infty$. In particular, for all sufficiently small $t$, the values $\mu_\pm(t)$ are real. Hence,
any solution $\mu$ to the ODE is decreasing when $\mu_-\leq\mu\leq\mu_+$, and increasing otherwise. Furthermore, again from 
\eqref{lim_alpha_meno}
we conclude that
\[
\lim_{t\to-\infty}\mu_-(t)\,=\,\frac{1}{3-p}.
\]
It is easily seen that there exists $\tau\in\R$ such that we can find an increasing function $\sigma^{(+)}(t)$ and a decreasing function $\sigma^{(-)}(t)$ such that $\lim_{t\to-\infty}\sigma^{(\pm)}(t)=1/(3-p)$, $\lim_{t\to-\infty}\dot\sigma^{(\pm)}(t)=0$ and $\sigma^{(-)}(t)<\mu_-(t)<\sigma^{(+)}(t)<\mu_+(t)$ for all $t\in(-\infty,\tau)$. 
In particular, for all $t<\tau$, we will have $\dot \mu(t)>0$ whenever $\mu(t)=\sigma^{(-)}(t)$ and $\dot \mu(t)< 0$ whenever $\mu(t)=\sigma^{(+)}(t)$. As a consequence of this, we easily conclude that, if $\sigma^{(-)}(\bar t)\leq \mu(\bar t)\leq \sigma^{(+)}(\bar t)$ for some $\bar t<\tau$, then $\sigma^{(-)}(t)\leq \mu(t)\leq \sigma^{(+)}(t)$ for all $t\in[\bar t,\tau]$.

To establish the existence, we now consideer a decreasing sequence of values $\{t_n\}_{n\in\N}$ with $t_1<\tau$ and $t_n\to-\infty$ as $n\to+\infty$, and consider functions $\mu^{(\pm)}_n(t)$ satisfying~\eqref{eq:ODE_mu} and $\mu_n^{(\pm)}(t_n)=\sigma^{(\pm)}(t_n)$ (clearly these solutions exist because of Cauchy's Theorem). From the above discussion we conclude that $\mu_n^{(\pm)}(\tau)\in[\sigma^{(-)}(\tau),\sigma^{(+)}(\tau)]$. 
Using the fact that different solutions to~\eqref{eq:ODE_mu} never intersect (this is again an easy consequence of Cauchy's Theorem), it is also easy to conclude that $\mu_n^{(+)}(\tau)$ is a decreasing sequence, whereas $\mu_n^{(-)}(\tau)$ is an increasing sequence, with $\mu_n^{(+)}(\tau)>\mu_n^{(-)}(\tau)$. As a consequence, we conclude that $\mu_n^{(+)}(\tau)\to \ell\in(\sigma^{(-)}(\tau),\sigma^{(+)}(\tau))$ and $\mu_n^{(-)}(\tau)<\ell<\mu_n^{(+)}(\tau)$ for all $n$.
We now claim that the solution $\bar \mu$ to~\eqref{eq:ODE_mu} with $\bar \mu(\tau)=\ell$ is such that $\lim_{t\to-\infty}\bar \mu(t)=1/(3-p)$.
In fact, this is immediately seen by observing that the solution $\bar \mu$ is bounded from above by the $\mu_n^{(+)}$ and from below by the $\mu_n^{(-)}$. In particular, this implies that $\sigma^{(-)}(t)<\bar\mu(t)<\sigma^{(+)}(t)$ for all $t\in(-\infty,\tau]$, hence $\bar \mu$ must have the same limit at $-\infty$ as $\sigma^{(\pm)}(t)$.

It remains to prove that such a solution is indeed unique. To this end, suppose by contradiction that there exist $\mu_1,\mu_2$ solutions to~\eqref{eq:ODE_mu} with $\lim_{t\to-\infty}\mu_1(t)=\lim_{t\to-\infty}\mu_2(t)=1/(3-p)$, defined in some interval $(-\infty,\tau)$. Since different solutions do not intersect, without loss of generality let us assume that $\mu_2>\mu_1$. Then the difference $\mu_2-\mu_1$ satisfies
\[
\frac{d}{dt}(\mu_2-\mu_1)\,=\,\left[-\frac{5-p}{p-1}\alpha-\frac{1}{p-1}+\frac{3-p}{p-1}(\mu_2+\mu_1)\right](\mu_2-\mu_1)\,.
\]
Now, since, from the limits of $\mu_1$ and $\mu_2$ as $t\to -\infty$ and from \eqref{lim_alpha_meno}, the term in square brackets approaches $-2/((p-1)(3-p))\leq-2$ at $-\infty$, it follows that $\mu_2-\mu_1$ satisfies the differential inequality
\[
\frac{d}{dt}(\mu_2-\mu_1)
\,\leq\,
-(\mu_2-\mu_1),
\]
for all sufficiently small $t$. Hence, by comparison, we deduce that
\[
\mu_2(t)-\mu_1(t)\,\geq\,\bar qe^{-(t-\bar t\,)},
\]
for some sufficiently small $\bar t$, some positive $\bar q$, and for all $t\leq \bar t$.
This, together with the fact that $\lim_{t\to -\infty} (\mu_2(t)-\mu_1(t)) = 0$, yields the desired contradiction.
\end{proof}

With the next lemma, we identify the next-order decay rate of \(\mu\) at \(-\infty\).

\begin{lemma}
\label{lem:mudecay}
For $1<p<3$ and $\Lambda \in \R$, let $\alpha=\alpha_\Lambda^{\!(p)}$ be the function defined in~\eqref{eq:alpha_t} and let $\mu=\mu_\Lambda^{\!(p)}$ be the unique solution of~\eqref{eq:ODE_mu} that satisfies $\lim_{t \to - \infty} \mu (t) = 1/(3-p)$. Then, $\mu$ satisfies the following asymptotic expansion
\begin{equation}
\mu = \frac{1}{3-p} \, + \, \mathcal{O}(e^{t/2}) \, , \qquad \hbox{as} \,\,\, t \to - \infty
\end{equation}
\end{lemma}

\begin{proof}
First of all, we recall that, by Theorem~\ref{thm:selection_mu} and Corollary~\ref{cor:at} we have
\[
\alpha(t)=\frac{1}{3-p}+\mathcal{O}(e^{t/2})
\qquad\text{as }t\to-\infty,
\qquad\text{and}\qquad
\mu(t)\longrightarrow \frac{1}{3-p}
\qquad\text{as }t\to-\infty.
\]
More precisely, the decay estimate for $\alpha$ has so far been proved only up to to $p=2$, in a sharper form. Nonetheless, as already pointed out in Remark~\ref{rem:extra_p} (with the details left to the interested reader), this result can be extended to the whole interval $1<p<3$, provided one is content with a slightly weaker estimate.
Equation~\eqref{eq:ODE_mu} implies the differential identity
\begin{equation}\label{eq:mu_minus_equation}
\frac{d}{dt}\left(\mu-\frac{1}{3-p}\right)
=
\left(\frac{5-p}{p-1}\right)\alpha\left(\alpha-\frac{1}{3-p}\right)
+
\left(\mu-\frac{1}{3-p}\right)\,G(t),
\end{equation}
where
\[
G(t):=
-\left[\left(\frac{5-p}{p-1}\right)\alpha+\left(\frac{1}{p-1}\right)\right]
+\left(\frac{3-p}{p-1}\right)\left(\mu+\frac{1}{3-p}\right).
\]
Observe that 
\begin{align*}
\lim_{t\to-\infty}G(t)
&=
-\frac{2}{(p-1)(3-p)}<0.
\end{align*}
Hence, there exists \(\tau\in\mathbb{R}\) such that, for all \(t\le\tau\),
\begin{equation}\label{eq:G_negative}
G(t)\le -\frac{1}{(p-1)(3-p)}.
\end{equation}
Also, observe that, for all $t\le\tau$,
\begin{equation}\label{eq:alpha_source}
\left|\alpha(t)\left(\alpha(t)-\frac{1}{3-p}\right)\right|
\le C e^{t/2} \, ,
\end{equation}
for some constant \(C>0\).

Now, let $t_0<t \leq \tau$. Multiplying \eqref{eq:mu_minus_equation} by the integrating factor
\(\exp \bigl(-\int_{t_0}^t G(\xi)\,d\xi\bigr)\) and integrating from \(t_0\) to \(t\), we obtain
\begin{align*}
\mu(t)-\frac{1}{3-p}
&=
\left(\mu(t_0)-\frac{1}{3-p}\right)\exp\!\left(\int_{t_0}^t G(\xi)\,d\xi\right)\\
&\quad
+\left(\frac{5-p}{p-1}\right)\int_{t_0}^t
\exp\!\left(\int_{s}^t G(\xi)\,d\xi\right)\alpha(s)\left(\alpha(s)-\frac{1}{3-p}\right)\,ds.
\end{align*}
Taking absolute values and using \eqref{eq:G_negative}, for all \(t_0<s\le t\le\tau\),
\[
\exp\!\left(\int_{s}^t G(\xi)\,d\xi\right) \, \le \,
\exp\!\left(-\frac{t-s}{(p-1)(3-p)}\right).
\]
Hence, by \eqref{eq:alpha_source},
\begin{align*}
\left|\mu(t)-\frac{1}{3-p} \right|
& \, \le \, \left|\mu(t_0)-\frac{1}{3-p}\right| \, \exp\!\left(-\frac{t-t_0}{(p-1)(3-p)}\right)\\
&\quad
+\left(\frac{5-p}{p-1}\right) \, C
\int_{t_0}^t
\exp\!\left(-\frac{t-s}{(p-1)(3-p)} + \frac{s}{2}\right) \,ds.
\end{align*}
Now let \(t_0\to-\infty\) and observe that the first term on the right hand side tends to \(0\), leading to
\begin{align*}
\left|\mu(t)-\frac{1}{3-p} \right|
& \, \le \, \left(\frac{5-p}{p-1}\right) \, C
\int_{-\infty}^t
\exp\!\left(-\frac{t-s}{(p-1)(3-p)} + \frac{s}{2}\right)\,ds\\
& \, =\,  C\left(\frac{2(5-p)(3-p)}{2+ (p-1)(3-p)}\right)\,e^{t/2}.
\end{align*}
Therefore,
\[
\mu(t)=\frac{1}{3-p}+\mathcal{O}(e^{t/2})
\qquad\text{as }t\to-\infty,
\]
and the proof is coompleted.
\end{proof}

Once $\alpha$ and $\mu$ are selected according to Subsection~\ref{sub:selection_alpha} and Theorem~\ref{thm:selection_mu}, we can integrate the second equation in~\eqref{structural} and obtain the last structural coefficient $\lambda$, up to an additive constant. This is the content of the following theorem. 

\begin{theorem}
\label{thm:selection_lambda}
For $1<p<3$ and $\Lambda \in \R$, let $\alpha=\alpha_\Lambda^{\!(p)}$ be the function defined in~\eqref{eq:alpha_t} and let $\mu=\mu_\Lambda^{\!(p)}$ be the unique solution of~\eqref{eq:ODE_mu} that satisfies $\lim_{t \to - \infty} \mu (t) = 1/(3-p)$, as selected by Theorem~\ref{thm:selection_mu}. Then, there exists a solution $\lambda$ to the equation 
\begin{equation}
\label{eq:lam}
\dot{\lambda} \,  =  \, \left[\left(\frac{5-p}{p-1}\right)\alpha-\left(\frac{1}{p-1}\right)\right] - \left(\frac{3-p}{p-1}\right)\mu \
\end{equation}
which is defined on the same interval where $\mu$ is defined and it is determined up to an additive constant. Moreover, $t \mapsto \lambda(t)$ satisfies
\[
\lambda(t) \, = \, \frac{t}{3-p} + \kappa  + o(1)\, , \qquad {\hbox{as $t \to - \infty$}}
\]
for some constant $\kappa$. For the choice of $\lambda$ to be consistent with the validity of a small sphere limit, it is useful to set
\begin{equation}
\label{eq:kappa}
\kappa = \left(\frac{p-1}{3-p}\right)\log\left(\frac{p-1}{3-p}\right)-\log\left[8\pi(3-p)\right]\,.
\end{equation}
\end{theorem}

\begin{proof} 
A simple integration provides the existence of a solution $\lambda$ to equation~\eqref{eq:lam}. Such a solution is clearly unique up to an additive constant. 
The asymptotic behavior at $-\infty$, can easily be deduced from the decay estimates 
\[
\alpha(t)\,= \, \frac{1}{3-p}+\mathcal{O}(e^{t/2})
 \, = \,  \mu(t) \, ,\qquad\text{as }t\to-\infty,
\]
that have been established in Corollary~\ref{cor:at} and Lemma~\ref{lem:mudecay}.
\end{proof}
At this point, we have fixed the structural coefficients $\alpha$, $\mu$, and $\lambda$ as the solution of system~\eqref{structural}, ensuring that the conditions~\eqref{eq:euclidean_as} -- whose geometric meaning has already been discussed -- are satisfied. We now show that, after a suitable choice of the integration constant (see~\eqref{eq:kappa}) in the definition of $\lambda$, the monotone functionals $t \mapsto \mmp_\Lambda(x,t)$ introduced in~\eqref{mp_gen} satisfy the {\em small sphere limit condition}. (Here, for obvious reasons, it is convenient to adopt the full notation, where the dependence on the pole $x$ is also made explicit). 
Referring to the foundational work~\cite{Bartnik_quasilocal} and the recent survey~\cite{McC_2024} for a detailed presentation of Bartnik’s quasi-local mass and the axiomatic criteria one expects from a quasi-local mass, we emphasize here only that the small-sphere limit condition constitutes a fundamental requirement for any proposed quasi-local mass functional.
Interpreting $\mmp_\Lambda(x,t)$ as the mass contained in the region $\Omega_t = \{w < t\}$, where $w$ solves~\eqref{wdistr}, the limit
\[
\lim_{t \to -\infty} \frac{\mmp_\Lambda(x,t)}{|\Omega_t|}
\]
must recover the mass density at the pole $x \in M$, the distinguished reference point for our constructions. For time-symmetric initial data, the Einstein Constraint Equations identify this mass density at $x$ as
\[
\frac{\RRR (x) - 2 \Lambda}{16 \pi} \ .
\]
The terminology {\em small sphere} is justified in this context by the fact that, as we approach the pole, the level sets of $w$ converge to small geodesic spheres, so that the regions $\Omega_t$ truly behave like shrinking geodesic balls around $x$.

\begin{theorem}
[Small Sphere Limit for  $\mmp_\Lambda$]
\label{thm:ssl}
Let $\Lambda \in \R$ and let $(M,g)$ be a smooth $3$-dimensional Riemannian manifold, whose scalar curvature satisfies 
\[
\RRR \, \geq \, 2 \Lambda \, . 
\]
Let $1<p<2$, and, for a fixed pole $x\in M\setminus\pa M$, let $u$ be a solution to
\[
\Delta_p u\,=\, -4\pi \delta_x \,
\]
in the sense of distributions. Set, as usual, $w=-(p-1)\log u$ and let $t \mapsto \mmp_\Lambda (x,t)$ be the function defined in~\eqref{mp_gen}, where the structural coefficients $\alpha$, $\mu$ and $\lambda$ are determined by formula~\eqref{eq:alpha_t}, Theorem~\ref{thm:selection_mu} and Theorem~\ref{thm:selection_lambda}, respectively. Then,
\[
\lim_{t\to-\infty}\frac{\mmp_\Lambda(x,t)}{|\Omega_t|}\,=\,
\frac{\RRR(x)-2\Lambda}{16 \pi}\,,
\]
where $\Omega_t=\{w < t\}$  are the sub-level sets of $w$ containing the pole $x$.
\end{theorem}

\begin{proof}
Let $y=(y^1,y^2,y^3)$ be normal coordinates centered at $x$.
Specializing \cite[Theorem 1.2]{Ago_Bor} to dimension $n=3$ and to the range $1<p<2$, we have 
that a solution to $\Delta_p u\,=\, -4\pi \delta_x$ has the following expansion near the pole, which distinguishes the behavior across three different ranges of $p$:
\begin{equation}
\label{eq:u_expansion}
u-\left(\frac{p-1}{3-p}\right)|y|^{-\frac{3-p}{p-1}}\,=\,
\begin{cases}
\mathcal{O}_2\big(|y|^{-\frac{5-3p}{p-1}}\big) & \quad\hbox{if } 1<p<\frac53,
\phantom{\frac1{\frac1{\frac12}}}\\
\mathcal{O}_2\big(\log(1/|y|)\big) & \quad\hbox{if } p=\frac53,
\phantom{\frac1{\frac1{\frac12}}}
\\
\mathcal{O}_2(1) & 
\quad\hbox{if }\frac53< p<2,
\phantom{\frac1{\frac1{\frac12}}}
\end{cases}
\qquad\mbox{as }\ y\to0.
\end{equation}
Here, the notation $e(y)=\mathcal{O}_2(f(|y|))$ means that $e(y)=\mathcal{O}(f(|y|))$, that $\pa_\alpha e(y)=\mathcal{O}(f(|y|)/|y|)$, and that $\pa_\alpha\pa_\beta e(y)=\mathcal{O}(f(|y|)/|y|^2)$, as $|y|\to 0$.

\smallskip
For every $1<p<2$, the previous expansion can be reformulated in terms of the function $w=-(p-1)\log u$. Moreover, it can be slightly weakened to yield a single, common regime implied by each of the three distinct regimes described above:
\begin{align}
w\,&=\,\log\left[\left(\frac{3-p}{p-1}\right)^{\!p-1} \!\! |y|^{3-p}\right]\, + \,o_2(|y|)\,,
\label{esp_w_p_sotto}
\\
|\na w|\,&=\,\frac{3-p}{|y|}\,+\,o_1(1)\,.
\label{esp_na_w_p_sotto}
\end{align}
The notation $e(y)=o_1\big(|y|^{c}\big)$ means that $e(y)=o\big(|y|^c\big)$ and that $\pa_\alpha e(y)=o\big(|y|^{c-1}\big)$. The notation $e(y)=o_2\big(|y|^{c}\big)$ means that $e(y)=o\big(|y|^c\big)$ and that $\pa_\alpha e(y)=o_1\big(|y|^{c-1}\big)$ for $\alpha=1,2,3$.
Inverting relationship \eqref{esp_w_p_sotto} on the level set $\Sigma_t=\{w=t\}$, we get
\begin{equation*}
|y| \,= \, \left(\frac{p-1}{3-p}\right)^{\!\!\frac{p-1}{3-p}} \!\!e^{\frac{t}{3-p}}\,+\, o\big(e^{\frac{2t}{3-p}}\big),
\qquad\mbox{on }\ \Sigma_t,
\end{equation*}
as $t\to-\infty$, and in turn, using \eqref{esp_na_w_p_sotto}, 
that
\begin{equation}
\label{eq:naw_near_pole}
|\na w|\,=\,(3-p)\left(\frac{3-p}{p-1}\right)^{\!\!\frac{p-1}{3-p}} \!\! e^{-\frac{t}{3-p}}+o(1),
\quad
\frac{\pa|\na w|}{\pa w}
\,=\,-\left(\frac{3-p}{p-1}\right)^{\!\!\frac{p-1}{3-p}} \!\! e^{-\frac{t}{3-p}}+o(1),
\end{equation}
on $\Sigma_t$.
Note that, by applying the Divergence Theorem to the region bounded by $\Sigma_t$ and $\{|y|=c\}$, and then letting $c\to0$, while using the fact that
\[
e^{-w}=\,
\left(\frac{p-1}{3-p}\right)^{p-1}|y|^{-(3-p)}
\Big(1+o(|y|)\Big),
\quad\,
|\na w|^{p-2}=\,
(3-p)^{p-2}
|y|^{2-p}
\Big(1+o(|y|)\bigg),
\]
one can compute the constant value
\eqref{int_costante}, namely:
\begin{equation}
\label{int_costante_3}
e^{-t}\int_{\Sigma_t} |\na w|^{p-1}\, d \sigma\,=\,
4\pi(p-1)^{p-1},
\qquad\mbox{for all\ }t\in\R.
\end{equation}
This formula, coupled with \eqref{eq:naw_near_pole}, yields the following expansion for the area $|\Sigma_t|$: 
\begin{equation}
\label{eq:stima_livelli}
|\Sigma_t|\,=\,
4\pi
\left(\frac{p-1}{3-p}\right)^{2(\frac{p-1}{3-p})}
e^{\frac{2t}{3-p}}
\,+\,
o\big(e^{\frac{3t}{3-p}}\big).
\end{equation}
Concerning the mean curvature of the level sets, recalling the formula 
\begin{equation}
\label{formula_curv_med}
\HHH\,=\,|\na w|-(p-1)\frac{\pa|\na w|}{\pa w}, 
\end{equation}
and using \eqref{eq:naw_near_pole}, we obtain 
\begin{equation}
\label{expacca}
\HHH\,=\,
2\left(\frac{3-p}{p-1}\right)^{(\frac{p-1}{3-p})}
e^{-\frac t{3-p}}
\,+\,
o(1).
\end{equation}
Now, before proceeding with the computation of the small sphere limit, let us recall the expression for $\mmp_\Lambda$: 
\begin{equation}
\mmp_\Lambda (x, t) \, = \int_{-\infty}^t e^{\lambda(\tau)} \cdot \big(4\pi - \Lambda|\Sigma_\tau|\big)  \, d\tau
\, - \,\,  e^{\lambda(t)}  \cdot \int_{\Sigma_t} |\nabla w| \,  \big(  \HHH - \mu(t)|\nabla w|  \big) \, d\sigma \,.
\end{equation}
Note that estimates
\eqref{eq:naw_near_pole}, 
\eqref{eq:stima_livelli}, and  \eqref{expacca}, imply in particular the bounds
\[
|\na w|\,=\,\mathcal O\big(e^{-\frac t{3-p}}\big),
\qquad
|\Sigma_t|\,=\,\mathcal O\big(e^{\frac{2t}{3-p}}\big),
\qquad
\HHH\,=\,
\mathcal O\big(e^{-\frac t{3-p}}\big),
\]
respectively.
This information,
together with
\[
\mu(t)\,=\,\mathcal O(1),
\qquad e^{\lambda(t)}\,=\,\mathcal O\big(e^{\frac t{3-p}}\big),
\]
(see Theorem~\ref{thm:selection_lambda}),
is already sufficient to conclude that
$\mmp_\Lambda(x,t)=\mathcal O\big(e^{\frac t{3-p}}\big)$, as $t\to-\infty$. 
Hence, we have that
\begin{equation}
\label{massa_nulla_polo}
\lim_{t\to-\infty}\mmp_\Lambda(x,t)\,=\, 0\,.
\end{equation}
We now proceed with the computation of the small sphere limit, first noticing that an application of De L'H\^opital rule gives
\begin{align*}
&\lim_{t\to-\infty}\frac{\mmp_\Lambda(x,t)}{|\Omega_t|}\,=\,\lim_{t\to-\infty}\frac{\frac{d}{dt}\mmp_\Lambda(x,t)}{\frac{d}{dt}|\Omega_t|}
\\
&=\,\lim_{t\to-\infty}\frac{e^{\lambda(t)}{\displaystyle\int_{\Sigma_t}}\left[  2 \left|\frac{\na^{\Sigma_t}|\na w|}{|\na w|}\right|^2\!\!+\,|\mathring\hhh|^2\,+\,(\RRR-2\Lambda)
\,+\, 2 \,\Big(\frac{5-p}{p-1}\Big) \,\big(({\HHH}/2)- \alpha |\na w|\big)^2 \right] d\sigma}{2 \int_{\Sigma_t} ({1}/{|\na w|}) \, d\sigma}\,,
\end{align*}
where in the second equality we have employed formula~\eqref{dermp}, together with the Coarea Formula to rewrite the denominator. 
In order to compute the above limit, we first note that 
\begin{equation}
\label{sphere_limit_1}
\lim_{t \to -\infty}\frac{e^{\lambda(t) }|\Sigma_t|}{\int_{\Sigma_t}( 1/|\na w| ) \, d\sigma} \, =  \, \frac{1}{8\pi} \,.
\end{equation}
To compute this limit, we relied on~\eqref{eq:naw_near_pole} and on the precise choice of the integration constant~\eqref{eq:kappa} for the structural coefficient $\lambda$ in  Theorem~\ref{thm:selection_lambda}. 
In turn, \eqref{sphere_limit_1} yields 
\[
\lim_{t \to -\infty}\frac{e^{\lambda(t) }    \int_{\Sigma_t}( \RRR - 2 \Lambda ) \, d\sigma
}{2 \,\, \int_{\Sigma_t}( 1/|\na w| ) \, d\sigma} \, =  \, \frac{\RRR (x) - 2 \Lambda}{16\pi}.
\]
To complete the proof, it remains to verify that all the remaining terms in De L'H\^opital quotient are negligible as $t \to -\infty$. 
First, note that Corollary~\ref{cor:at} yields $\alpha(t)=1/(3-p)+o\big(e^{t/(3-p)}\big)$, throughout the range $1<p<2$. 
This fact, together with 
expansions~\eqref{eq:naw_near_pole} and~\eqref{expacca}, gives
\[
\frac{\HHH}{2}-\alpha|\na w|\,=\,
o(1),
\qquad\mbox{as }\
t\to-\infty.
\]
Hence, the last term in De L'Hôpital quotient is infinitesimal as \(t \to -\infty\).
We now turn our attention to the term
$|\mathring{\hhh}|$. To analyze it, we follow the approach developed in~\cite{AMO}. Namely, we fix parameters $0 < \delta < \mathrm{inj}(x)/2$ and $0<\varepsilon < 1/\sqrt{3}$, and we consider the sets
\[
U_\alpha\,=\,\Big\{p\in B(x,\delta)\setminus\{x\}\,:\,|y^\alpha|>\ep|y| \hbox{ and }|\na w|(p)\neq 0\Big\}\,,
\qquad\alpha=1,2,3.
\]
We then work in $U_1$ (a similar procedure works for $U_2$ and $U_3$), using the frame $\{\na w/|\na w|,X_2,X_3\}$, where $X_2=-\pa_2 w\pa_1+\pa_1 w\pa_2$, $X_3=-\pa_3 w\pa_1+\pa_1 w\pa_3$: 
it is readily seen that $X_2,X_3$ are tangent to the level sets $\Sigma_t$. 
Expansion \eqref{esp_w_p_sotto}
yields, in particular, that
\begin{equation}
\label{espa_deriva}
\frac{\pa w}{\pa y^\alpha}\,=\,(3-p)\frac{y^\alpha}{|y|^2}+o(1)\,,\qquad  \frac{\pa^2 w}{\pa y^\alpha\pa y^\beta}
\,=\,\frac{(3-p)}{|y|^2}\left[\delta_{\alpha\beta}-2\frac{y^\alpha y^\beta}{|y|^2}\right]+o(|y|^{-1}),
\end{equation}
$\alpha=1,2,3$.
It then follows from the expansion of the first derivatives that
\begin{equation}
\label{gijy}
g^\Sigma_{ij}\,=\,g(X_i,X_j)\,=\,
(3-p)^2\frac{\Big[(y^1)^2\delta_{ij}+y^iy^j\Big]}{|y|^4}+o(|y|^{-1}),
\qquad i,j\in\{2,3\}.
\end{equation}
Moreover, using the fact that $|y^1|> \varepsilon|y|$ in $U_1$, it is easy to see that 
the determinant of the $2{\times}2$ matrix whose entries are given by $g^\Sigma_{ij}$ is bounded below by $C|y|^{-4}$, where $C$ is a positive constant (depending on $\ep$). 
This fact implies in particular that
$g_\Sigma^{ij} = \mathcal O(|y|^2)$.
Expansions \eqref{espa_deriva}, together with the fact that 
$g_{\alpha\beta}=\delta_{\alpha\beta}+\mathcal O_1\big(|y|^2\big)$, imply
\[
\na\na w(X_i,X_j)\,=\,(3-p)^3\frac{\Big[(y^1)^2\delta_{ij}+y^iy^j\Big]}{|y|^6}+o(|y|^{-3}),
\qquad i,j\in\{2,3\}.
\]
Consequently, by \eqref{esp_na_w_p_sotto}, we get
\begin{equation}
\label{hijy}
\hhh_{ij}=\frac{\na\na w(X_i,X_j)}{|\na w|}
\,=\,(3-p)^2\frac{\Big[(y^1)^2\delta_{ij}+y^iy^j\Big]}{|y|^5}+o(|y|^{-2}),
\qquad i,j\in\{2,3\}.
\end{equation}
Moreover, formulas \eqref{esp_na_w_p_sotto} and \eqref{formula_curv_med}, and the fact that
$\pa|\na w|/\pa w=-|y|^{-1}+o(1)$,
yield
\begin{equation}
\label{curv_y}
\HHH\,=\,\frac2{|y|}+o(1).
\end{equation}
All in all, 
expansions \eqref{gijy}, \eqref{hijy}, and \eqref{curv_y} give that
\[
\mathring{\hhh}_{ij}\,=\,\hhh_{ij}-\frac{\HHH}{2}g^\Sigma_{ij}\,=\,o(|y|^{-2}),
\qquad i,j\in\{2,3\}.
\]
and finally, using the bound $g_\Sigma^{ij} = \mathcal O(|y|^2)$, that 
$|\mathring{\hhh}|^2=o(1)$.
This shows that the second term in De L'H\^opital quotient is infinitesimal as \(t \to -\infty\). 

Considering now the term $\big|\na^{\Sigma_t}|\na w|\big|/|\na w|$ 
in De L'H\^opital quotient, let us first compute the  $\alpha$-component, $\alpha=1,2,3$, of the tangential vector field $\na^{\Sigma_t}|\na w|$, namely
\begin{align*}
\hspace{-.3cm}
\na_\alpha|\na w|-\langle\na|\na w|,\nu\rangle \nu_\alpha
&=
\frac{\pa|\na w|}{\pa y^\alpha}-\frac{\pa|\na w|}{\pa w}g^{\alpha\beta}\frac{\pa w}{\pa y^\beta} 
\\
&=
-(3-p)\frac{y^\alpha}{|y|^3}+o(|y|^{-1})
+\left(\frac{\HHH-|\na w|}{p-1}\right)\!\!\Big(\delta_{\alpha\beta}\,+\,\mathcal O  \big(|y|^2\big)\!\Big)
\!\!\left[(3-p)\frac{y^\beta}{|y|^2}+o(1)\right]\\
&=\,o(|y|^{-1}). 
\end{align*}
Note that
in the third equality
we have used \eqref{formula_curv_med},
and in the forth equality \eqref{esp_na_w_p_sotto} and \eqref{curv_y}. 
As a consequence,
$\big|\na^{\Sigma_t}|\na w|\big|=o(|y|^{-1})$
and in turn
$\big|\na^{\Sigma_t}|\na w|\big|/|\na w|=o(1)$.
Thus, also the first term in De L'H\^opital quotient vanishes, as \(t \to -\infty\). This concludes the proof, in view of the fact that
the limit of such quotient implies
\begin{equation*}
\mmp_\Lambda(x,t)\, = \,\frac{\RRR(x)-2\Lambda}{12} \,|y(q)|^3 \,\, (1 + o(1))  \, ,
\qquad q\in\Sigma_t,\qquad\mbox{as }\ t\to-\infty,
\end{equation*}
since 
\[
\lim_{t\to-\infty}\frac{|\Omega_t|}{|y(q)|^3}=\frac{4\pi}3,\qquad q\in\Sigma_t.
\]
\end{proof}

\begin{remark}
\label{rem:ssl2}
We note that the third regime in the expansions~\eqref{eq:u_expansion} used in the above proof holds throughout the range $5/3<p<3$ (see \cite[Theorem 1.2]{Ago_Bor}). Expressed in terms of $w$, this regime gives, in particular,
\begin{equation}\label{agosto}
|\na w|\,=\,\frac{3-p}{|y|}+\mathcal O_1\Big(|y|^{2\big(\frac{2-p}{p-1}\big)}\Big),
\qquad\mbox{as}\quad y\to0, \qquad \frac53<p<3.
\end{equation}
When $p=2$, the error term is bounded. By explicitly keeping track, throughout the proof, of all contributions arising from this bounded term, one finds that the argument can still be completed - that is, the small sphere limit continues to hold - provided that a suitable additive constant is included in the Green function $u$, as done in~\cite[Theorem 3.1]{AMO}.
By contrast, when $2<p<3$, the error term in~\eqref{agosto} is bounded by a negative power of $|y|$. Tracing this contribution through the computations in the proof shows that, in general, the small sphere limit need not hold.
\end{remark}

\subsection{Global existence 
and asymptotic behavior of the structural coefficients.}
\label{sub:global}

In this last subsection, we provide global  existence for the structural coefficients $\alpha$, $\lambda$ and $\mu$ that have been defined or selected in the previous subsections. 

\begin{theorem}
[Global existence for the structural coefficients]
\label{thm:global}
For $\Lambda \in \R$ and $1 < p < 3$, let 
$\alpha = \alpha_\Lambda^{\!(p)}$ be 
the structural coefficient defined in formula~\eqref{eq:alpha_t}. Then the ancient solutions $\mu = \mu_\Lambda^{\!(p)}$ and $\lambda = \lambda_\Lambda^{\!(p)}$ to system~\eqref{structural}, selected in Theorem~\ref{thm:selection_mu} and Theorem~\ref{thm:selection_lambda}, respectively, are globally defined in time. 
\end{theorem}
The strategy we adopt to establish the above result is different depending on the sign of $\Lambda$. In the case $\Lambda = 0$, as already discussed in the introduction to this section, one can derive explicit formulas for all three structural coefficients. These closed-form expressions immediately yield global-in-time existence and, in addition, allow the derivation of more explicit monotonicity formulas. When $\Lambda < 0$, global existence will instead be established by means of a barrier argument, which is by now standard in the qualitative study of ordinary differential equations. This approach, however, cannot be straightforwardly extended to the case $\Lambda > 0$. Indeed, for values of $p$ close to $1$, the barriers degenerate and provide effective control only in neighborhoods of $-\infty$ and $+\infty$. To address this problem, we will show that it is possible to obtain a semi-explicit representation of the solution by performing a suitable change of variables and employing convergent power series expansions related to the definition of hypergeometric functions. Although this method is computationally more involved, its main advantage is that it makes it possible to characterize the asymptotic behavior of the structural coefficients as $t \to +\infty$. This information will be essential for deriving the explicit expression of the total mass, leading to Definition~\ref{def:polmass}. We finally remark that, in principle, the same strategy could also be applied in the analysis of the case $\Lambda < 0$. However, as explained in Remark~\ref{rem:Y_negativo}, the corresponding power series representation is only valid in a neighborhood of the pole and therefore does not provide the same global information.

\begin{proof}[Proof of Theorem~\ref{thm:global} - Case $\Lambda=0$]
As already observed, in the discussion leading to formula~\eqref{astd}, we have that $\alpha = \alpha_0^{\!(p)} \equiv 1/(3-p)$, in this case. With this coice of $\alpha$, the only ancient solution to~\eqref{eq:ODE_mu}, obeying the requirement in Therorem~\ref{thm:selection_mu} is the equilibrium solution $\mu_0^{\!(p)} \equiv 1/(3-p)$, which is clarly defined for all times. Having fixed $\alpha$ and $\mu$ as above, the only solution $\lambda$ to~\eqref{eq:lam} that obeys the prescription of Theorem~\ref{thm:selection_lambda} -- and thus allows for a small sphere limit theorem -- is 
\[
\lambda(t) \, =\, \lambda_0^{\!(p)} (t)  \,=\,\frac{t}{3-p}+\left(\frac{p-1}{3-p}\right)\log\left(\frac{p-1}{3-p}\right)-\log\left[8\pi(3-p)\right] \, ,
\]
which is globally defined as well.
\end{proof}

\medskip

\begin{proof}[Proof of Theorem~\ref{thm:global} - Case $\Lambda<0$]
We start by showing that \(0 < \alpha_\Lambda^{\!(p)} = \alpha < 1/(3-p)\). Positivity of \(\alpha\) follows directly from its definition. To prove the upper bound, 
we recall that \(R_\Lambda = +\infty\) in this parameter regime. Thus, applying de l’Hôpital’s rule, we obtain
\[
\lim_{r\to +\infty} \frac{u_\Lambda^{\!(p)}(r)}{r^{-\frac{2}{p-1}}} 
= 
\lim_{r\to +\infty}\frac{\displaystyle \int_r^{+\infty}\frac{d\rho}{\rho^{\frac{2}{p-1}}\sqrt{1-\frac{\Lambda}{3}\rho^2}}}{r^{-\frac{2}{p-1}}}
= 
\lim_{r\to +\infty}\frac{p-1}{2}\frac{r}{\sqrt{1-\frac{\Lambda}{3}r^2}}
= 
\frac{p-1}{2}\sqrt{\frac{3}{|\Lambda|}}\,,
\]
which in turn implies that \(\lim_{t\to+\infty}\alpha(t)=1/2 < 1/(3-p)\), since
\[
\alpha (t)\,= \,
\frac{1}{p-1} \,\, \frac{\sqrt{1-\frac{\Lambda}{3}r^2(t)} }{r(t)} \,\,\, \frac{u(r(t))}{r(t)^{-\frac{2}{p-1}}} \, .
\]
Differentiating the above expression with the help of~\eqref{eq:dotr}, one gets 
\begin{equation}
\label{eq:ODE_alpha}
\dot{\alpha} \,  =  \, - \left( \frac{1}{p-1} \right) \alpha + \left[\left(\frac{3-p}{p-1}\right) -\left(\frac{ ({\Lambda}/{3})\,r(t)^2}{1-(\Lambda/3) \, r(t)^2}\right)\right] \alpha^2\,.
\end{equation}
We then observe that the right-hand side of~\eqref{eq:ODE_alpha} is positive when $\alpha\geq 1/(3-p)$. In particular, if $\alpha$ ever reaches the value $\alpha=1/(3-p)$, then it would be increasing from then on, in contradiction with the fact that $\lim_{t\to+\infty}\alpha(t)<1/(3-p)$.
Let us denote by $\mu_\pm$ the functions defined in~\eqref{eq:mu_plus_minus}, which represent the values at which $\dot\mu(t)=0$. As already noted in the proof of Theorem~\ref{thm:selection_mu}, any solution $\mu$ of~\eqref{eq:ODE_mu} is strictly decreasing when $\mu_-<\mu<\mu_+$, and strictly increasing outside this interval. The upper bound for $\alpha$ 
obtained above allows us, as a first consequence, to infer that the functions $\mu_+(t)$ and $\mu_-(t)$ are real-valued and distinct for every $t\in\R$. Indeed, a straightforward, albeit long, algebraic computation shows that, using the condition $\alpha < 1/(3-p)$, the radicand in \eqref{eq:mu_plus_minus} remains strictly positive for all admissible values of the parameters. Now, by virtue of the ODE~\eqref{eq:ODE_mu} 
satisfied by $\mu$, we observe that
whenever $\mu (t)=1/(3-p)$ one has 
\[
\dot\mu (t) \,=\,\left(\frac{5-p}{3-p}\right)\alpha(t)\!\left(\!\alpha(t)-\frac{1}{3-p}\right)\,<\,0\,,
\]
since $0<\alpha (t) < 1/(3-p)$ for all $t \in \R$. In particular, this implies that $\mu_-(t)<1/(3-p)$ for all $t$. 
Therefore, we can repeat the proof of Theorem~\ref{thm:selection_mu} with the choice $\sigma^{(+)} (t) \equiv 1/(3-p)$ -- which is now admissible -- to deduce that the solution $\mu$ must remain below $1/(3-p)$ for $t$ sufficiently close to $-\infty$. Since $\dot \mu<0$ when $\mu=1/(3-p)$, we then conclude that $\mu<1/(3-p)$ for all $t$. On the other hand, from equation~\eqref{eq:ODE_mu} it is immediate that, when $\mu=0$, one has $\dot\mu=[(5-p)/(p-1)]\alpha^2>0$. Hence, the solution $\mu$ cannot cross below $0$. We thus obtain $0<\mu(t)<1/(3-p)$ for all $t$, which guarantees that $\mu$ is globally defined. The same conclusion applies to $\lambda$, by integrating~\eqref{eq:lam}.
\end{proof}

\medskip

\begin{proof}[Proof of Theorem~\ref{thm:global} - Case $\Lambda>0$] To establish global existence in this case, we are going to exhibit some global solution to our system of ODE~\eqref{structural} and verify that they obey the desired prescription, as $t \to -\infty$. As a first step, we reformulate~\eqref{structural},introducing the new unknowns
\begin{align}
\label{def:phipsi}
\Phi (r)\,=\,&\frac{\mu (w^{(p)}_\Lambda(r))  \, \exp(\lambda(w^{(p)}_\Lambda(r)))}{\a^2(w^{(p)}_\Lambda(r))} \,=\,\frac{\mu \, e^\lambda}{\a^2}_{\left|{t}=w^{(p)}_\Lambda(r)\right.} \, , \\
\Psi(r)  \,=\,&\frac{ \exp(\lambda(w^{(p)}_\Lambda(r)))}{\a(w^{(p)}_\Lambda(r))} \,=\,\frac{ e^\lambda}{\a}_{\left|{t}=w^{(p)}_\Lambda(r)\right.}\,,
\end{align}
where $0<r<R_\Lambda$. Using repeatedly the differential identities~\eqref{eq:dotrinv} and~\eqref{eq:ODE_alpha}, it is not hard to check that $\lambda$ and $\mu$ solve~\eqref{structural} with a positive $\alpha$ if and only if $\Phi$ and $\Psi$ are positive solutions to the new ODE system
\begin{equation}
\label{eq:ODE_Phi_Psi_app}
\begin{dcases}
\frac{d\Phi}{dr}\,=\,2\left[\frac{(\Lambda/3) \, r^2}{1-(\Lambda/3) \, r^2}-\left(\frac{3-p}{p-1}\right)\right]\frac{1}{r}\, \Phi \, + \,  \left(\frac{5-p}{p-1}\right)\frac{1}{r}\, \Psi \, ,
\\
\frac{d\Psi}{dr}\,=\,- \left(\frac{3-p}{p-1}\right)\frac{1}{r}\, \Phi \, + \, \left[\frac{(\Lambda/3) \, r^2}{1- (\Lambda/3)\, r^2}+\left(\frac{2}{p-1}\right)\right]\frac{1}{r} \, \Psi \, .
\end{dcases}
\end{equation}
Compared to the previous formulation, the new one offers two key advantages: the resulting system is linear in the unknowns, and every coefficient is an explicit rational function of the variable. The trade-off is that the system is no longer semi-decoupled. For the main objective of the theorem, this trade-off is benign. Linearity alone is enough to guarantee global existence, and the comparatively simple coefficient structure will later let us extract sharp information on the behavior of solutions as $r \to R_\Lambda$, or, equivalently, as $t \to +\infty$, once we translate back to the original unknowns and variable.

To identify solutions of system~\eqref{eq:ODE_Phi_Psi_app} that are consistent with the choice of structural coefficients $\alpha$, $\mu$, and $\lambda$ made in the preceding subsections, we impose the following asymptotic conditions in the limit $r \to 0$:
\begin{equation}
\label{limphipsizero}
\lim_{r\to 0^+}\frac{\Phi(r)}{r}\,=\,\frac{1}{8\pi}\,\qquad \hbox{and} \qquad
\lim_{r\to 0^+}\frac{\Psi(r)}{r}\,=\,\frac{1}{8\pi}\,.
\end{equation}
These conditions are, in turn, dictated by the asymptotic expansions of the selected coefficients in the regime $t \to -\infty$. 

In Proposition~\ref{pro:Phi_Psi_solutions} of Appendix~\ref{app:ODE}, we provide a detailed proof that the only positive solutions of system~\eqref{eq:ODE_Phi_Psi_app} satisfying the asymptotic conditions~\eqref{limphipsizero} are given by
\begin{equation}
\label{eq:Phi_Psi_app}
\begin{aligned}
\Phi(r)\,&=\,\frac{1}{8\pi}\, \frac{r}{1-(\Lambda/{3}) \, r^2}\,\, \Upsilon\!\bigl(\tfrac{\Lambda}{3} r^2\bigr)\,,
\\
\Psi(r)\,&=\,\frac{1}{8\pi} \, \frac{r}{1-(\Lambda/{3}) \,r^2} \, \left[\Upsilon\!\bigl(\tfrac{\Lambda}{3}r^2\bigr) \, +2 \!\left(\frac{p-1}{5-p}\right)({\Lambda}/{3}) \, r^2\,\, \dot\Upsilon\!\bigl(\tfrac{\Lambda}{3}r^2\bigr)\right]\,  \, .
\end{aligned}
\end{equation}
The function $\Upsilon$ that appears in these expressions can be expressed as a power series and coincides with the hypergeometric function with parameters $a_p$, $b_p$, and $c_p$, that is,
\[
\Upsilon(x)\,=\,{}_2F_1\bigl(a_p,b_p,c_p;x\bigr)\, = \, \sum_{k=0}^{+\infty}\frac{(a)_k (b)_k}{(c)_k}\frac{x^k}{k!}\,,\quad x\in(0,1)\,,
\]
where
\begin{align}
\label{p-parameters}
a_p\,&=\,\frac{3-p+\sqrt{4+12(p-1)-3(p-1)^2}}{4(p-1)}\,, \nonumber
\\
b_p\,&=\,\frac{3-p-\sqrt{4+12(p-1)-3(p-1)^2}}{4(p-1)}\,,
\\ \nonumber
c_p \,& = \, \frac{p}{p-1}\,.
\end{align}
We recall that, for $h \in \R$ and $k \in \mathbb{N}$, the notation $(h)_k$ denotes the so‑called Pochhammer symbol, defined by
\[
(h)_k\,=\,
\begin{dcases}
1 & \text{se }k=0 \, ,
\\
h(h+1)(h+2)\cdots (h+k-1)& \text{se }k>0 \, ,
\end{dcases}
\]
whereas the notation $\dot\Upsilon(x)$ simply denotes the first derivative of the function $\Upsilon$ with respect to the variable $x$.

Without entering into the detailed proof -- for which we refer to Proposition~\ref{pro:Phi_Psi_solutions} mentioned above -- we provide here a brief synopsis of the main ideas and the essential components of the argument. First of all, the fact that the functions $\Phi$ and $\Psi$ defined in~\eqref{eq:Phi_Psi_app} indeed solve the system~\eqref{eq:ODE_Phi_Psi_app} is a consequence of the following general property of hypergeometric functions: every hypergeometric function ${}_2F_1(a,b,c;x)$ with parameters $a$, $b$, and $c$ satisfies the second‑order differential equation
\[
x(1-x) \, \ddot y+[c-(a+b+1)x] \, \dot y-ab \, y\,=\,0\,.
\]
Since $a_p+b_p=(3-p)/(2(p-1))$ and $a_p b_p=-(5-p)/(4(p-1))$, we have that, in particular, $\Upsilon$ satisfies
\begin{equation*}
x (1-x) \,\ddot\Upsilon(x)\,=\,\left[-\frac{p}{p-1}+\frac{p+1}{2(p-1)}x\right]{\dot\Upsilon(x)} -\frac{5-p}{4(p-1)}{\Upsilon(x)}\,.
\end{equation*}
This identity can be employed to compute the first derivatives of $\Phi$ and $\Psi$ and to verify, by means of a lengthy yet straightforward calculation, that the ODE system is indeed satisfied.  
In order to confirm that the asymptotic initial conditions at $r=0$ are also fulfilled, one differentiates the corresponding power series, which yields the identity
\begin{equation*}
\dot\Upsilon(x)\,=\,-\frac{5-p}{4p}\,{}_2F_1(a_p+1,b_p+1,c_p+1;x)\,.
\end{equation*}
Recalling that any hypergeometric function evaluated at $x=0$ equals $1$ -- a fact that follows immediately from its power-series representation -- one concludes that conditions~\eqref{limphipsizero} are satisfied. The positivity of $\Phi$ and $\Psi$ follows from the fact -- established in Subsection~\ref{app:Upsilon} -- that
\[
\Upsilon(x)>0\,,\quad \dot\Upsilon(x)<0\,,\quad
\ddot\Upsilon(x)<0\,, \quad x\in(0,1) \, .
\]
These properties imply directly that $\Phi$ is strictly positive and that the mapping
\[
(0,1) \,\ni\, x \,\longmapsto\, \Upsilon(x) \,+\, 2 \left(\frac{p-1}{5-p}\right) x \,\dot\Upsilon(x)\,,
\]
of which $\Psi$ is a positive multiple, is strictly decreasing on $(0,1)$. Noting that
\[
\Upsilon(1)\,+\, 2\!\left(\frac{p-1}{5-p}\right)\dot\Upsilon(1)=0\,,
\]
one concludes that $\Psi(r)>0$ for all $0<r<R_\Lambda$. Having all these elements at our disposal, we can finally invert relations~\eqref{def:phipsi}, thereby constructing two solutions $\lambda$ and $\mu$ of system~\eqref{structural}, defined for all times and exhibiting the prescribed asymptotic expansions as $t \to -\infty$. This ensures that these globally defined solutions are indeed the maximal extensions fo the ones constructed in Theorem~\ref{thm:selection_mu} and Theorem~\ref{thm:selection_lambda}.
\end{proof}

\section{Further properties of the structural coefficients in the case \texorpdfstring{$\Lambda>0$}{Λ>0}}
\label{sec:further}

In this section we focus on the case $\Lambda>0$ and we carry out a more detailed study of the structural coefficients $\alpha$, $\lambda$, and $\mu$ selected in Theorem~\ref{thm:global}. In Subsection~\ref{sub:pifissato} we analyze their asymptotic behavior as $t\to+\infty$ with $p$ fixed. Combining Theorem~\ref{thm:expansions_mu_lambda} with the analysis of Section~\ref{sec:selection} for $t\to-\infty$, we obtain the summarizing Corollary~\ref{cor:behaviour_with_t}.
These results are crucial for Section~\ref{sec:polar_PMT}, especially for the heuristics behind Definition~\ref{def:polmass} and for the proof of Theorem~\ref{thm:pos_pass_pol}, which corresponds to Theorem~\ref{thm:pos_pass_pol_Intro} in the Introduction and is one of the main results of the manuscript. 

In Subsection~\ref{sub:tifissato} we then study the behavior of the structural coefficients as $p\to1^+$ with $t$ fixed. Although these results are not used directly in the proofs of Theorems~\ref{thm:pos_pass_pol_Intro} and~\ref{thm:pos_pass_total_Intro}, they have significant conceptual value: they support the heuristic observations of Subsection~\ref{sub:pto1}, they pave the way for a further systematic study of the elusive geometric limit of our theory as $p\to1^+$, where well-posedness of the Dirichlet problem and monotonicity formulas seem to break down.

\subsection{
Asymptotics for a fixed \texorpdfstring{$p$}{p}.
\label{sub:pifissato}
}

We start with the asymptotic behavior of the structural coefficients $\mu$ and $\lambda$ as $t \to +\infty$. We focus on the case $\Lambda>0$, since it is mainly based on the possibility of representing the selected solutions by power series (or equivalently through hypergeometric functions), as detailed in Appendix~\ref{app:ODE}. 

\begin{theorem} 
\label{thm:expansions_mu_lambda}
For $\Lambda >0 $ and $1 < p < 3$, let 
$\alpha = \alpha_\Lambda^{\!(p)}$ be 
the structural coefficient defined in formula~\eqref{eq:alpha_t} and let 
$\mu = \mu_\Lambda^{\!(p)}$ and $\lambda = \lambda_\Lambda^{\!(p)}$ be the globally defined solutions to system~\eqref{structural} obtained in Theorem~\ref{thm:global}, extending the ancient solution selected in Theorem~\ref{thm:selection_mu} and Theorem~\ref{thm:selection_lambda}, respectively. Then, it holds
\begin{align*}
\lim_{t \to +\infty} e^{\frac{t}{p-1}}e^{\lambda(t)} \, &= \, \frac{R_\Lambda^{\frac{2}{p-1}}}{8 \pi (p-1)}\, \frac{ \Gamma(\tfrac12) \, \Gamma(c_p)}{\Gamma(a_p+1) \, \Gamma(b_p+1)} \,,\\
\lim_{t \to +\infty} e^{\frac{t}{p-1}}\mu(t)\, &= \, \frac{R_\Lambda^{\frac{3-p}{p-1}}}{2(p-1)}\, \, \frac{\Gamma(a_p+1)\, \Gamma(b_p+1)}{\Gamma(a_p+\frac{3}{2}) \, \Gamma(b_p+\frac{3}{2})} \,, 
\end{align*}
where $a_p,b_p$ and $c_p$ are defined in~\eqref{p-parameters} and $\Gamma$ is Euler's Gamma function.
\end{theorem}

\begin{proof}
First of all, we observe that the Taylor expansion at $r = R_\Lambda$ of the function $u_\Lambda^{\!(p)}$, defined in~\eqref{eq:up_dS}, is given by
\[
u_\Lambda^{\!(p)} (r)\, = \, R_{\Lambda}^{-\frac{3-p}{p-1}} \, \sqrt{1- (\Lambda/3) r^2} \,\, (1 + o(1)) \, , \quad \hbox{as} \,\, r \to  R_\Lambda \, .
\]
It follows that the function $A_\Lambda^{\!(p)}$ defined in~\eqref{eq:alpha_r} satisfies the expansion
\[
A_\Lambda^{\!(p)} (r) \, = \, \frac{1}{p-1} \,  \left(1- (\Lambda/3) r^2 \right) \,\, (1 + o(1)) \, , \quad \hbox{as} \,\, r \to  R_\Lambda \, .
\]
Combining this information with the expansions given by Proposition~\ref{pro:behaviour_near_one} in Appendix~\ref{app:ODE}, one gets 
\begin{align*}
\left( A_\Lambda^{\!(p)} \,\, \Psi  \right) (r)\,  \,& = \, \frac{R_\Lambda}{8{\pi (p-1)}}\,\, \frac{\Gamma(\tfrac{1}{2})\, \Gamma(c_p)}{\Gamma(a_p+1)\, \Gamma(b_p+1)} \, \sqrt{1- (\Lambda/3) r^2} \,\, (1 + o(1)) \, , \quad \hbox{as} \,\, r \to  R_\Lambda 
\\
\left(A_\Lambda^{\!(p)}\! \, \frac{\Phi}{\Psi} \right)  (r)\, & = \, \frac{1}{2(p-1)}\, \, \frac{\Gamma(a_p+1)\, \Gamma(b_p+1)}{\Gamma(a_p+\frac{3}{2}) \, \Gamma(b_p+\frac{3}{2})} \, \sqrt{1- (\Lambda/3) r^2} \,\, (1 + o(1)) \, , \quad \hbox{as} \,\, r \to  R_\Lambda
\end{align*}
Now, combining the expansion of $u_\Lambda^{\!(p)}$ with the definition~\eqref{eq:rp_dS} of $r_\Lambda^{\!(p)}$, we get 
\[
R_{\Lambda}^{\frac{3-p}{p-1}} \, e^{-\frac{t}{p-1}} \,  = \,R_{\Lambda}^{\frac{3-p}{p-1}} \, u_\Lambda^{\!(p)} ( r_\Lambda^{\!(p)} (t) ) \, =   \, \sqrt{1- (\Lambda/3) \big(r_\Lambda^{\!(p)} (t)\big)^2} \,\, (1 + o(1))\, , \quad \hbox{as} \,\,\,  t \to + \infty \, .
\]
This fact can be employed to deduce that 
\begin{align*}
e^{\lambda(t)} \, &= \, \left( A_\Lambda^{\!(p)} \,\, \Psi  \right) \big(r_\Lambda^{\!(p)} (t) \big)   = \,\frac{R_\Lambda^{\frac{2}{p-1}}}{8 \pi (p-1)}\, \frac{ \Gamma(\tfrac12) \, \Gamma(c_p)}{\Gamma(a_p+1) \, \Gamma(b_p+1)} \, e^{-\frac{t}{p-1}} \,(1 +o(1)) \,,\\
\mu(t)\, &= \, \left( A_\Lambda^{\!(p)} \,\, \frac{\Phi}{\Psi}  \right) \big(r_\Lambda^{\!(p)} (t) \big)   = \frac{R_\Lambda^{\frac{3-p}{p-1}}}{2(p-1)}\, \, \frac{\Gamma(a_p+1)\, \Gamma(b_p+1)}{\Gamma(a_p+\frac{3}{2}) \, \Gamma(b_p+\frac{3}{2})} \, e^{-\frac{t}{p-1}} \,(1 +o(1)) \,, 
\end{align*}
as $t \to +\infty$. The thesis follows at once. 
\end{proof}
For the sake of reference, we state the following corollary, in which the asymptotic expansions for the chosen coefficients are collected.
\begin{corollary}[Asymptotic behavior of the structural coefficients]
\label{cor:behaviour_with_t}
For $\Lambda >0 $ and $1 < p < 3$, let 
$\alpha = \alpha_\Lambda^{\!(p)}$ be 
the structural coefficient defined in formula~\eqref{eq:alpha_t} and let 
$\mu = \mu_\Lambda^{\!(p)}$ and $\lambda = \lambda_\Lambda^{\!(p)}$ be the globally defined solutions to system~\eqref{structural} obtained in Theorem~\ref{thm:global}, extending the ancient solution selected in Theorem~\ref{thm:selection_mu} and Theorem~\ref{thm:selection_lambda}, respectively. Then, it holds
\begin{align*}
e^{\lambda(t)}&=\frac{1}{8\pi(3-p)}\left(\frac{p-1}{3-p}\right)^{\!\frac{p-1}{3-p}} \!\! e^{\frac{t}{3-p}}(1+o(1))\,,\quad\hbox{as } t\to-\infty\,,
\\
e^{\lambda(t)}  &= \, \frac{R_\Lambda^{\frac{2}{p-1}}}{8 \pi (p-1)}\, \frac{ \Gamma(\tfrac12) \, \Gamma(c_p)}{\Gamma(a_p+1) \, \Gamma(b_p+1)} \,\,\, e^{-\frac{t}{p-1}} \,(1 +o(1)) \,,\quad\hbox{as } t\to+\infty\,,
\end{align*}
and 
\begin{align*}
\mu(t)&= \, \frac{1}{3-p} \, \,(1+o(1))\,,\quad\hbox{as } t\to-\infty\,,
\\
\mu(t) &= \, \frac{R_\Lambda^{\frac{3-p}{p-1}}}{2(p-1)}\, \, \frac{\Gamma(a_p+1)\, \Gamma(b_p+1)}{\Gamma(a_p+\frac{3}{2}) \, \Gamma(b_p+\frac{3}{2})} \,\,\, e^{-\frac{t}{p-1}} \,(1 +o(1)) \,,\quad\hbox{as } t\to+\infty\,,
\end{align*}
where $a_p,b_p$ and $c_p$ are defined in~\eqref{p-parameters} and $\Gamma$ is Euler's Gamma function.
\end{corollary}

\subsection{Asymptotics for a fixed \texorpdfstring{$t$}{t}.}
\label{sub:tifissato}

We now analyze the structural coefficients in the limit as $p \to 1^+$. A key role is played by the pseudo-radial functions implicitly defined by formula~\eqref{eq:rp_dS}; we begin with them.
\begin{lemma}
\label{le:r1}
Let $\Lambda>0$.
The function $r_\Lambda^{\!(p)}(t)$, defined by~\eqref{eq:rp_dS}, converges pointwise to the function
\[
r_\Lambda^{\!(1)}(t)\,=\,
\begin{dcases}
e^{t/2} & \hbox{if $t< 2 \log (R_\Lambda)$}
\\
R_\Lambda & \hbox{if $t \geq 2 \log (R_\Lambda)$}
\end{dcases}
\]
as $p\to 1^+$.
\end{lemma}

\begin{proof}
By definition of $r_\Lambda^{(p)}$, we have $r_\Lambda^{(p)}(t)\in(0,R_\Lambda)$ and
\[
e^{-\frac{t}{p-1}}\,=\,u_\Lambda^{(p)}(r_\Lambda^{(p)}(t))\,=\,\int_{r_\Lambda^{(p)}(t)}^{R_\Lambda}\frac{1}{\rho^{\frac{2}{p-1}}\sqrt{1-\frac{\Lambda}{3}\rho^2}}\,d\rho
\,.
\]
In particular, using the fact that $\sqrt{1-\Lambda\rho^2/3}\leq 1$, we compute
\[
e^{-\frac{t}{p-1}}\geq
\left[\frac{\rho^{-\frac{2}{p-1}+1}}{-\frac{2}{p-1}+1}\right]^{R_\Lambda}_{r_\Lambda^{(p)}(t)}\,=\,\frac{p-1}{3-p}\left({r_\Lambda^{(p)}(t)}^{-\frac{3-p}{p-1}}-R_\Lambda^{-\frac{3-p}{p-1}}\right)\,,
\]
which can be rewritten as
\[
r_\Lambda^{(p)}(t)\geq R_\Lambda \left[1+\frac{3-p}{p-1}\frac{1}{R_\Lambda}\left(R_\Lambda^2e^{-t}\right)^{\frac{1}{p-1}}\right]^{-\frac{p-1}{3-p}}\,.
\]
Taking the limit as $p\to 1$ of the right hand side of the above inequality, we deduce that
\[
\liminf_{p\to 1^+} r_\Lambda^{(p)}(t)\geq
\begin{dcases}
e^{\frac{t}{2}} & \hbox{if }t<2\log R_\Lambda\,,
\\
R_\Lambda              & \hbox{if }t\geq 2\log R_\Lambda\,.
\end{dcases}
\]
Since $r_\Lambda^{(p)}(t)\in(0,R_\Lambda)$, it follows that $r_\Lambda^{(p)}(t)\to R_\Lambda$ for every $t\geq 2\log R_\Lambda$. Concerning the opposite inequality, we have the following estimate
\begin{align*}
e^{-\frac{t}{p-1}}\,&=\,\int_{r_\Lambda^{(p)}(t)}^{R_\Lambda}\frac{1}{\rho^{\frac{2}{p-1}}\sqrt{1-\frac{\Lambda}{3}\rho^2}}d\rho
\\
&\leq\,r_\Lambda^{(p)}(t)^{-\frac{2}{p-1}}\int_{r_\Lambda^{(p)}(t)}^{R_\Lambda}\frac{1}{\sqrt{1-\frac{\Lambda}{3}\rho^2}}d\rho
\\
&=\,r_\Lambda^{(p)}(t)^{-\frac{2}{p-1}}R_\Lambda\left[\frac{\pi}{2}-\arcsin\left(r^{(p)}_\Lambda(t)/R_\Lambda\right)\right]
\\
&\leq\,\frac{\pi}{2}R_\Lambda\,r_\Lambda^{(p)}(t)^{-\frac{2}{p-1}}\,.
\end{align*}
As a consequence, we have
\[
r_\Lambda^{(p)}(t)^2\leq \left(\frac{\pi}{2}R_\Lambda\right)^{p-1}e^{t}\,.
\]
Taking the limit as $p\to 1$, we have $\limsup_{p\to 1}r_p(t)\leq e^{t/2}$. This concludes the proof.
\end{proof}
We now turn to the analysis of the structural coefficient $\alpha$. Considering that in the model case it should represent a proportionality factor between a half of the mean curvature and the gradient of the $p$-IMCF, and assuming that as $p \to 1^+$ the $p$-IMCF converges to a form of IMCF, it is not difficult to conjecture that $\alpha$ must converge to $1/2$. The next lemma confirms this heuristic.
\begin{lemma}
Let $\Lambda >0$. The structural coefficient $\alpha_\Lambda^{\!(p)}$, defined in formula~\eqref{eq:alpha_t}, satisfies the following pointwise convergence for $t \neq 2 \log(R_\Lambda)$
\[
\alpha_\Lambda^{\!(p)}(t)\,\longrightarrow \,
\begin{dcases}
\frac{1}{2} & \hbox{if $t< 2 \log (R_\Lambda)$}
\\
0 & \hbox{if $t > 2 \log (R_\Lambda)$}
\end{dcases}
\]
as $p\to 1^+$.
\end{lemma}

\begin{proof}
Recall from~\eqref{eq:alpha_t} that
\[
\alpha^{\!(p)}_\Lambda (t)\,= \,
\frac{e^{-\frac{t}{p-1}}}{p-1} \,\, \sqrt{1-\frac{\Lambda}{3}r_\Lambda^2}\,\,\, r_\Lambda^{\frac{3-p}{p-1}} \, , \quad t \in (-\infty,+ \infty)
\]
We start by discussing the limit of $\alpha_\Lambda^{(p)}(t)$ in the case $t>2\log R_\Lambda$. Since $r_\Lambda^{(p)}(t)\leq R_\Lambda$, we have
\[
\alpha_\Lambda^{(p)}(t)\,\leq\,\frac{1}{p-1}(R_\Lambda)^{\frac{3-p}{p-1}}e^{-\frac{t}{p-1}}\,=\,\frac{1}{(p-1)R_\Lambda}e^{-\frac{t-2\log R_\Lambda}{p-1}}\,.
\]
For $t>2\log R_\Lambda$, the limit on the right-hand side as $p\to 1$
is zero, as desired. It remains to show that $\lim_{p\to 1^+}\alpha_\Lambda^{(p)}(t)=1/2$ for all $t<2\log R_\Lambda$. To this end, in view of Lemma~\ref{le:r1}, it is enough to prove that the function
\[
A^{\!(p)}_\Lambda(r) \,= \, 
\frac{1}{p-1} \,\, \sqrt{1-\frac{\Lambda}{3}r^2}\,\,\, r^{\frac{3-p}{p-1}}\,\int_{r}^{R_\Lambda}\frac{1}{\rho^{\frac{2}{p-1}}\sqrt{1-\frac{\Lambda}{3}\rho^2}}\,d\rho
\]
satisfies
$\lim_{p\to 1}A_\Lambda^{(p)}(r)=1/2$ for all $r\in(0,R_\Lambda)$. On the one hand, since $\sqrt{1-\Lambda r^2/3}\geq\sqrt{1-\Lambda\rho^2/3}$ for all $r\leq\rho< R_\Lambda$, we have
\[
A_\Lambda^{(p)}(r)\,\geq\,
\frac{1}{p-1}r^{\frac{3-p}{p-1}}\int_r^{R_\Lambda}{\rho^{-\frac{2}{p-1}}}d\rho\,=\,
\frac{1}{3-p}\left[1-\left(\frac{r}{R_\Lambda}\right)^{\frac{3-p}{p-1}}\right].
\]
Since $0\leq r<R_\Lambda$, we then get
\begin{equation}
\label{liminf}
\liminf_{p\to 1}A_\Lambda^{(p)}(r)\,\geq\,\frac12\,.
\end{equation}
We now prove that 
\begin{equation}\label{limsup}
\limsup_{p\to 1}A_\Lambda^{(p)}(r)\,\leq\,\frac12\,.
\end{equation}
To do this, we fix an integer $k\geq2$ and consider the function
\[
[0,1)\ni \rho\longmapsto\sqrt{1-\Lambda\rho^2/3}\,\rho^{\frac1{k(p-1)}}.
\]
This function is strictly increasing up to the value $s_k(p):=R_\Lambda/\sqrt{k(p-1)+1}$,
which is such that $s_k(p)\to R_\Lambda^-$, as $p\to 1^+$. We then estimate  
$A_\Lambda^{(p)}(r)$ from above as 
\begin{equation*}
A_\Lambda^{(p)}(r)
\,\leq\,
\underbrace{
\frac{1}{p-1}\frac{r^{\frac{3-p}{p-1}}}{r^{\frac1{k(p-1)}}}
\int_r^{s_k(p)}
\rho^{-\frac2{p-1}+\frac1{k(p-1)}}
}
_{(I)}
\,+\,
\underbrace{
\frac{r^{\frac{3-p}{p-1}}}{p-1}\sqrt{1-\frac{\Lambda}{3}r^2}
\int_{s_k(p)}^{R_\Lambda}\frac{\rho^{-\frac{2}{p-1}}}{\sqrt{1-\frac{\Lambda}{3}\rho^2}}d\rho
}
_{(II)}.
\end{equation*}
It turns out that
\[
(I)\,=\,
-\frac k{3k-kp-1}\left(r^{\frac{3k-kp-1}{k(p-1)}}
\big(s_k(p)\big)^{\frac{-3k+kp+1}{k(p-1)}}
-1\right)
\ \longrightarrow\ \frac k{2k-1},
\qquad\mbox{as\ }p\to1,
\]
because 
\[
r^{\frac{3k-kp-1}{k(p-1)}}
\ \longrightarrow\ 0,
\qquad\quad
\big(s_k(p)\big)^{\frac{-3k+kp+1}{k(p-1)}}
\ \longrightarrow\
{\rm e}^{k(k-\frac12)}.
\]
At the same time,
\begin{align*}
(II)\,&\leq \,
\frac{r^{\frac{3-p}{p-1}}}{p-1}\sqrt{1-\frac{\Lambda}{3}r^2}\,
\big(s_k(p)\big)^{-\frac{2}{p-1}}
\int_{s_k(p)}^{R_\Lambda}\frac 1{\sqrt{1-\frac{\Lambda}{3}\rho^2}}{\rm d}\rho\\
\,&= \,
\frac{r^{\frac{3-p}{p-1}}}{p-1}\sqrt{1-\frac{\Lambda}{3}r^2}\,
\big(s_k(p)\big)^{-\frac{2}{p-1}}
R_\Lambda \left[\frac\pi2-\arcsin\left(\frac{s_k(p)}{R_\Lambda}\right)\right]
\,\longrightarrow\ 0,
\qquad\mbox{as}\quad p\to1,
\end{align*}
because
\[
\frac{r^{\frac{3-p}{p-1}}}{p-1}\,\longrightarrow\,0,
\qquad
\big(s_k(p)\big)^{-\frac{2}{p-1}}
\,\longrightarrow\,{\rm e}^k,
\qquad
\left[\frac\pi2-\arcsin\big(s_k(p)\big)\right]
\,\longrightarrow\,0.
\]
All in all, we have that
\[
\limsup_{p\to 1} A_\Lambda^{(p)}(r)
\,\leq \,
\frac k{2k-1}.
\]
The fact that this inequality holds for every integer $k\geq 2$
yields the desired \eqref{limsup}
and in turn concludes the proof of the lemma.
\end{proof}

With the next two lemmas we analyze the asymptotic behavior, as $p \to 1^+$, of the structural coefficients $\lambda$ and $\mu$, focusing already on the expressions that appear in the monotonicity formulas, namely $e^\lambda$ and $\mu e^\lambda$, so as to facilitate the considerations of Subsection~\ref{sub:pto1} and to make the formal convergence to the Hawking mass immediately evident.

\begin{lemma}
\label{le:eallalambda}
Let $\Lambda>0$. Let $\lambda_\Lambda^{\!(p)}$ be the ancient solution to system~\eqref{structural}, selected in Theorem~\ref{thm:selection_lambda}. The function $e^{\lambda_\Lambda^{\!(p)} (t)}$ satisfies the following pointwise convergence for $t \neq 2 \log(R_\Lambda)$
\[
e^{\lambda_\Lambda^{(p)} \!(t)}\,\longrightarrow \,
\begin{dcases}
\frac{e^{t/2}}{16 \pi} & \hbox{if $t< 2 \log (R_\Lambda)$}
\\
0 & \hbox{if $t > 2 \log (R_\Lambda)$}
\end{dcases}
\]
as $p\to 1^+$.
\end{lemma}

\begin{proof}
In the proof of Theorem~\ref{thm:global} it has been shown that $e^{\lambda_\Lambda^{(p)}}$ is given by
\[
e^{\lambda_\Lambda^{(p)}(t)}\,=\,\alpha_\Lambda^{(p)}(t)\Psi_p\left(r_\Lambda^{(p)}(t)\right)\,.
\]
We recall from Proposition~\ref{pro:behaviour_near_one_2} that
\[
\lim_{p\to 1^+}\Psi_p(r)= \frac{r}{8\pi}\,,\quad r\in[0,R_\Lambda)\,.
\]
As a consequence, recalling that $\alpha_\Lambda^{(p)}(t)\to 1/2$ for all $t<2\log R_\Lambda$ and $r_\Lambda^{(p)}(t)\to e^{t/2}$ for all $t<2\log R_\Lambda$, we conclude
\[
\lim_{p\to 1^+}e^{\lambda_{\Lambda}^{(p)}(t)}\,=\,\frac{e^{t/2}}{16\pi}\,.
\]
Concerning the values $t>2\log R_\Lambda$, notice that in this case we have $r_\Lambda^{(p)}(t)\to R_\Lambda$. In particular, from~\eqref{eq:behaviour_rp_to_one_psi} we deduce that 
\[
\lim_{p\to 1^+}\Psi_p(r_\Lambda^{(p)}(t))\sqrt{1-\frac{\Lambda}{3}(r_\Lambda^{(p)}(t))^2}=0\,.
\]
Since $r_\Lambda^{(p)}(t)\leq R_\Lambda$, we thus find
\begin{align*}
e^{\lambda_\Lambda^{(p)}(t)}\,&=\,\alpha^{\!(p)}_\Lambda (t)\Psi_p(r_\Lambda^{(p)}(t))\,= \,
\frac{e^{-\frac{t}{p-1}}}{p-1} \,\, \sqrt{1-\frac{\Lambda}{3}(r_\Lambda^{(p)}(t))^2}\,\,\,(r_\Lambda^{(p)}(t))^{\frac{3-p}{p-1}}\Psi_p(r_\Lambda^{(p)}(t))
\\
&\leq\,\frac{e^{-\frac{t-2\log R_\Lambda}{p-1}}}{p-1}\frac{1}{R_\Lambda}\left[\Psi_p(r_\Lambda^{(p)}(t))\sqrt{1-\frac{\Lambda}{3}(r_\Lambda^{(p)}(t))^2}\right]\,.
\end{align*}
As shown above, the quantity inside the square brackets goes to zero. Furthermore, the quantity outside the square brackets decays exponentially as $p\to 1$.
Since $e^{\lambda_\Lambda^{(p)}(t)}> 0$, this proves that $e^{\lambda_\Lambda^{(p)}(t)}\to 0$ as $p\to 1$ for all $t>2\log R_\Lambda$.
\end{proof}

\begin{lemma}
\label{mupi}
Let $\Lambda>0$. Let $\mu_\Lambda^{\!(p)}$ and $\lambda_\Lambda^{\!(p)}$ be the ancient solutions to system~\eqref{structural}, selected in Theorem~\ref{thm:selection_mu} and Theorem~\ref{thm:selection_lambda}, respectively.
The function $\mu_\Lambda^{\!(p)}(t)$ satisfies the following pointwise convergence for $t \neq 2 \log(R_\Lambda)$
\[
\mu_\Lambda^{\!(p)}(t)e^{\lambda_\Lambda^{\!(p)}(t)}\,\longrightarrow \,
\begin{dcases}
\frac{e^{t/2}}{32\pi} & \hbox{if $t< 2 \log (R_\Lambda)$}
\\
0 & \hbox{if $t > 2 \log (R_\Lambda)$}
\end{dcases}
\]
as $p\to 1^+$.
\end{lemma}

\begin{proof}
Recall from the proof of Theorem~\ref{thm:global} that
\[
\mu_\Lambda^{(p)}(t)e^{\lambda_\Lambda^{\!(p)}(t)}\,=\,\left(\alpha_\Lambda^{(p)}(t)\right)^2\Phi_p(r_\Lambda^{(p)}(t))\,,
\]
and from Proposition~\ref{pro:behaviour_near_one_2} that
\[
\lim_{p\to 1^+}\Phi_p(r)\,=\,\frac{r}{8\pi}\quad r\in[0,R_\Lambda)\,.
\]
For $t<2\log R_\Lambda$, the result now follows immediately from the previously discussed behaviors of $\alpha_\Lambda^{(p)}$ and $\Psi_p$ as $p\to 1$. Concerning $t>2\log R_\Lambda$, again we observe that $r_\Lambda^{(p)}(t)\to R_\Lambda$ as $p\to 1$, hence from~\eqref{eq:behaviour_rp_to_one_phi}, we deduce
\[
\lim_{p\to 1^+}\Phi_p(r_\Lambda^{(p)}(t))\left(1-\frac{\Lambda}{3}(r_\Lambda^{(p)}(t))^2\right)=0\,.
\]
Consequently, since $r_\Lambda^{(p)}(t)<R_\Lambda$, we estimate
\begin{align*}
\mu_\Lambda^{(p)}(t)e^{\lambda_\Lambda^{\!(p)}(t)}\,&=\,\left(\alpha_\Lambda^{(p)}(t)\right)^2\Phi_p(r_\Lambda^{(p)}(t))\,= \,
\frac{e^{-\frac{2t}{p-1}}}{(p-1)^2} \left(1-\frac{\Lambda}{3}(r_\Lambda^{(p)}(t))^2\right)\left(r_\Lambda^{(p)}(t)\right)^{2\frac{3-p}{p-1}}\Phi_p(r_\Lambda^{(p)}(t))
\\
&\leq\,\frac{e^{-2\frac{t-2\log R_\Lambda}{p-1}}}{(p-1)^2}\frac{1}{R^2_\Lambda}\left[\Phi_p(r_\Lambda^{(p)}(t))\left(1-\frac{\Lambda}{3}(r_\Lambda^{(p)}(t))^2\right)\right]\,.
\end{align*}
As shown above, the quantity inside the square brackets goes to zero. Furthermore, the quantity outside the square brackets decays exponentially as $p\to 1$.
Since $\mu_\Lambda^{(p)}(t)> 0$ and $e^{\lambda_\Lambda^{(p)}(t)}>0$, this proves that $\mu_\Lambda^{(p)}(t)e^{\lambda_\Lambda^{(p)}(t)}\to 0$ as $p\to 1$ for all $t>2\log R_\Lambda$.
\end{proof}

\section{A Positive Mass Theorem for the Polarized  \texorpdfstring{$p$}{p}-harmonic Mass}
\label{sec:polar_PMT}

In this section we establish one of the principal results of the present work, namely a positive mass theorem for the Polarized $p$-harmonic Mass. Before introducing its definition and formulating the main statement, it is convenient to specify the geometric and analytic framework that will be assumed from now on.

\begin{enumerate}
\item $(M,g)$ is a three-dimensional compact Riemannian manifold with smooth, compact boundary $\partial M$.
\smallskip
\item The cosmological constant $\Lambda$ is assumed to be strictly positive and provides a lower bound for the scalar curvature $\RRR$ of $(M,g)$:
\[
\RRR \,\geq\, 2\Lambda \,>\, 0 \, .
\]
\item For any point $x \in M \setminus \partial M$ and any exponent $1 < p < 3$, we consider the $p$-Green's function of the $p$-Laplacian with zero Dirichlet boundary condition and pole at $x$, that is, a distributional solution of the boundary value problem
\begin{equation}
\label{pb_p_Green}
\begin{dcases}
\Delta_p u \,=\, - 4\pi \delta_x & \text{in } M\,,
\\
\quad u \,=\, 0 & \text{on } \partial M.
\end{dcases}
\end{equation}
We refer the reader to Appendix~\ref{app:pharmonic}, where we summarize the key properties of $p$-harmonic functions used in this work and specify the functional setting in which the above problem is formulated, guaranteeing existence and uniqueness of $u$. When it is necessary to highlight the dependence of the solution on the pole $x$, on the exponent $p$, or on both, we will use the notations $u_x$, $u^{(p)}$, and $u_x^{(p)}$, respectively.
\smallskip
\item Whenever convenient for the analysis, we shall systematically perform the change of variable
\[
w \,=\, - (p-1)\log u \,,
\]
so that the new unknown $w$ satisfies
\begin{equation}
\label{eq:p-imcf}
\begin{dcases}
\mathrm{div}\big(e^{-w}|\nabla w|^{p-2}\nabla w\big) \,=\, 4\pi (p-1)^{p-1} \delta_x & \text{in } M\,,\\
\quad w(y)\longrightarrow +\infty & \text{as } y\to\partial M.
\end{dcases}
\end{equation}
In analogy with the notation adopted for the $p$-Green's functions, when we need to highlight the dependence of $w$ on the pole $x$, on the exponent $p$, or on both, we will write $w_x$, $w^{(p)}$, or $w_x^{(p)}$, respectively.
\smallskip
\item Concerning the structural coefficients that will appear in the forthcoming monotonicity formulas, once $\Lambda > 0$ and $1 < p < 3$ are fixed, we shall systematically adopt the functions
\[
t \longmapsto \mu(t)  
\qquad \text{and} \qquad 
t \longmapsto \lambda(t)
\]
selected in the previous section, namely the globally defined solutions of system~\eqref{structural} obtained in Theorem~\ref{thm:global}. In particular, the asymptotic expansions stated in Corollary~\ref{cor:behaviour_with_t} will be in force.
\end{enumerate}

\begin{definition}
[Polarized $p$-harmonic Mass]
\label{def:polmass}
Let $(M,g)$ be a compact three-dimensional Riemannian manifold with smooth boundary $\pa M$. For $1<p<3$, the {\em Polarized $p$-harmonic Mass} of $(M,g)$, with pole at $x \in M\setminus \pa M$, is given by
\begin{align}
\label{eq:polarized_total_mass}
\mmp_\Lambda(M,g,x)\,= \int_{-\infty}^{+\infty} \!\!\!\!\!e^{\lambda(\tau)}  \, \big(4\pi - \Lambda \, {\rm Per} (\Omega_\tau)\big)&  \, d\tau 
\, 
- \,\frac{R_\Lambda^{\frac{2}{p-1}}}{8 \pi}\,
\frac{\Gamma (\frac12) \, \Gamma(c_p)}{\Gamma(a_p+1)\Gamma(b_p+1)}
\, \int_{\pa M}\!\!\!|\na u|\,\HHH\,d\sigma 
\nonumber
\\
&
+ \, \frac{R_\Lambda^{\frac{5-p}{p-1}}}{16\pi}\, 
\frac{\Gamma(\frac12) \, \Gamma(c_p)}{\Gamma(a_p+\frac{3}{2})\Gamma(b_p+\frac{3}{2})} 
\, \int_{\pa M} \!\!\! |\na u|^2 d\sigma \, 
\end{align}
where $u=u_x^{(p)}$ is the $p$-Green's function with pole at $x$ and null Dirichlet boundary conditions, the set $\Omega_\tau = \Omega_{x, \tau}^{(p)}$ is given by 
\[
\Omega_\tau = \left\{ u > e^{-\tau/(p-1)} \right\} \, = \, \big\{ w < \tau \big\}
\]
the coefficients $a_p, b_p$ and $c_p$ are the ones defined in~\eqref{p-parameters}, and $R_\Lambda=\sqrt{3/\Lambda}$.
\end{definition}

Although in a somewhat indirect manner, one can verify that on the hemisphere equipped with the model metric
\[
g_\Lambda \,=\, \frac{dr\otimes dr}{1-\frac{\Lambda}{3}r^2}+r^2g_{\mathbb{S}^2}
\]
the Polarized $p$-harmonic Mass, computed with respect to the north pole $N$ (which, in the coordinate representation adopted here, corresponds to the locus $r=0$), vanishes:
\[
\mmp_\Lambda(\SSS^3_{+},g_\Lambda,N) \, = \, 0 \, . 
\]
Indeed, this follows from the fact that  $\frac d{dt}\mmp_\Lambda (x,t)=0$ on the model metric (cf. the proof of Theorem~\ref{thm:pos_pass_pol})).
In Theorem~\ref{thm:pos_pass_pol} below, we will show that if the scalar curvature of the manifold is bounded from below by $2\Lambda$, then the Polarized $p$-harmonic Mass is always nonnegative, and it vanishes if and only if the underlying manifold is the model hemisphere and the chosen pole coincides with the north pole.
Before we do so, let us prove, with the following proposition, that the Polarized $p$-harmonic Mass is always finite.

\begin{proposition}
\label{pro:finite_total_polarized_mass}
Let $(M,g)$ be a compact three-dimensional Riemannian manifold with smooth boundary $\pa M$.
For every $1<p<3$ and every pole $x\in M\setminus\pa M$, 
the Polarized $p$-harmonic Mass $\mmp_\Lambda(M,g,x)$ with pole at $x$ is finite.
\end{proposition}

\begin{proof}
Showing that $\mmp(M,g,x)$ is finite amounts to proving that the first integral in \eqref{eq:polarized_total_mass} -- the so‑called {\em bulk term} -- is finite, since the remaining two terms are manifestly finite. By the Coarea Formula, this integral can be rewritten as
\[
\int_{-\infty}^{+\infty} \!\!\!e^{\lambda(\tau)}\bigl(4\pi - \Lambda \, {\rm Per}(\Omega_\tau)\bigr)  \, d\tau 
\,=\, 4\pi\int_{-\infty}^{+\infty} \!\!e^{\lambda(\tau)}\,d\tau \;-\; \Lambda\int_M |\nabla w|\,e^{\lambda(w)}\,d\mu \, ,
\]
where $w$ is the solution to~\eqref{eq:p-imcf}. We first observe that 
the integral $\int_{-\infty}^{+\infty}
e^{\lambda(\tau)}\,d\tau$ 
is finite. Indeed, the function $\lambda$ given by Theorem~\ref{thm:global} is globally defined and continuous, and its integrability follows directly from the asymptotic expansions stated in Corollary~\ref{cor:behaviour_with_t}, 
namely
\[
e^{\lambda(\tau)} = \mathcal{O}\bigl(e^{\frac{\tau}{3-p}}\bigr)\quad \text{as }\tau\to -\infty \quad \hbox{and}
\quad
e^{\lambda(\tau)} = \mathcal{O}\bigl(e^{-\frac{\tau}{p-1}}\bigr)\quad \text{as }\tau\to +\infty.
\]
To treat the second summand, we claim that 
\begin{equation}
\label{stima_bulk}
\int_M \! |\na w|\, e^{\lambda(w)}d\mu
\,\leq\,
\left[4\pi \, (p-1)^{p-1} 
|M|^{p-1}\int_{-\infty}^{\infty}
\!\!\!\! \,e^{p\lambda(\tau) + \tau} \, d\tau
\right]^{\frac1p}.
\end{equation}
The thesis is then derived from the claim together with the observation that the function $\tau \mapsto e^{p\lambda(\tau)+ \tau}$ is also integrable, as 
\[
e^{p\lambda(\tau) + \tau} = \mathcal{O}\bigl(e^{\frac{3 \tau}{3-p}}\bigr)\quad \text{as }\tau\to -\infty \quad \hbox{and}
\quad
e^{p\lambda(\tau) + \tau} = \mathcal{O}\bigl(e^{-\frac{\tau}{p-1}}\bigr)\quad \text{as }\tau\to +\infty.
\]
The claimed estimate~\eqref{stima_bulk} follows from the following chain of inequalities:
\begin{align*}
\int_M|\na w| \, e^{\lambda(w)}d\mu\, \, 
&\,\leq\,
\left[\int_M|\na w|^p \, {e^{p \lambda(w)}} \,  d\mu\right]^{\!\frac1p}  |M|^{\frac{p-1}p}\\
&\,=\,
\left[\int_{-\infty}^{+\infty}
\!\!\!\!\!{e^{p\lambda(\tau) +\tau}} \left( e^{-\tau} \!
\int_{\Sigma_\tau}
|\na w|^{p-1}d\sigma \right)  d\tau \, 
\right]^{\!\frac1p}
|M|^{\frac{p-1}p}\\
&\,=\,
\left[4\pi \, (p-1)^{p-1}  \!\! \int_{-\infty}^{+\infty}
\!\!\!\!{e^{p\lambda(\tau) + \tau}} \, 
d\tau \, 
\right]^{\frac1p}
|M|^{\frac{p-1}p},
\end{align*}
Note that
in the first passage we have used H\"older inequality, in the second  identity we have applied the Coarea Formula, whereas the last passage comes from the identity
\[
e^{-t}\int_{\Sigma_t}|\na w|^{p-1}d\sigma\,=\,4\pi \,(p-1)^{p-1} \, ,
\]
holding for every $t \in \R$. 
\end{proof} 

\begin{remark}
\label{rem:unibulk}
For future convenience, note that the right hand side of~\eqref{stima_bulk} is independent on the pole, providing a uniform upper bound.
\end{remark}

\begin{theorem}
[Positive Mass Theorem for the Polarized $p$-harmonic Mass]
\label{thm:pos_pass_pol}
Let $(M,g)$ be a compact three-dimensional Riemannian manifold with smooth connected boundary $\pa M$,
whose scalar curvature $\RRR$ satisfies 
\[
\RRR \, \geq \, 2 \Lambda \, , 
\]
for some $\Lambda>0$. Assume that $H_2(M, \pa M; \Z) = \{0\}$. 
Then, for every $1<p<3$ and every $x\in M\setminus\pa M$, the Polarized $p$-harmonic Mass with pole at $x$ satisfies
\[
\mmp_{\Lambda}(M,g,x)\geq 0\,.
\]
Furthermore, if $\mmp_{\Lambda}(M,g,x)=0$, then $(M,g)$ is isometric to a round hemisphere with constant sectional curvatures equal to $\Lambda/3$ and $x$ coincides with the north pole. 
\end{theorem}

\begin{proof}
Let $u$ be a solution to~\eqref{pb_p_Green} and let $w$ be the corresponding solution to~\eqref{eq:p-imcf}. According to formula~\eqref{mp_gen}, we set, for every $t \in \R$
\begin{equation}
\mmp_\Lambda (x,t) \, = \int_{-\infty}^t \!\!\!\! e^{\lambda(\tau)} \big(4\pi - \Lambda|\Sigma_\tau|\big)  \, d\tau
\, - \,\,  e^{\lambda(t)} \!\int_{\Sigma_t} \!\! |\nabla w| \,  \big(  \HHH - \mu(t)|\nabla w|  \big) \, d\sigma \,,
\end{equation}
where $\Sigma_\tau = \{ w=\tau \}$, whereas $\lambda= \lambda_\Lambda^{\!(p)}$ and $\mu = \mu_\Lambda^{\!(p)}$ are the structural coefficients selected in Theorem~\ref{thm:global}. We now note that, for $S<0$ chosen sufficiently large and negative and $T>0$ chosen sufficiently large and positive, the level sets $\Sigma_S$ and $\Sigma_T$ are regular. This follows directly from the asymptotic expansion of $u$ near the pole $x$ (see Theorem~\ref{fund_growth}) and from the Hopf Lemma for $p$-harmonic functions~\cite{Tolksdorf1983}, applied to $u$ along the boundary $\pa M$. By repeatedly applying Theorem~\ref{GMF}, for instance to the domains in an exhaustion of the manifold given by Riemannian bands with regular boundary of the form $\{S < w < T\}$, we obtain
\[
0 \, = \, \lim_{S \to -\infty} \mmp_\Lambda (x,S) \, \leq \, \lim_{T \to +\infty} \mmp_\Lambda (x,T) \, .
\]
Indeed, as already observed in~\eqref{massa_nulla_polo}, the quantity $\mmp_\Lambda(x,S)$ tends to zero as $S \to -\infty$. We now claim that
\[
\lim_{T \to +\infty} \mmp_\Lambda (x,T) \, = \, \mmp_\Lambda (M,g,x) \, .
\]
For this purpose, we observe that on the level $\Sigma_T$ the identities hold
\[
|\na w|\,=\,(p-1)\frac{|\na u|}{u} \, = \,(p-1)|\na u|\, e^{\frac{T}{p-1}}\,.
\]
Combining these facts with the asymptotic expansions provided by Corollary~\ref{cor:behaviour_with_t}, we get
\begin{align*}
e^{\lambda(T)} |\na w|\,  \HHH \,  &= \, \frac{R_\Lambda^{\frac{2}{p-1}}}{8 \pi}\, \frac{ \Gamma(\tfrac12) \, \Gamma(c_p)}{\Gamma(a_p+1) \, \Gamma(b_p+1)} \,\,(1 +o(1)) \,,\quad\hbox{as } T\to+\infty\,, \\
e^{\lambda(T)} \mu(T) \,|\na w|^2 &= \, \frac{R_\Lambda^{\frac{5-p}{p-1}}}{16 \pi}\, \, \frac{\Gamma(\tfrac12) \, \Gamma(c_p)}{\Gamma(a_p+\frac{3}{2}) \, \Gamma(b_p+\frac{3}{2})} \, \,(1 +o(1)) \,,\quad\hbox{as } T\to+\infty\,,
\end{align*}
from which our claim easily follows.
We have thus shown that the Polarized $p$-harmonic Mass is always nonnegative. It remains to prove the rigidity statement.
For this purpose, assume that $\mmp_\Lambda(M,g,x)=0$ and observe, in view of Theorem~\ref{GMF}, that for every regular value $t \in \R$ we have $\mmp_\Lambda(x,t)=0$. As already remarked, all values below a certain threshold are regular. Denote by $T_0$ the minimum of the critical values of the function $w$, namely
\[
T_0 \, = \, \min\{ \, t \in \R \,  : \, |\na w|(y) = 0\,, \,  \hbox{for some} \, y \in \Sigma_t\} \, > -\infty \, ,
\]
and consider the open set $
\Omega_{T_0} = \{ w < T_0 \}$.
Since the function $t \mapsto \mmp_\Lambda(x,t)$ vanishes identically for $t<T_0$, its derivative also vanishes identically on the same half-line. From formula~\eqref{dermp} we then infer that $|\nabla w|$ is constant along the level sets of $w$ and hence can be regarded as a function of the coordinate $w$ alone. Furthermore, on each level set $\Sigma_t$ with $t<T_0$ the following equations are satisfied
\begin{equation}
\label{eq:rig}
|\mathring\hhh|\,=\,0\,,\qquad
\RRR=2\Lambda\,,
\qquad
\frac{\HHH}{2}\,=\,\alpha|\na w|\,.
\end{equation}
Using $w$ as a coordinate on the open set
$\Omega_{T_0}$ and noting that $\Omega_{T_0}$ is in fact diffeomorphic to the cylinder $(-\infty, T_0) \times \SSS^2$ -- since all level sets of $w$ are mutually diffeomorphic in this region and the level sets near the pole are spherical -- we can express the metric in the form
\begin{equation}
\label{eq:expr_g}
g\,=\,
\frac{dw\otimes dw}{|\nabla w|^2} +   \, g_{ij}(w,\vartheta) \, d\vartheta^i\otimes d\vartheta^j \, ,
\end{equation}
where $\{\vartheta^1, \vartheta^2\}$ denote local coordinates on $\SSS^2$. Exploiting the first and third equations in~\eqref{eq:rig}, we derive the relations
\[
\alpha|\nabla w|\, g_{ij}=\hhh_{ij}=\frac{\nabla_{\!i}\nabla_{\!j} \,w}{|\nabla w|}\,=\,-\frac{\Gamma^w_{ij}}{|\nabla w|}\,=\,
\frac{|\nabla w|}{2}\frac{\partial g_{ij} }{\partial w} \, .
\]
Consequently, the coefficients $g_{ij}$ of the metric satisfy on $\Omega_{T_0}$ the following first-order system of partial differential equations:
\begin{equation}
\label{gij0}
\frac{\partial g_{ij} }{\partial w}\,=\,2\alpha g_{ij}\,=\,
2\left[\frac{\partial}{\partial w}\log\big(r(w)\big)\right]g_{ij} \, ,
\end{equation}
where, in the second equality, we have invoked equation~\eqref{eq:dotr} and adopted the shorthand notation
$r(\cdot)$ for $r_\Lambda^{\!(p)} (\cdot)$. Since this system is completely decoupled, each equation can be integrated as an ordinary differential equation by treating the dependence on the angular variables $\vartheta$ as a fixed parameter. Fixing an arbitrary $w_0 \in (-\infty, T_0)$, we obtain
\[
g_{ij}(w,\vartheta)\,=\,
\frac{r^2(w)}{r^2(w_0)} \, g_{ij}(w_0,\vartheta) \, , \qquad \text{for } \, w \in (-\infty, T_0) \, .
\]
Next, employing again the third equation in~\eqref{eq:rig}, together with the identity 
\[
\nabla\nabla w (\nabla w, \nabla w) = |\nabla w|^3 \, \partial_w |\nabla w| \, ,
\]
we can express the $p$-Laplacian of $w$ as 
\[
\Delta_p w\,=\,
(p-1)|\nabla w|^{p-1}\frac{\partial |\nabla w|}{\partial w}+2\alpha|\nabla w|^p.
\]
Since on $\Omega_{T_0}$ equation~\eqref{eq:p-imcf} reduces to $\Delta_p w=|\nabla w|^p$, it follows that $|\nabla w|$ satisfies the ordinary differential equation
\[
\frac{\partial |\nabla w|}{\partial w}
\,=\,\left(\frac{1-2\alpha}{p-1}\right)|\nabla w| \, ,
\]
on the interval $(-\infty, T_0)$.
With the help of~\eqref{eq:dotr}, this can be rewritten as
\[
\frac{\pa}{\pa w}\log(|\na w|)\,=\,
\frac1{p-1}-\left(\frac2{p-1}\right)
\frac{\pa}{\pa w}\log\big(r(w)\big).
\]
Fixing an arbitrary $w_0 \in (-\infty, T_0)$, we thus obtain
\[
|\na w|(w)\,=\,
\left[
\frac{|\na w|(w_0)}{e^{\frac{w_0}{p-1}} \, \big(r(w_0)\big)^{\!-\frac2{p-1}}}
\right] \, 
e^{\frac w{p-1}} \, 
\big(r(w)\big)^{\!-\frac2{p-1}} \, , \qquad \text{for } \, w \in (-\infty, T_0) \, .
\]
To simplify this expression, we exploit the arbitrariness of $w_0$ and observe that the term in square brackets converges to $(p-1)$ as $w_0 \to -\infty$. This can be verified by means of the second identity in~\eqref{cose_modello}, together with the fact that $|\nabla w(w_0)|/|\nabla w_\Lambda(r(w_0))|\to1$ as $w_0\to-\infty$. Consequently, 
\[
|\nabla w|(w)\,=\,
(p-1) \, 
e^{\frac w{p-1}} \, 
\big(r(w)\big)^{\!-\frac2{p-1}} \, , \qquad \text{for } \, w \in (-\infty, T_0) \, .
\]
From this formula, which provides an explicit representation of $|\nabla w|$ in the region $\Omega_{T_0}$, it follows that the function $|\nabla w|$  is uniformly positive therein. Hence $T_0=+ \infty$, that is, all level sets of $w$ are regular, and our entire argument extends to the whole manifold.

Recalling the definition of the structural coefficient $\alpha$ in~\eqref{eq:alpha_t} and using relation~\eqref{eq:dotr}, which in the present setting yields $dr = \alpha r \, dw$, we obtain 
\begin{equation*}
\frac{dw}{|\nabla w|} \, = \, \frac{dw}{(p-1)\, e^{\frac{w}{p-1}} \, (r(w))^{\!-\frac{2}{p-1}}} \, = \,\frac{\alpha r \, dw}{\sqrt{1 - \frac{\Lambda}{3}r^2(w)}} \, = \, \frac{ dr}{\sqrt{1 - \frac{\Lambda}{3}r^2}} \, .
\end{equation*}
This identity shows that the coordinate $r= r(w(y))$ is asymptotic, as $y \to x$, to the distance function from the pole $x$. Since the metric is smooth at the pole, we infer that 
\[
\lim_{w_0 \to - \infty}\frac{g_{ij}(w_0,\vartheta)}{r^2(w_0)}
=g_{ij}^{\mathbb S^2}(\vartheta). 
\]
Thus we can refine expression~\eqref{eq:expr_g} and finally obtain 
\[
g
\,=\,
\frac{dr\otimes dr}{1-\frac\Lambda3r^2}
+
r^2g_{\mathbb S^2},
\]
namely, the standard spherical metric with constant sectional curvatures equal to $\Lambda/3$ on a hemisphere, with $r=0$ at the north pole and $r=R_\Lambda$ at the equator. Since $w\to+\infty$ as $r\to R_\Lambda$, we also conclude that the boundary of $M$ coincides with the equator of this hemisphere.
\end{proof}

\section{A Positive Mass Theorem for  the \texorpdfstring{$p$}{p}-harmonic Total Mass}
\label{sec:blowup}

Continuing to work within the framework established in the preceding section and, in particular, building upon the notion of Polarized $p$-harmonic Mass introduced in Definition~\ref{def:polmass}, we now introduce the concept of $p$-harmonic Total Mass.
\begin{definition}[
$p$-harmonic Total Mass]
\label{def:pmass}
Let $(M,g)$ be a compact three-dimensional Riemannian manifold with smooth boundary $\pa M$. For $1<p<3$, the {\em $p$-harmonic Total Mass} of $(M,g)$ is given by
\begin{equation}
\label{def:total_mass}
\mmp_\Lambda(M,g)\,=\,\inf_{x\in{M \setminus \pa M}} \mmp_\Lambda(M,g,x)\,, 
\end{equation}
where $\mmp(M,g,x)$ is the Polarized $p$-harmonic Mass with pole at $x \in M\setminus \pa M$, introduced in Definition~\ref{def:polmass}.
\end{definition}
For this new concept as well, we can establish a positive mass theorem, with the corresponding rigidity, and more precisely:
\begin{theorem}[Positive Mass Theorem for the $p$-harmonic Total Mass]
\label{thm:pos_pass_total}
Let $(M,g)$ be a compact three-dimensional Riemannian manifold with smooth connected minimal boundary $\pa M$,
whose scalar curvature $\RRR$ satisfies 
\[
\RRR \, \geq \, 2 \Lambda \, , 
\]
for some $\Lambda>0$. Assume that $H_2(M, \pa M; \Z) = \{0\}$. 
Then, for every $1<p<3$, the $p$-harmonic Total mass of $(M,g)$ satisfies
\[
\mmp_{\Lambda} (M,g) \, \geq  \, 0\,.
\]
Furthermore, if $\mmp_{\Lambda}(M,g)=0$, then $(M,g)$ is isometric to a round hemisphere with constant sectional curvatures equal to $\Lambda/3$. 
\end{theorem}

The most substantial and conceptually significant aspect of this theorem resides in the rigidity statement, since the non-negativity of the $p$-harmonic Mass is an immediate consequence of Definition~\ref{def:pmass} together with Theorem~\ref{thm:pos_pass_pol}. Our actual goal is to establish -- in Corollary~\ref{cor:minimum} below -- that the infimum in~\eqref{def:total_mass} is indeed attained, at one or more poles. At that point, the rigidity conclusion follows from the vanishing of the Polarized $p$-harmonic Mass centered at one of the optimal poles, in conjunction with the corresponding rigidity statement already proved in Theorem~\ref{thm:pos_pass_pol}.

In order to show that the infimum in~\eqref{def:total_mass} is in fact a minimum, we will prove that the Polarized $p$-harmonic Mass Function, i.e., the Polarized $p$-harmonic Mass, regarded as a function of the pole,
\begin{equation}
\label{eq:poletomass}
M \setminus \pa M \ni x \,\longmapsto\, \mmp_\Lambda(M,g,x),
\end{equation}
is continuous and proper. Continuity will be established in Proposition~\ref{pole_continuity}, whereas properness will be derived in Proposition~\ref{prop_PB}, under the convenient assumption that the boundary is minimal.

\begin{remark}
We use the minimal boundary assumption mainly to simplify the exposition. With some additional work this assumption can be removed, but we postpone this to future developments of the notion of $p$-harmonic mass, when its necessity will be conceptually clearer. In particular, removing the minimal boundary assumption will be essential for studying a localized version of the present theory, leading to a notion of $p$-harmonic quasi-local mass defined on any sufficiently regular open subset of a manifold with scalar curvature bounded from below and satisfying Bartnik’s axioms from his seminal paper~\cite{Bartnik_quasilocal} (see also the recent survey~\cite{McC_2024}).
\end{remark}

To prove that the Polarized $p$-harmonic Mass depends continuously on the pole~\eqref{eq:poletomass}, we next establish continuity of the $p$-Green's function with respect to the pole in the $\mathscr C^{1,\beta}$ topology.
While this type of estimate is essentially classical, there appears to be neither an explicit statement nor a complete proof available in the literature, so we provide both in the following theorem.
\begin{theorem}
\label{C1-dependece_Green}
Let $(M,g)$ be a compact three-dimensional Riemannian manifold with smooth boundary $\pa M$ and let $1<p<3$.
Given $x\in M\setminus\pa M$ and a sequence of points $\{x_j\}_{j \in \N} \subset M\setminus \pa M$, let $u_x$ and $\{u_{x_j}\}_{j \in \N}$ be the $p$-Green's functions with pole at $x$ and $x_j$'s, respectively, satisfying null Dirichlet boundary conditions, as prescribed in~\eqref{pb_p_Green}. For every compact set $K\subset M\setminus\{x\}$, we have that
\[
x_j \xrightarrow[j \to +\infty]{} x \, \quad    \Rightarrow   \quad u_{x_j}\xrightarrow[j \to+\infty]{} u_x,
\qquad\mbox{in}\quad{\mathscr C}^{1,\beta}(K),
\]
for some $0<\beta<1$. 
\end{theorem}
\begin{remark}
Note that the compact set $K$ in the statement may include (parts of) the boundary of $M$. The restriction to dimension $3$ is only for uniformity of exposition, since it is clear from the proof that the result holds in any dimension $n$ and for all $1<p<n$.
\end{remark}
\begin{proof}
As a first step, we prove that the sequence $\{ u_{x_j}\}_{j \in \mathbb{N}}$ admits a subsequence that converges, with respect to the $\mathscr{C}^1$-topology on the compact subsets of $M\setminus \{ x\}$, to a $p$-harmonic function $v$, that is vanishing at $\pa M$. 
Let $(y^1,\hdots,y^3)$ be coordinates centered at $x$ and,
for every $n$ larger than some suitable threshold $n_0 >0$, consider the following exhaustion of $ M \setminus \{x\}$  by compact sets
\[
K_n\,:=\,M
\setminus\left\{ \sqrt{(y^1)^2+ (y^2)^2+(y^3)^2}<\frac1{n} \right\}.
\]
Theorem~\ref{fund_growth}, coupled with the Comparison Principle for $p$-harmonic functions
(see, e.g.~ \cite[Theorem 3.5.1]{Puc_Ser_book}),
implies that for every $n$ there exists $M_n$ such that
\[
\|u_{x_j}\|_{L^{\infty}(K_n)}
\,\leq\, M_n.
\]
Hence, by Theorem~\ref{thm_est_int_estest}-\emph{(ii)},
we deduce that
\[
\|u_{x_j}\|_{{\mathscr C}^{1,\beta}(K_n)}
\,\leq\, C_n,
\]
for some constants  $C_n>0$ and $0<\beta<1$ (possibly depending on $K_n$ and $M_n$, but independent on $j$). Ascoli-Arzel\`a Theorem thus  yields the existence of a subsequence $\{u_j^{(n_0)}\}_{j\geq n_0}\subseteq\{u_{x_j}\}_{j\in\mathbb N}$ such that
\[
u_j^{(n_0)}\longrightarrow v^{(n_0)},\quad\mbox{as }j\to +\infty,\quad\mbox{in }\ {\mathscr C}^1(K_{n_0}).
\]
Inductively, for every $n>n_0$, there exists
$\{u_j^{(n)}\}_{j\geq n}\subseteq\{u_j^{(n-1)}\}_{j\geq n-1}$ such that
\[
u_j^{(n)}\longrightarrow v^{(n)},\quad\mbox{as }j\to +\infty,\quad\mbox{in }\ {\mathscr C}^1(K_{n}).
\]
Note that $v^{(m)}=v^{(n)}$ on $K_n$, for every $m \geq n$. In particular, the function $v$ defined in $M\setminus\{x\}$ as
\[
v(y)\,=\,v^{(n)}(y),
\qquad\mbox{if}\quad y\in K_n,
\]
is well-defined.
Also, the sequence
$\{u_n^{(n)}\}_{n\geq n_0}
$
converges to $v$ on each compact set $K\subset M\setminus\{x\}$. So far we have proven that the sequence $\{ u_{x_j}\}_{j \in \mathbb{N}}$ sub-converges to the function $v$ in $\mathscr{C}^1_{loc}(M\setminus \{x\})$. In particular, we have that $v=0$ on $\pa M$. To see that $v$ is actually $p$-harmonic in $M\setminus\{x\}$ we recall that for every test function $\varphi \in \mathscr{C}^\infty_0(M\setminus \{x\}) $ the functions $u_{x_j}$'s satisfy the identities 
\[
\int_M \left\langle |\na u_{x_j}|^{p-2}  \na u_{x_j}  \, \big| \, \na \varphi \right\rangle \, d\mu_g \, = \, 4 \pi \varphi(x_j) \, .
\]
Using the $\mathscr{C}^1_{loc}(M\setminus \{x\})$ convergence of the subsequence, it is straightforward to deduce that 
\[
\int_M \left\langle |\na v|^{p-2}  \na v  \, \big| \,  \na \varphi \right\rangle \, d\mu_g \, = \, 0 \, ,
\]
as $\varphi(x_j) = 0$ for large enough $j \in \mathbb{N}$.

In the remaining part of the proof we are going to show that indeed $v=u_x$. As a first step, we show that $v$ has the same asymptotics as $u_x$ at the pole $x$. To see this, we invoke Theorem~\ref{fund_growth} and observe that for every $\eta>0$ there exist $r_\eta>0$ and $j_\eta \in \mathbb{N}$ such that, if $0<|y-x_j|\leq r_{\eta}$, then 
\[
\Big(\frac{p-1}{3-p}-\eta\Big)|y-x_j|^{-\frac{3-p}{p-1}}
\,\leq\,
u_{x_j}(y)
\,\leq\,
\Big(\frac{p-1}{3-p}+\eta\Big)|y-x_j|^{-\frac{3-p}{p-1}},
\quad\mbox{for every}
\,\,
 j\geq j_{\eta}.
\]
This follows from that fact that the choice of the radius in Theroem~\ref{fund_growth} ultimately depends on the coefficients of the equation; hence, it is clear that one can chose the same radius, provided the poles $x_j$ are close enough to $x$. Next, we cover the punctured ball where $0<|y-x|\leq r_{\eta}$ by means of the sequence of compact annuli
\[
A_k \, = \, \big\{ y \in M \, : \,  2^{-k-1} r_\eta \leq |y-x| \leq 2^{-k} r_\eta \big\} \, , \quad k \in \mathbb{N} \, .
\]
Combining the $\mathscr{C}^1$-convergence on the compact subset of $M \setminus \{x\}$ with the uniform growth described above, we deduce that for every $\eta>0$ and every $k\in \mathbb{N}$, the function $v$ satisfies
\[
\Big(\frac{p-1}{3-p}-\eta\Big)|y-x|^{-\frac{3-p}{p-1}}
\,\leq\,
v(y)
\,\leq\,
\Big(\frac{p-1}{3-p}+\eta\Big)|y-x|^{-\frac{3-p}{p-1}} \,,  \quad y 
\in A_k \,.
\]
Since this fact holds on every annulus $A_k$, $k \in \mathbb{N}$, the same estimate is valid in the whole punctured ball of radius $r_\eta$.
As $u_x$ obeys the very same growth condition, it is easy to see that, for every $\eta>0$, it holds
\[
\left( 1- \frac{2\eta}{\frac{p-1}{3-p}+\eta}\right) u_x(y) \leq \, v(y) \, \leq \, \left( 1+ \frac{2\eta}{\frac{p-1}{3-p}-\eta}\right) u_x(y) \, ,
\]
provided $0<|y-x| \leq r_\eta$.
Applying the Maximum Principle in $\Omega_\eta:=\{y\in M:|y-x|\geq r_\eta/2\}$, and using the fact that $v=u_x=0$ on $\pa M$, we deduce that
the same estimate holds throughout $\Omega_\eta$.
Since $r_\eta\to0$ as $\eta\to0$,
we conclude that
$v=u_x$ in $M\setminus\{x\}$.
\end{proof}
We are now ready to establish, with the following proposition, the continuity of the Polarized $p$-harmonic Mass with respect to the pole.
\begin{proposition}
\label{pole_continuity}
Let $(M,g)$ be a compact three-dimensional Riemannian manifold with smooth boundary $\pa M$, and let $1<p<3$.
Then the mapping  
$$
M\setminus\pa M\ni x \longmapsto \mmp_\Lambda(M,g,x) \, ,
$$ 
where $\mmp_\Lambda(M,g,x)$ denotes the Polarized $p$-harmonic Mass with pole at $x$ as defined in~\eqref{eq:polarized_total_mass}, is a continuous function.
\end{proposition}

\begin{proof}
Fix $x\in M\setminus\pa M$ and let $\{x_j\}_{j \in \N}$ be a sequence in $M$ converging to $x$. We want to prove that 
\[
\mmp_\Lambda(M,g,x_j)\longrightarrow \mmp_\Lambda(M, g,x),\qquad\mbox{as}
\quad j\to +\infty.
\]
Looking at the expression expression \eqref{eq:polarized_total_mass}, let us first observe that the term $\int_{-\infty}^{+\infty}e^{\lambda(\tau)}d\tau$ is independent of the pole $x$. Thus, since
\[
\int_{-\infty}^{+\infty} \!\!\!\! |\Sigma_{x,\tau}| \, e^{\lambda(\tau)} \, d\tau
\,=\,
\int_M|\na w_x|e^{\lambda(w_x)} \, d\mu \, ,
\]
proving continuity is equivalent to proving that
\begin{equation}
\label{cont_1}
\int_{\pa M}|\na u_{x_j}|^2 d\sigma
\,\longrightarrow\,
\int_{\pa M}|\na u_x|^2 d\sigma\,,
\quad 
\int_{\pa M}|\na u_{x_j}|\,\HHH\,d\sigma
\,\longrightarrow\,
\int_{\pa M}|\na u_x|\,\HHH\,d\sigma,
\qquad\mbox{as}\quad j\to +\infty,
\end{equation}
and that
\begin{equation}
\label{cont_2}
\int_M|\na w_{x_j}|e^{\lambda(w_{x_j})} d\mu
\,\longrightarrow\,
\int_M|\na w_x|e^{\lambda(w_x)} d\mu \, ,
\qquad\mbox{as}\quad j\to + \infty.
\end{equation}
Theorem \ref{C1-dependece_Green} gives in particular that
\[
u_{x_j}\longrightarrow u_x,
\qquad\mbox{in}\quad{\mathscr C}^{1,\beta}(\pa M),
\]
which directly implies \eqref{cont_1}. 
Observe that Theorem \ref{C1-dependece_Green}
also implies that
\[
w_{x_j}\longrightarrow w_x,
\qquad\mbox{in}\quad{\mathscr C}^{1,\beta}(K),
\]
for every compact set $K\subset M\setminus\{x\}$, as it can be readily checked.
Now, from \eqref{eq:naw_near_pole} and Corollary \ref{cor:behaviour_with_t}, we get
\[
|\na w_x(y)| \, e^{\lambda(w_x(y))}
\longrightarrow\frac1{8\pi},
\qquad\mbox{as}\quad
y\to x.
\]
Therefore, for every $\ep>0$ arbitrarily small, there exist
$r_{\ep}>0$ and $j_\ep\in\N$ such that
\[
\int_{B(x,r_\ep)}|\na w_x|e^{\lambda(w_x)} d\mu
\,\leq\,\ep,
\qquad\int_{B(x_j,r_\ep)}|\na w_{x_j}|e^{\lambda(w_{x_j})} d\mu
\,\leq\,\ep,
\qquad\mbox{for every}\quad j\geq j_\ep.
\]
Hence,
\begin{multline*}
\left|
\int_M|\na w_{x_j}|e^{\lambda(w_{x_j})} d\mu
\,-\,
\int_M|\na w_x|e^{\lambda(w_x)} d\mu\right|\\
\leq
2\ep
+
\int_{M\setminus B(x,r_\ep/2)}
\left| \, 
|\na w_{x_j}|e^{\lambda(w_{x_j})}\,-\,
|\na w_x|e^{\lambda(w_x)}
\right| \, d\mu.
\end{multline*}
Finally, the last term on the right-hand side of the above inequality converges to $0$ as $j\to+\infty$, again by Theorem~\ref{C1-dependece_Green}. Passing to the limit as $\ep\to0$ then concludes the proof of \eqref{cont_2} and hence of the proposition.
\end{proof}
We are now in the position to establish that the Polarized $p$-harmonic Mass is proper as a function of its pole.
\begin{proposition}
\label{prop_PB}
Let $(M,g)$ be a compact three-dimensional Riemannian manifold with smooth minimal boundary $\pa M$.
For every $1<p<3$ and every $x\in M\setminus\pa M$, 
let $\mmp_\Lambda(M,g,x)$ be the Polarized $p$-harmonic Mass with pole at $x$, according to Definition~\ref{def:polmass}.
Then,
\[
\liminf_{x\to\pa M}
\mmp_\Lambda(M,g,x)
\,=\,
+\infty.
\]
\end{proposition}
\begin{proof}
Fix a point $x_0 \in \partial M$,
and let $\nu$ be the unit normal vector to $\partial M$ at $x_0$, pointing toward the interior of $M$.
It is evident that establishing the desired statement is equivalent to proving that
\begin{equation}
\label{BU_ep}
\liminf_{\ep\to0}
\mmp_\Lambda(M, g,x_\ep)
\,=\,
+\infty,
\qquad
x_\ep:=
{\rm exp}_{x_0}(\ep\nu).
\end{equation}
We claim that
\begin{equation}
\label{eq_claim}
\liminf_{\ep\to0}
\int_{\pa M}|\na u_{x_\ep}|^2d\sigma
\,=\,
+\infty.
\end{equation}
Before proving the claim, let first note that it implies \eqref{BU_ep}. Indeed, by the formula~\eqref{eq:polarized_total_mass} for the Polarized $p$-harmonic Mass, the minimality of $\pa M$ and the lower bound for the bulk term~\eqref{stima_bulk} -- together with the subsequent Remark~\ref{rem:unibulk} --
we have that
\begin{align*}
\mmp(M,g,x) 
\,\geq\,& \,\,\, 4\pi\int_{-\infty}^{\infty}\!\!\!\!
e^{\lambda(\tau)}d\tau \, - \, \Lambda \left[4\pi \, (p-1)^{p-1} 
|M|^{p-1}\int_{-\infty}^{\infty}
\!\!\!\! \,e^{p\lambda(\tau) + \tau} \, d\tau
\right]^{\frac1p} \\
&
+ \, \frac{R_\Lambda^{\frac{5-p}{p-1}}}{16\pi}\, 
\frac{\Gamma(\frac12) \, \Gamma(c_p)}{\Gamma(a_p+\frac{3}{2})\Gamma(b_p+\frac{3}{2})} 
\, \int_{\pa M} \!\!\! |\na u_x|^2 d\sigma \, 
\end{align*}
As both integrals 
\[\int_{-\infty}^{\infty}\!\!\!\!
e^{\lambda(\tau)}d\tau \, \quad \hbox{and} \quad  \int_{-\infty}^{\infty}
\!\!\!\! \,e^{p\lambda(\tau) + \tau} \, d\tau
\]
are finite, it is clear that the claim implies thesis. To establish claim~\eqref{eq_claim} we argue by contradiction, assuming that the infimum limit is finite. Then, there exists a constant $K>0$ such that
\begin{equation}
\label{eq:L2boud}
\int_{\pa M}|\na u_{x_\ep}|^2d\sigma
\,\leq\,
K,
\end{equation}
provided $\ep$ is chosen along a minimizing sequence. To avoid overburdening the notation with the double indexing that would result from considering the sequence of poles $\{x_{\ep_k}\}_{k \in \N}$, we shall, for the sake of simplicity, assume that claim~\eqref{eq_claim} is violated along an entire sufficiently small segment of the curve $\ep \mapsto x_\ep$. Consequently, the upper bound~\eqref{eq:L2boud} is taken to hold for all sufficiently small values of $\ep$.
It will then be apparent to the reader that the subsequent argument can be straightforwardly adapted to the case in which~\eqref{eq:L2boud} is valid only along a sequence.
To obtain a contradiction, we proceed by means of a blow-up analysis. For each sufficiently small $\ep > 0$, we consider the following rescaling of the metric and of the associated functions:
\[
g_\ep
\,:=\,
\frac1{\ep^2}g,
\qquad\qquad\quad
v_\ep\,:=\,
\ep^{\frac{3-p}{p-1}}u_{x_\ep}.
\]
It can be readily checked that
\[
d_{g_\ep}(x_\ep,x_0)\,=\,1.
\]
Moreover,
$v_{\ep}$ is a $p$-Green's function for the metric $g_\ep$ with pole at $x_{\ep}$ vanishing on $\pa M$. Namely,
\[
\int_M|\na v_\ep|^{p-2}\langle\na v_\ep\big|\na\varphi\rangle d\mu_\ep\,=\,4\pi\varphi(x_\ep),
\quad\mbox{for every}\  \varphi\in\mathscr C_0^{\infty}(M),
\quad
\mbox{and}
\quad
v_\ep\,=\,0
\quad\mbox{on}\ \pa M,
\]
where the norm and the scalar product is meant to be computed with respect to $g_\ep$.
The contradiction assumption then translates into
\begin{equation}
\label{grad_ep}
\int_{\pa M}|\na v_{\ep}|^2d\sigma
\,\leq\,
K\ep^{2\left(\frac{3-p}{p-1}\right)},
\qquad\mbox{for every small}\ \ep>0.
\end{equation}
By extracting a subsequence from $\{v_{\ep}\}_{\ep>0}$ and by a suitable change of variable, we are going to obtain a
\emph{nontrivial} $p$-harmonic function $v$ defined in a upper-half ball $B_R^+(0)$ of $\R^3$, vanishing at the flat portion of the boundary. 
The contradiction assumption further implies that the normal derivative of $v$ also vanishes there. 
The contradiction stems from the fact that, due to the Strong Maximum Principle for $p$-harmonic functions (see, e.g.~\cite{Tolksdorf1983}), 
$v>0$ in $B_R^+(0)$, which in turn gives $\pa v/\pa\nu>0$ on the flat part of the boundary, by the Hopf Lemma.

Using Fermi coordinates around the
geodesic 
$\gamma(t):={\rm exp}_{x_0}(t\nu)$, the metric $g$ can be written as
\[
g\,=\,
dr\otimes dr
+
g_{ij}(r,\vartheta)d\vartheta^i\otimes d\vartheta^j,
\]
where $r$ is the $g$-distance from $\pa M$
and
$\{\vartheta^1,\vartheta^2\}$ are $g_{|_{\pa M}}$-normal coordinates on $\pa M$ centered around $x_0$.
We recall that, by virtue of this choice, we have that
\[
g_{ij}(r,\vartheta)
\,=\,
\delta_{ij}+\eta_{ij}(r,\vartheta),
\qquad\mbox{with}
\quad 
\eta_{ij}(r,\vartheta)\,=\,
\mathcal O_1(\sqrt{r^2+|\vartheta|^2}).
\]
Observe that, by construction,
\[
r(x_\ep)\,=\,\ep,
\qquad\qquad
\vartheta^j(x_\ep)\,=\,0,
\quad j=1,2.
\]
To fix ideas, we can suppose w.l.o.g. that
$(r,\vartheta^1,\vartheta^2)\in[0,2]\times[-1,1]\times[-1,1]$.
Using the rescaled coordinates
\[
s\,:=\,
\frac r\ep,
\qquad\qquad
\varphi^j\,:=\,\frac{\vartheta^j}\ep,
\quad j=1,2,
\]
the metric $g_\ep$ rewrites as
\begin{equation}
\label{metrica_g_ep}
g_\ep\,=\,
ds\otimes ds
+\Big[\delta_{ij}+\eta_{ij}(\ep s,\ep\varphi)\Big]
d\varphi^i\otimes d\varphi^j.
\end{equation}
In particular,  there exist $0<\alpha<A$ such that
\[
\alpha|\xi|^2
\,\leq\, 
\delta_{ij}+\eta_{ij}(\ep s,\ep\varphi)
\xi^i\xi^j
\,\leq\,
A|\xi|^2,
\qquad
(s,\varphi)\in
[0,2]\times[-1,1]
\times[-1,1]=:Q,\quad\xi\in\R^3.
\]
This implies that, in turn, the structural constants of the PDE of interest can be chosen uniformly in 
the compact region $Q$.
Also, note that
\[
\big(s,\varphi^1,\varphi^2\big) (x_\ep)
\,=\,
(1,0,0)\,=\,:=q
\qquad
\mbox{and}
\qquad
\big(s,\varphi^1,\varphi^2\big) (x_0)
\,=\,
(0,0,0).
\]
Hence, in the $(s,\varphi)$-coordinates, all the functions $v_\ep$ have the same pole, at $(1,0,0)$.
Theorem~\ref{fund_growth} implies that,
fixed $\eta_0>0$, there exists $1/2>\rho_0(\eta_0,n,p,\alpha,A)>0$ such that, for every
$\rho_0/2
\leq
\rho=\sqrt{|s-1|^2+|\varphi|^2}
\leq
\rho_0,
$
it holds
\begin{equation}
\label{from_growth}
0
\,<\,
L_0
\,\leq\,
\left(\frac{3-p}{p-1}-\eta_0\right)
\rho^{-\frac{3-p}{p-1}}
\,\leq\,
v_\ep(s,\varphi)
\,\leq\,
\left(\frac{3-p}{p-1}+\eta_0\right)
\rho^{-\frac{3-p}{p-1}}
\,\leq\,
M_0.
\end{equation}
Exploiting the lower bound and the so called Harnack chain, we are going to prove that
the functions $v_\ep$ are uniformly bounded away from $0$ at $(1/4,0,0)$. Indeed, 
setting $\sigma_0:=\min\{\rho_0/4,1/12\}$ and 
\[
B^{(0)}\,:=\,
\left\{(s,\varphi):
\sqrt{\left|s-
\left(1-3\rho_0/4\right)
\right|^2+|\varphi|^2}<\sigma_0
\right\},
\]
from \eqref{from_growth}
we get in particular that
${v_\ep}_{|_{B^{(0)}}}
\geq L_0>0$.
Also, there exists a 
finite number $N_0\in\mathbb N$ of balls $B^{(1)},\hdots,B^{(N_0)}$ of radius $\sigma_0$, pairwise intersecting each other and connecting 
$B^{(0)}$
to $(1/4,0,0)$.
Hence, applying Theorem \ref{thm_Har} on the ball $B^{(1)}$ gives
\[
L_0\,\leq\,
\inf_{B^{(0)}}v_\ep
\,\leq\,
\sup_{B^{(1)}}v_\ep
\,\leq\,
C_H\inf_{B^{(1)}}v_\ep.
\]
In turn, applying applying Theorem \ref{thm_Har} $(N_0-1)$ more times yields
\begin{equation}
\label{uniform_bd_below}
v_\ep(1/4,0,0)\geq L_0/(C_H^{N_0})>0.
\end{equation}
Now, the upper bound in \eqref{from_growth} coupled with the Comparison Principle
for $p$-harmonic functions (see, e.g. ~\cite[Theorem 3.5.1]{Puc_Ser_book})
implies that the functions $v_\ep$
are uniformly bounded in the complement of suitably small coordinate balls
centered at $x_\ep$, namely
\[
v_\ep\,\leq\,M_0,
\qquad\mbox{in}\quad
M\setminus\left\{(r,\vartheta):
\sqrt{\left|r-\ep
\right|^2+
|\vartheta|^2}<\ep\rho_0
\right\}.
\]
In turn, Theorem \ref{thm_est_int_estest}-\emph{(ii)} implies
that both the functions $v_\ep$ and their gradients are uniformly bounded, up to the boundary. More precisely, we have that
\[
\|v_\ep\|_{\mathscr C^1(\overline{B_R^+})}
\,\leq\,
C,
\qquad\mbox{for every }\ep>0,
\]
where $1/4<R<(1-\rho_0)/2$ and
$B_R^+=\{s^2+|\varphi|^2<R^2,\,s>0\}$.
By Ascoli-Arzel\`a Theorem we can then extract from $\{v_\ep\}_{\ep>0}$ a subsequence
converging on $\overline{B_R^+}$
to some function $v$, with respect to the $\mathscr C^1$-norm.
Noting that the metric in  \eqref{metrica_g_ep}
converges to $g_{\R^3}=ds\otimes ds+\delta_{ij}d\varphi^i\otimes d\varphi^j$, as $\ep\to0$, and using \eqref{grad_ep},
we have that the limiting function $v$ satisfies
\[
\Delta_p^{\R^3}v\,=\,0\quad\mbox{in}\quad B_R^+,
\qquad\quad
v=0=|\na v|
\quad\mbox{on}\quad\pa B_R^+\cap\pa M.
\]
At the same time, from \eqref{uniform_bd_below} we get that $v(1/4,0,0)>0$: this fact gives the desired contradiction, arguing as explained above.   
\end{proof}

A straightforward application of Proposition \ref{pole_continuity} 
and of Proposition \ref{prop_PB} yields the following result. 

\begin{corollary}
\label{cor:minimum}
Let $(M,g)$ be a compact three-dimensional Riemannian manifold with smooth minimal boundary $\pa M$. For every $1<p<3$, the Polarized $p$-harmonic Mass Function~\eqref{eq:poletomass} attains its minimum at some interior point. In other words, if $\mmp_\Lambda(M,g)$ is the $p$-harmonic Total Mass introduced in Definition~\ref{def:polmass}, then
\[
\mmp_\Lambda(M,g) \, = \, \mmp_\Lambda(M,g,x_0) 
\]
for some $x_0\in M\setminus \pa M$.
\end{corollary}
It would be very interesting to understand under which additional geometric hypotheses -- beyond the connectedness of the boundary of \(M\) -- one can obtain a precise characterization of the locus of optimal poles, namely the set of points at which the Polarized $p$-harmonic Mass Function attains its minimum. From an intuitive standpoint, such points should lie sufficiently far from \(\partial M\), which might suggest the existence of a relationship with the cut locus of the boundary, \(\mathrm{Cut}(\partial M)\).  
Another possibility, suggested by the case of the static metrics, is that the locus of optimal poles may be related to the set of points at which a lapse-type function attains its maximum; in this context, such a function can be identified with the first eigenfunction of the Laplace operator subject to homogeneous Dirichlet boundary conditions. 

\section{Positive Mass Theorem and Penrose Inequality for the \texorpdfstring{$1$}{1}-harmonic Mass}
\label{sec:harm_1}

In this section, we prove two fundamental results concerning the $1$-harmonic Mass,
namely Theorem~\ref{thm_eich} and Theorem~\ref{thm:RPI-1harm} from the introduction (see Subsection~\ref{sub:1harm}). For the reader’s convenience, we restate these theorems below, as Theorem~\ref{thm_eich_7} and Theorem~\ref{thm:RPI-1harm_7}, respectively.

An important tool in the proofs of the aforementioned theorems will be the IMCF, either emanating from a pole or from a black-hole type horizon,  depending on the geometric setting. A key requirement will be to ensure that the level sets generated by the flow remain connected throughout the entire evolution. This is guaranteed by the following topological lemma. 

\begin{lemma}
\label{lem:topol}
Let $(M,g)$ be a compact, three-dimensional Riemannian manifold, with compact minimal boundary $\partial M$. Suppose that there is a connected component of $\partial M$, that we denote by $\pa M^+$, that is simply connected and unstable.
Finally, suppose that there are no closed minimal surfaces in $M\setminus \pa M$. Then $M$ is diffeomorphic to $\mathbb{S}^3$ minus some balls.
\end{lemma}

\begin{remark}
\label{rmk:H_2}
We will be mostly interested in the case where $\partial M \setminus \partial M^+$ is either empty or connected; in the latter case, we will denote by $\partial M^-$ the other connected component of $\partial M$. Under this assumption, as a direct consequence of the previous lemma we obtain in particular that $H_2(M,\partial M^+;\mathbb{Z})$ is trivial, and likewise $H_2(M,\partial M^-;\mathbb{Z})$ is trivial whenever $\partial M^-$ is nonempty. It is well known that these conditions are sufficient to ensure that the leaves of the IMCF remain connected for all times.
\end{remark}

\begin{proof}
The proof of this lemma follows~\cite[Lemma~4.1-(ii)]{HI}, but we give some details on how to adapt the proof to the case at hand. The first step of the proof consists in showing that $M$ is simply connected.  To this end, let $\widetilde M$ be the universal covering of $M$. The restriction of this covering to $\pa M^+$ gives a covering of $\pa M^+$. Since $\pa M^+$ is simply connected (equivalently, $\pa M^+$ is diffeomorphic to a sphere), this covering is a disjoint union of spheres. Thus, if $M$ is not simply connected, there are at least two boundary components $\Sigma_1,\Sigma_2$ of $\widetilde M$ that project onto $\pa M^+$. Since $\pa M^+$ is an unstable minimal surface, so are $\Sigma_1,\Sigma_2$. By minimizing the area among surfaces that separate $\Sigma_1$ and $\Sigma_2$, we obtain a stable minimal surface $N$. 

No connected component of $N$ can be contained inside $\pa\widetilde M$ (it cannot be contained inside $\Sigma_1,\Sigma_2$ because they are unstable, and it cannot be contained inside any other boundary component because it would not contribute in separating $\Sigma_1,\Sigma_2$). In fact, no connected component of $N$ can even contain a portion of the boundaries: if it does, it must intersect them transversely otherwise by the maximum principle it would coincide with the whole boundary component; on the other hand, if it intersects transversely then the angle can be smoothed out to decrease the area. 
Therefore $N$ is entirely contained in the interior of $\widetilde M$. Projecting onto $M$, we obtain a minimal surface contained in the interior of $M$. Stability and area minimality force the surface to be smooth, hence we have produced a smooth minimal surface inside $M$, against our hypothesis. 
This proves that $M$ is simply connected. 

We then prove that the boundary components of $M$ are diffeomorphic to spheres. 
By the half lives half dies principle (see~\cite[Lemma~3.5]{Hatcher_3M}) and the exact sequence of the pair, we conclude that the kernel of the inclusion map $i_*:H_1(\pa M)\to H_1(M)$ has rank equal to one half the rank of $H_1(\pa M)$.
Since $M$ is simply connected, and thus $H_1(M)$
is trivial, the kernel of $i_*$ coincides with the whole $H_1(\pa M)$. Thus, we conclude that the rank of $H_1(\pa M)$ must be equal to zero, and this can only happen if $\pa M$ is a collection of spheres, as wished.

Since every boundary component is diffeomorphic to a sphere, we can fill all boundary components with standard balls, obtaining a compact simply connected manifold without boundary.
The proof of the Poincar\'e Conjecture then allows us to conclude that such a manifold must be diffeomorphic to $\mathbb{S}^3$.
\end{proof}

We are now in a position to proceed with the proof of the Positive Mass Theorem and the Riemannian Penrose Inequality for the $1$-harmonic Mass.

\begin{theorem}[Positive Mass Theorem for the Polarized $1$-harmonic Mass]
\label{thm_eich_7}
Let $(M,g)$ be a compact three-dimensional Riemannian manifold with smooth minimal boundary $\pa M$, whose scalar curvature satisfies
\[
\RRR \geq 2 \Lambda,
\]
for some $\Lambda>0$.
Assume that $\pa M$ is simply connected and unstable and that there are no closed minimal surfaces in $M\setminus \pa M$. Then, for every $x \in M\setminus \pa M$, we have
\[
\mathfrak{m}_\Lambda^{(1)} (M, g, x) \, \geq \, 0,
\]
with equality if and only if $(M,g)$ is isometric to a round hemisphere of constant sectional curvatures equal to $\Lambda/3$.
\end{theorem}

\begin{proof}
We first recall from~\eqref{eq:1harmass} that the polarized $1$-harmonic Mass of $(M, g)$ with pole at $x \in M\setminus \pa M$ is given by    
\begin{equation*}
\mathfrak{m}_\Lambda^{(1)} (M, g, x) \, = \, \int_{-\infty}^{T^+_x}
\!\!\!\left(\frac{e^{\tau/2}}{16\pi}\right)  \, \big(4\pi - \Lambda \, {\rm Per} (\Omega_{\tau})\big)  \, d\tau  \, ,
\end{equation*}
where $T^+_x = \min \{ T_\Lambda ,T^*(x)\}$ and $\Omega_{\tau} = \Omega_{x,\tau}^{(1)} = \{ w < \tau \}$. Here $w = w_x$ is the IMCF with outer obstacle emanating from $x$, namely the weak solution of
\begin{equation}
\label{eq:obs-imcf_7}
\begin{dcases}
\mathrm{div}\left(e^{-w} \frac{\nabla w}{|\nabla w|}\right) \,=\, 4\pi  \delta_x & \text{in } M\,,
\\
\quad \frac{\nabla w}{|\nabla w|}(y)\longrightarrow \nu(y_0) & \text{as } y\to y_0 \in \partial M\, ,
\end{dcases}
\end{equation}
where $\nu(y_0)$ is the outer unit normal at $y_0\in \partial M$. The existence and uniqueness of this pole-normalized solution, with the Dirac coefficient $4\pi$, are provided by Theorem~\ref{thm:pole-existence}. The time $T^*(x)$ is the maximal contact time:
\begin{equation}
\label{richiamo_def_T^*}
T^*(x) \, =\, \max_{\pa M  }  \, w_x \, .
\end{equation}
Also recall that the first contact time is instead given by
\begin{equation}
\label{richiamo_def_T_*}
T_*(x) \, =\, \min_{\pa M }  \, w_x \, .
\end{equation}
To prove the positivity of the $1$-harmonic Mass, we add and subtract the quantity $(\Lambda/4)\, e^{3\tau/2}$ inside the integral. Recalling 
that an IMCF with obstacle subject to our chosen normalizations 
exhibits subexponential growth of the areas and, more precisely, satisfies
${\rm Per} (\Omega_t) \leq 4\pi e^t$, with strict inequality for $t >T_*(x)$ (see~\cite[Remark 3.5-(iii)]{Xu}
and 
\eqref{eq:area-growth-pole}, where
the normalized equality before the first contact is recorded), we arrive at
\begin{align*}
\mathfrak{m}_\Lambda^{(1)} (M, g, x) \,& =  \, 
\int_{-\infty}^{T^+_x} \!\frac{e^{\tau/2}  (1 - \Lambda \, e^\tau)}{4}    \, d\tau  \,  + \Lambda \int_{-\infty}^{T^+_x} \!\!\!\left(\frac{e^{\tau/2}}{16\pi}\right)  \, \big(4\pi e^\tau - \, {\rm Per} (\Omega_{\tau})\big)  \, d\tau
\, \\
&\geq \, 
\frac{e^{(T^+_x)/2}}{2} \, \left( 1 - 
\frac{\Lambda}3
e^{T^+_x}\right) \, \geq \, 0 \, ,
\end{align*}
where we performed a standard integration to treat the first summand. The last inequality follows form the fact that $T^+_x \leq T_\Lambda = \log (3/\Lambda)$, by definition. 

For the rigidity, observe that if $\mathfrak{m}_\Lambda^{(1)} (M, g, x) = 0$, then it necessarily follows from the last inequality that $T^+_x = T_\Lambda \leq T^*{(x)}$. Furthermore, in this situation one must also have $T_\Lambda \leq T_*(x)$; otherwise, the penultimate inequality would be strict. Now, let us restrict our attention to the interval $(-\infty, T_\Lambda]$ and note that, since $T_\Lambda$ does not exceed the first contact time $T_*$, the perimeters of the corresponding sublevel sets evolve according to the law~\eqref{eq:area-growth-pole}, namely
\[
\operatorname{Per}(\Omega_t) = 4\pi e^t,
\]
and, on this interval, the Hawking mass is nondecreasing. 
It should be emphasized that, in order to establish monotonicity, it is essential to invoke Lemma~\ref{lem:topol}, which can be employed to guarantee the connectedness of the IMCF level sets (see Remark~\ref{rmk:H_2}). This topological property, in turn, permits the application of the Gauss-Bonnet formula in determining the sign of the derivative of the Hawking mass along the evolution of the flow.
In particular, we obtain
\begin{align*}
0 \, &= \, \lim_{t \to -\infty}\mathfrak{m}_{\rm Haw} (\partial \Omega_{t})  \, \leq \, \mathfrak{m}_{\rm Haw} (\partial \Omega_{T_\Lambda}) \, =
\\ 
& =  \sqrt{\frac{|\partial \Omega_{T_\Lambda}|}{16 \pi}}  \, \left\{ 1 - \frac{\Lambda}{12 \pi} |\partial \Omega_{T_\Lambda}| - \frac{1}{16 \pi} \int_{\partial \Omega_{T_\Lambda}} \!\!\!\!\!\!\HHH^2 \, d \sigma \right\}
\\
& = \, - \, \frac{R_\Lambda}{32 \pi} \int_{\partial \Omega_{T_\Lambda} }\!\!\!\!\!\! \HHH^2 \, d \sigma \, ,
\end{align*}
as $|\partial \Omega_{T_\Lambda}| = \operatorname{Per}(\Omega_{T_\Lambda}) = 12\pi/\Lambda$. It follows that $\partial \Omega_{T_\Lambda}$ is a minimal surface and therefore, by our standing assumptions, must coincide with $\partial M$. Consequently, we deduce that if the 
$1$-harmonic Mass 
vanishes, then $T_* = T_\Lambda = T^*$; in other words, the level sets of the flow foliate the entire punctured manifold $M \setminus \{x\}$, and the Hawking mass remains identically zero along the evolution.
Employing the vanishing of the derivative of the Hawking mass along the IMCF~\eqref{eq:geroch}, in direct analogy with the argument used in the proof of Theorem~\ref{thm:pos_pass_pol}, one concludes that the manifold must be isometric to a hemisphere with constant sectional curvature equal to $3/\Lambda$.    
\end{proof}

Having already addressed in Subsection~\ref{sub:1harm} the extension of the notion of Polarized $1$-harmonic Mass~\eqref{def:1harm} to the setting of time-symmetric initial data admitting a black-hole-type horizon, we now turn to the proof of the Riemannian Penrose Inequality. In accordance with the classical heuristic, this inequality asserts that, in the presence of a black hole, the total mass is bounded below by the mass of a corresponding static black hole solution (specifically, a Schwarzschild–de Sitter spacetime) whose horizon has the same area as the given one.

\begin{theorem}
[Riemannian Penrose Inequality for the $1$-harmonic Mass]
\label{thm:RPI-1harm_7}
Let $(M,g)$ be a compact, three-dimensional Riemannian band, with scalar curvature satisfying 
\[
\RRR\geq 2 \Lambda, 
\]
for some $\Lambda >0$, and with compact minimal boundary $\partial M$ decomposing into two connected components
\[
\partial M \, = \, \partial M^- \sqcup \partial M^+,
\]
such that $|\partial M^-| < |\partial M^+|$. 
Assume that $\pa M^-$ is strictly stable, that $\pa M^+$ is simply connected and unstable and that there are no closed minimal surfaces in $M\setminus \pa M$. Then, the $1$-harmonic Mass of $(M,g)$ defined in~\eqref{def:1harm} satisfies
\begin{equation}
\mathfrak{m}_\Lambda^{(1)} (M,g) \, \geq \, \sqrt{\frac{|\partial M^-|}{16 \pi}} \left( 1 - \frac{\Lambda}{12 \pi} |\partial M^-|\right)  ,
\end{equation}
with equality if and only if $(M, g)$ is isometric to the Schwarzschild–de Sitter solution with mass parameter $m = \mathfrak{m}_\Lambda^{(1)} (M,g)$.
\end{theorem}

\begin{proof} 
We first recall that the $1$-harmonic Mass of a Riemannian band is given by
\begin{equation}
\mathfrak{m}_\Lambda^{(1)} (M,g) \, := \,  \mathfrak{m}_{\rm Haw} (\partial M^-) \, + \, \int_{T^-}^{T^+} \left(\frac{e^{\tau/2}}{16\pi}\right)  \big(4\pi - \Lambda \, {\rm Per} (\Omega_{\tau})\big)  \, d\tau ,
\end{equation}
where $\Omega_\tau = \{ w< \tau \}$, $T^-= 2 \log R^-=\log (|\pa M^-|/4\pi)$ and $T^+= \min\{T_\Lambda^+(p_\Lambda(R^-)), T^*\}$. Here $w$ is the unique solution to problem~\eqref{eq:ivp-obs-imcf}, namely the IMCF emanating at time $T^-$ from the black-hole horizon $\pa M^-$, having the cosmological horizon $\pa M^+$ as an outer obstacle. Also recall that $p_\Lambda$ is the cubic polynomial 
\[
p_\Lambda(r) \, = \, \frac{r}{2}\left(1-\frac{\Lambda}{3}r^2 \right) \, ,
\]
moreover, for $0 < m < 1/(3 \sqrt \Lambda)$, we let $R_\Lambda^{\pm}(m)$ denote the two positive roots of the equation $p_\Lambda(r) = m$ and we set $T_\Lambda^\pm(m) = 2 \log (R_\Lambda^\pm(m))$. 

Before proving the Penrose inequality~\eqref{eq:RPI_ABM}, we first substantiate the claim made in Remark~\ref{rmk:stabledef}, namely that the stability of $\pa M^-$ entails the area bound $|\pa M^-| \leq 4 \pi/\Lambda$, from which the well-posedness of the definition of $T^+$ follows. To see this, we observe that the strict stability of $\pa M^-$ as a minimal surface implies that the (one sided) second variation of the area functional is positive. With the help of the well known re-arrangements induced by the Gauss Equation, the latter condition reads
\[
\int_{\pa M^- } \!\!\! |\na \phi|^2 d \sigma > \frac12 \int_{\pa M^-} \!\!\!\left(\RRR + |\mathring{\rm h}|^2 + \frac32 {\HHH^2} \right) \phi^2 \, d\sigma \, - \frac12 \int_{\pa M^-} \!\!\!\!\! \RRR^{\pa M^-} \phi^2\, d\sigma \, , 
\]
for every nontrivial nonnegative function $\phi\in\mathscr C^\infty(\partial M^-)$, with the gradient taken in the induced metric.
Choosing $\phi \equiv 1$ in the above expression, one finds that $\pa M^-$ has positive Euler characteristic and hence it is spherical. Moreover, it holds
\[
4\pi = 2 \pi \chi(\pa M^-) \,  = \, \frac12 \int_{\pa M^-} \!\!\!\!\! \RRR^{\pa M^-} \, d\sigma \, > \,  
\frac12 \int_{\pa M^-} \!\!\!\!\! \RRR + |\mathring{\rm h}|^2  \, d\sigma \, \geq \, \Lambda \, |\pa M^-| \,,
\]
which implies at once
\begin{equation*}
0<R^- < \sqrt{{1}/{\Lambda}}    \qquad \hbox{and} \qquad  - \infty < T^- < - \log \Lambda \,  .
\end{equation*}
Consequently, both $R^+$ and $T^+$ are strictly larger than $R^-$ and $T^-$, respectively, since $p_\Lambda(R^-) < 1/(3\sqrt{\Lambda})$ and the inverse mean curvature flow under consideration~\eqref{eq:ivp-obs-imcf} 
cannot extinguish instantaneously.
This latter observation follows from the
fact that the strict inequality $|\pa M^-| <|\pa M^+|$ prevents the whole manifold $M$ from being the strictly outward minimizing hull of $\pa M^-$ (see~\cite[Theorem 3.12 and Theorem 3.13]{Xu}).
To verify the validity of inequality~\eqref{eq:RPI_ABM}, we first observe that the right-hand side of this inequality coincides with the Hawking mass of $\partial M^-$. Hence, using the definition of the 
$1$-harmonic Mass, 
we can write
\begin{align*}
\mathfrak{m}_\Lambda^{(1)}(M,g) \, - \, & \sqrt{\frac{|\partial M^-|}{16\pi}} \left( 1 - \frac{\Lambda}{12\pi} |\partial M^-| \right)
= \int_{T^-}^{T^+} \left(\frac{e^{\tau/2}}{16\pi}\right)\big(4\pi - \Lambda\,{\rm Per}(\Omega_\tau)\big)\,d\tau \\
=& \int_{T^-}^{T^+} \frac{e^{\tau/2}(1 - \Lambda e^\tau)}{4}\,d\tau
+ \Lambda \int_{T^-}^{T^+} \left(\frac{e^{\tau/2}}{16\pi}\right)\big(4\pi e^\tau - {\rm Per}(\Omega_\tau)\big)\,d\tau \,\geq
 \\
\geq &\,
\left[\frac{e^{\tau/2}}{2}\left(1 - \frac{\Lambda}{3}e^{\tau}\right)\right]_{\tau=T^-}^{\tau=T^+}
= p_\Lambda(R^+) - p_\Lambda(R^-)\, .
\end{align*}
Here, the inequality is a consequence of the subexponential growth of the areas of the level sets (see~\cite[Remark 3.5-(iii)]{Xu}). Indeed, by definition of $T^-$, we have
\[
{\rm Per}(\Omega_\tau) \leq {\rm Per}(\Omega_{T^-})\,e^{\tau - T^-} = 4\pi e^\tau \, ,
\]
with strict inequality for $\tau >T_*$ (see \eqref{richiamo_def_T^*}-\eqref{richiamo_def_T_*} for a definition of $T^*$ and $T_*$).
The nonnegativity of the rightmost term now follows from the definition of $T^+$ or, equivalently, from that of $R^+$. Specifically, setting $R^* = e^{T^*/2}$, we have that if 
{$R_\Lambda^+(p_\Lambda(R^-)) \leq R^*$, then $R^+ = R_\Lambda^+(p_\Lambda(R^-))$ 
}and consequently $p_\Lambda(R^+) = p_\Lambda(R^-)$. If instead $R_\Lambda^+(p_\Lambda(R^-)) > R^*$, then $R^+ = R^* \in (R^-, R_\Lambda^+(p_\Lambda(R^-)))$, which implies $p_\Lambda(R^+) > p_\Lambda(R^-)$.
The reader can refer to Figure~\ref{fig:p-lambda}, for an immediate check of these simple facts.

Regarding rigidity, observe that equality in~\eqref{eq:RPI_ABM} immediately implies that $T^+_\Lambda(p_\Lambda(R^-)) \leq T^*$, by the previous observation.
Consequently,
\begin{equation*}
0 \, = \int_{T^-}^{T^+_\Lambda(p_\Lambda(R^-))} \left(\frac{e^{\tau/2}}{16\pi}\right)\big(4\pi - \Lambda\,{\rm Per}(\Omega_\tau)\big)\,d\tau \, = \,   \Lambda \int_{T^-}^{T^+_\Lambda(p_\Lambda(R^-))} \left(\frac{e^{\tau/2}}{16\pi}\right)\big(4\pi e^\tau - {\rm Per}(\Omega_\tau)\big)\,d\tau\,.
\end{equation*}
Hence, it follows that ${\rm Per}(\Omega_\tau) = 4 \pi e^\tau$ for every $\tau \in [T^-,T^+_\Lambda(p_\Lambda(R^-)) ]$, from which we deduce
\[
T^+_\Lambda(p_\Lambda(R^-)) \leq T_* \,,
\]
using again~\cite[Remark 3.5-(iii)]{Xu}.
We now consider the IMCF with obstacle $\partial M^+$ emanating from $\partial M^-$ at time $T^-$, and we restrict our attention to the time interval $[T^-, T^+_\Lambda(p_\Lambda(R^-))]$. Since $T^+_\Lambda(p_\Lambda(R^-))$ does not exceed the first contact time $T_*$, the Hawking mass is nondecreasing on this interval. It should be emphasized that, in order to establish monotonicity, it is essential to invoke Lemma~\ref{lem:topol}, which can be employed to guarantee the connectedness of the IMCF level sets (see Remark~\ref{rmk:H_2}). This topological property, in turn, permits the application of the Gauss-Bonnet formula in determining the sign of the derivative of the Hawking mass along the evolution of the flow.
In particular, we obtain
\begin{align*}
p_\Lambda (R^- ) \, &= \, \mathfrak{m}_{\rm Haw} (\partial M^-)  \, \leq \, \mathfrak{m}_{\rm Haw} (\partial \Omega_{T^+_\Lambda(p_\Lambda(R^-))}) \, 
\\ 
& = \, p_\Lambda(R^+_\Lambda(p_\Lambda(R^-))) \, - \, \frac{R^+_\Lambda(p_\Lambda(R^-))}{32 \pi} \int_{\partial \Omega_{T^+_\Lambda(p_\Lambda(R^-))}} \!\!\!\!\!\!\!\!\!\!\!\!\!\!\! \HHH^2 \, d \sigma \, .
\end{align*}
Since $R^+= R^+_\Lambda(p_\Lambda(R^-))$ and in turn $p_\Lambda (R^- ) = p_\Lambda(R^+_\Lambda(p_\Lambda(R^-)))$, it follows that $\partial \Omega_{T^+_\Lambda(p_\Lambda(R^-))}$ is a minimal surface. Moreover, as $\partial M^-$ is an outermost minimal surface in $M \setminus \partial M^+$, this surface must coincide with $\partial M^+$. Therefore, the IMCF under consideration starts from $\partial M^-$ at time $T^-$ and terminates at $\partial M^+$ at the standard existence time for the model flow, namely $T_* = T^+_\Lambda(p_\Lambda(R^-)) = T^*$. The constancy of the Hawking mass along the flow then implies, by standard rigidity arguments, that the manifold is isometric to the Schwarzschild-de Sitter solution with the corresponding mass.
\end{proof}

\appendix

\section{Power series solutions to the  structural ODE system }
\label{app:ODE}

In this appendix, we study several distinguished solutions of the linear ODE system
\begin{equation}
\label{eq:ODE_Phi_Psi_app_A}
\begin{dcases}
\displaystyle \frac{d\Phi}{dr}
= 2\left[\frac{(\Lambda/3)\, r^2}{1-(\Lambda/3)\, r^2} - \left(\frac{3-p}{p-1}\right)\right]\frac{1}{r}\, \Phi 
+ \left(\frac{5-p}{p-1}\right)\frac{1}{r}\, \Psi, 
\\[0.7em]
\displaystyle \frac{d\Psi}{dr}
= - \left(\frac{3-p}{p-1}\right)\frac{1}{r}\, \Phi 
+ \left[\frac{(\Lambda/3)\, r^2}{1-(\Lambda/3)\, r^2} + \left(\frac{2}{p-1}\right)\right]\frac{1}{r}\, \Psi,
\end{dcases}
\end{equation}
where $1<p<3$, $\Lambda \in \mathbb{R}$, and $r \in (0, R_\Lambda)$, with $R_\Lambda = \sqrt{3/\Lambda}$ if $\Lambda > 0$ and $R_\Lambda = +\infty$ if $\Lambda \leq 0$. As explained in the introduction to Subsection~\ref{sub:global}, this system is an alternative formulation of~\eqref{structural}, which is central for the analysis of the monotone quantities of Section~\ref{sec:MF}.

We mainly treat the case $\Lambda > 0$, which is most relevant in the present contribution. However, a similar analysis applies when $\Lambda < 0$, and we will indicate, when needed, how to adapt the arguments. The case $\Lambda = 0$ is elementary and serves as a model illustrating the basic mechanisms of the theory. The main purpose of this Appendix is to establish the following result.

\begin{proposition}
\label{pro:Phi_Psi_solutions}
For $1<p<3$ and $|x|<1$, let 
\[
\Upsilon(x)\,=\,{}_2F_1\left(a_p,b_p,c_p;x\right)\,,
\] 
be the hypergeometric function whose  parameters $a_p,b_p$ and $c_p$ are given by~\eqref{eq:parameters}. For $\Lambda\geq 0$ and $0<r<R_\Lambda$, let $\Phi$ and $\Psi$ be the functions defined by
\begin{equation}
\label{eq:Phi_Psi_app_A}
\begin{aligned}
\Phi(r)\,&=\,\frac{1}{8\pi}\, \frac{r}{1-(\Lambda/{3}) \, r^2}\,\, \Upsilon\!\bigl(\tfrac{\Lambda}{3} r^2\bigr)\,,
\\
\Psi(r)\,&=\,\frac{1}{8\pi} \, \frac{r}{1-(\Lambda/{3}) \,r^2} \, \left[\Upsilon\!\bigl(\tfrac{\Lambda}{3}r^2\bigr) \, +2 \!\left(\frac{p-1}{5-p}\right)({\Lambda}/{3}) \, r^2\,\, \dot\Upsilon\!\bigl(\tfrac{\Lambda}{3}r^2\bigr)\right]\,  \, .
\end{aligned}
\end{equation}
Then, $\Phi$ and $\Psi$ are the only positive solutions to~\eqref{eq:ODE_Phi_Psi_app_A} that satisfy
\begin{equation}
\label{limphipsizero_A}
\lim_{r\to 0^+}\frac{\Phi(r)}{r}\,=\,\frac{1}{8\pi}\,\qquad \hbox{and} \qquad
\lim_{r\to 0^+}\frac{\Psi(r)}{r}\,=\,\frac{1}{8\pi}\,.
\end{equation}
\end{proposition}

\begin{remark}
\label{rem:Y_negativo}
The above result can be easily extended to the case $\Lambda = 0$. The proof is indeed straightforward in this regime, as setting $\Lambda = 0$ into~\eqref{eq:Phi_Psi_app_A} yields $\Phi(r) = r/(8\pi) = \Psi(r)$. A straightforward verification then shows that this choice simultaneously satisfies the differential equation~\eqref{eq:ODE_Phi_Psi_app_A} as well as the  condition~\eqref{limphipsizero_A}.
With the same proof we can show that~\eqref{eq:Phi_Psi_app_A} also solves~\eqref{eq:ODE_Phi_Psi_app_A} when $\Lambda<0$. In this case, however, formula~\eqref{eq:Phi_Psi_app_A} is well defined only while the argument $(\Lambda/3)r^2$ of $\Upsilon$ is greater than $-1$, i.e.\ for $r\leq \sqrt{3/|\Lambda|}$. Nevertheless, the solution extends beyond $r=\sqrt{3/|\Lambda|}$. Indeed, Theorem~\ref{thm:global} proves that structural coefficients $\mu$ and $\lambda$ with the chosen asymptotics at $-\infty$ exist for all times, and thus so do the corresponding $\Phi=\mu e^\lambda/\alpha^2$ and $\Psi=e^\lambda/\alpha$. Note that condition~\eqref{limphipsizero_A} guarantees this correspondence.
\end{remark}

The second main result of this appendix concerns the behavior of the solutions $\Phi$ and $\Psi$, defined by~\eqref{eq:Phi_Psi_app_A}, as $r \to R_\Lambda$ and hence as $p \to 1$.

\begin{proposition}
\label{pro:behaviour_near_one}
For $1<p<3$ and $\Lambda>0$, let $\Phi$ and $\Psi$ be the functions defined by~\eqref{eq:Phi_Psi_app_A}. Then, it holds
\begin{align}
\label{eq:behaviour_rp_to_one_phi}
\lim_{r\to R_\Lambda}\Phi(r)\left(1-\frac{\Lambda}{3}r^2\right)\,&=\,\frac{R_\Lambda}{16{\pi}}\,\, \frac{\Gamma(\tfrac{1}{2})\, \Gamma(c_p)}{\Gamma(a_p+\frac{3}{2})\, \Gamma(b_p+\frac{3}{2})} \, =\, \frac{R_\Lambda}{16\pi}\,(p-1) \, \left(1+o(1)\right)\,,
\\
\label{eq:behaviour_rp_to_one_psi}
\lim_{r\to R_\Lambda}\Psi(r) \,\, \sqrt{1-\frac{\Lambda}{3}r^2}\,&=\,\frac{R_\Lambda}{8{\pi}}\,\, \frac{\Gamma(\tfrac12) \, \Gamma(c_p)}{\Gamma(a_p+1)\, \Gamma(b_p+1)}\,\, = \,\frac{3 R_\Lambda}{32\sqrt{\pi}}\,(p-1)^{\frac{3}{2}}  \left(1+o(1)\right)\,.
\end{align}
as $p \to 1^+$.
\end{proposition}

Finally, we discuss the behaviour of $\Phi$ and $\Psi$ as $p\to 1$. 

\begin{proposition}
\label{pro:behaviour_near_one_2}
Let $\Phi$ and $\Psi$ be the functions defined by~\eqref{eq:Phi_Psi_app_A} and let $\Lambda>0$.
Then, for all $r\in[0,R_\Lambda)$ it holds
\[
\Phi(r)\,=\,\frac{r}{8\pi}+o(1)\,,\quad \Psi(r)\,=\, \frac{r}{8\pi}+o(1)
\]
as $p\to 1^+$.
\end{proposition}

The appendix is organized as follows: Subsection~\ref{app:Upsilon} states algebraic-differential properties of the hypergeometric function $\Upsilon$, and Subsections~\ref{app:ODE_proof},~\ref{app:ASYM_proof} and~\ref{app:ASYM2_proof} contain the proofs of Propositions~\ref{pro:Phi_Psi_solutions},~\ref{pro:behaviour_near_one} and~\ref{pro:behaviour_near_one_2}, respectively.

\subsection{Properties of the function \texorpdfstring{$\Upsilon$}{ϒ}}
\label{app:Upsilon}

We discuss here the main features and properties of the function
\[
\Upsilon(x)\,=\,{}_2F_1\left(a_p,b_p,c_p;x\right)\,,\qquad |x|<1\,,
\]
where
\begin{align}
\label{eq:parameters}
a_p\,&=\,\frac{3-p+\sqrt{4+12(p-1)-3(p-1)^2}}{4(p-1)}\,, \nonumber
\\
b_p\,&=\,\frac{3-p-\sqrt{4+12(p-1)-3(p-1)^2}}{4(p-1)}\,,
\\
\nonumber
c_p \,& = \, \frac{p}{p-1}\,.
\end{align}
The function ${}_2F_1\left(a_p,b_p,c_p;x\right)$ is called the hypergeometric function of parameters $a_p,b_p$ and $c_p$, whereas $x$ is the variable.
We recall that ${}_2F_1$ is defined by 
\[
{}_2F_1(a,b,c;x)=\sum_{k=0}^{+\infty}\frac{(a)_k (b)_k}{(c)_k}\frac{x^k}{k!}\,,\quad |x|<1\,,
\]
where the notation $(h)_k$ is the Pochhammer symbol, defined as
\[
(h)_k\,=\,
\begin{dcases}
1 & \hbox{if }k=0,
\\
h(h+1)(h+2)\cdots (h+k-1)& \hbox{if }k>0.
\end{dcases}
\]
Concerning the parameters $a_p,b_p$ and $c_p$, one can directly check that $a_p$ is decreasing with $\lim_{p\to 1}a_p=+\infty$ and $a_3=1/2$, whereas $b_p$ is increasing with $\lim_{p\to 1}b_p=-1$ and $b_3=-1/2$. As a consequence, the following bounds hold for any $1<p<3$:
\begin{equation}
\label{eq:bound_abc}
a_p>\frac{1}{2}\,,\quad -1<b_p<-\frac{1}{2}\,,\quad c_p>\frac{3}{2}\,.
\end{equation}
From the power series definition of the hypergeometric function, we immediately get
\[
\Upsilon(0)=1\,.
\]
The series defining $\Upsilon$ also converges at $x=1$. By Gauss's summation formula, for any hypergeometric function ${}_2F_1(a,b,c;x)$ with $c>a+b$, it holds the following Gamma function representation in $x=1$
\[
{}_2F_1(a,b,c;1)=\frac{\Gamma(c)\Gamma(c-a-b)}{\Gamma(c-a)\Gamma(c-b)}\,.
\]
Since $c_p=a_p+b_p+3/2$, it follows that
\[
\Upsilon(1) \, = \, \frac{\Gamma(\frac{3}{2})\Gamma(c_p)}{\Gamma(a_p+\frac{3}{2})\Gamma(b_p+\frac{3}{2})} >0\,.
\]
For future convenience, we also point out that and $\Gamma(3/2)=\sqrt{\pi}/2$.

\medskip

Another important feature of hypergeometric function ${}_2F_1(a,b,c;x)$ is that they  satisfy a second order differential equation determined by their parameters
\[
x(1-x)\ddot y+[c-(a+b+1)x]\dot y-ab y\,=\,0\,.
\]
Since $a_p+b_p=(3-p)/(2(p-1))$ and $a_p b_p=-(5-p)/(4(p-1))$, we have that $\Upsilon$ satisfies
\begin{equation}
\label{eq:ODE_Upsilon}
x\ddot\Upsilon(x)\,=\,\left[-\frac{p}{p-1}+\frac{p+1}{2(p-1)}x\right]\frac{\dot\Upsilon(x)}{1-x}-\frac{5-p}{4(p-1)}\frac{\Upsilon(x)}{1-x}\,,
\end{equation}
where we have denoted by $\dot\Upsilon$ and $\ddot\Upsilon$ the first and second derivative of $\Upsilon$ with respect to $x$. 
Furthermore, from the definition of hypergeometric function we also deduce
\begin{equation}
\label{eq:Upsilon_derivative}
\dot\Upsilon(x)\,=\,-\frac{5-p}{4p}\,{}_2F_1(a_p+1,b_p+1,c_p+1;x)\,.
\end{equation}
Since every hypergeometric function computed in $x=0$ gives $1$, we deduce 
\[
\dot\Upsilon(0)\,=\,-\frac{5-p}{4p}\,.
\]
We can also compute $\dot\Upsilon(1)$, with the help of the Euler's Gamma functions , as we just did for $\Upsilon(1)$
\begin{equation}
\label{eq:Upsilondot_one}
\dot\Upsilon(1)\,=\,-\frac{5-p}{4(p-1)}\frac{\Gamma(\frac12) \,\Gamma(c_p)}{\Gamma(a_p+\frac{3}{2})\Gamma(b_p+\frac{3}{2})}\,=\,-\frac{5-p}{2(p-1)}\Upsilon(1)\,,
\end{equation}
where we have also used the identities $\Gamma(c_p + 1) = c_p \, \Gamma (c_p)$ and $\Gamma(\frac12) = \sqrt{\pi}$.
Notice also that the coefficients in the hypergeometric function in the right hand side of~\eqref{eq:Upsilon_derivative} are all positive thanks to~\eqref{eq:bound_abc}, hence $\dot\Upsilon <0$. It follows that $x \mapsto \Upsilon(x)$ is decreasing, and thus positive, since $\Upsilon(1) >0$. Differtiating~\eqref{eq:Upsilon_derivative}, we obtain
\[
\ddot\Upsilon(x)\,=\,-\frac{5-p}{4p}\,\frac{(a_p+1)(b_p+1)}{(c_p+1)}\,{}_2F_1(a_p+2,b_p+2,c_p+2;x)<0\,.
\]
Summarizing, we have proved that
\[
\Upsilon(x)>0\,,\quad \dot\Upsilon(x)<0\,,\quad
\ddot\Upsilon(x)<0\,,
\]
for all $x\in(0,1)$.

\subsection{Proof of Proposition~\ref{pro:Phi_Psi_solutions}}
\label{app:ODE_proof}

In this subection, we rigorously verify that the functions defined in~\eqref{eq:Phi_Psi_app_A} satisfy all the prescribed requirements.

\smallskip

First, observe that the conditions in~\eqref{limphipsizero_A} concerning the behavior of $\Phi$ and $\Psi$ as $r \to 0^+$ are indeed satisfied, since $\Upsilon(0) = 1$ and $\dot\Upsilon(0) = -(5-p)/(4p)$, as established in Subsection~\ref{app:Upsilon}. 

\smallskip

The positivity of $\Phi(r)$ for $r\in(0,1)$ is an immediate consequence of the positivity of $\Upsilon$, proved in Subsection~\ref{app:Upsilon}.
To prove the positivity of $\Psi$, we first recall that $\dot\Upsilon(x)$ and $\ddot\Upsilon(x)$ are both negative for all $x\in(0,1)$, hence
\[
\frac{d}{dx}\left(\Upsilon(x)+2\frac{p-1}{5-p}x\dot\Upsilon(x)\right)=\left(1+2\frac{p-1}{5-p}\right)\dot\Upsilon(x)+2\frac{p-1}{5-p}x\ddot\Upsilon(x)<0
\]
Since we know from~\eqref{eq:Upsilondot_one} that
\[
\Upsilon(1)+2\frac{p-1}{5-p}\dot\Upsilon(1)=0\,,
\]
we conclude that the quantity in square bracket in the definition of $\Psi$ is always nonnegative, hence $\Psi(r)\geq 0$ for all $r$, as wished.

\smallskip

It only remains to prove that $\Phi,\Psi$ given by~\eqref{eq:Phi_Psi_app_A} indeed solve~\eqref{eq:ODE_Phi_Psi_app_A}.
To this end, it is convenient to consider the functions $F=\Phi/r$ and $G=\Psi/r$ and to see them as functions of $x=(\Lambda/3)r^2$. To prove that $\Phi$, $\Psi$ satisfy~\eqref{eq:ODE_Phi_Psi_app_A} is equivalent to prove that  
\begin{align*}
F(x)\,&=\,\frac{1}{8\pi}\frac{1}{1-x}\Upsilon(x)\,,
\\
G(x)\,&=\,\frac{1}{8\pi}\frac{1}{1-x}\left[\Upsilon(x)+2\,\frac{p-1}{5-p}\,x\,\dot\Upsilon(x)\right]\,.
\end{align*}
solve
\begin{equation}
\label{eq:ODE_F_G}
\begin{dcases}
\frac{dF}{dx}(x)\,=\,\left(\frac{1}{1-x}-\frac{5-p}{2(p-1)}\frac{1}{x}\right)F(x)+\frac{5-p}{2(p-1)}\frac{1}{x}G(x)
\\
\frac{dG}{dx}(x)\,=\,-\frac{3-p}{2(p-1)}\frac{1}{x}F(x)+\left(\frac{1}{2(1-x)}+\frac{3-p}{2(p-1)}\frac{1}{x}\right)G(x)
\end{dcases}
\end{equation}
To check the first prescription in~\eqref{eq:ODE_F_G}, we compute
\begin{align*}
\frac{dF(x)}{dx}\,&=\,\frac{1}{8\pi(1-x)}\left[\frac{1}{1-x}\Upsilon(x)+\dot\Upsilon(x)\right]
\\
&=\frac{1}{1-x}F(x)+\frac{5-p}{2(p-1)}\frac{1}{x}\left[G(x)-F(x)\right]\,.
\end{align*}
In the second passage we have used
\begin{equation}
\label{eq:formula_Upsilondot_FG}
\frac{1}{4\pi}\frac{p-1}{5-p}\frac{x}{1-x}\dot\Upsilon(x)\,=\,G(x)-F(x)\,.
\end{equation}
which in turn follows from the very definition of $F$ and $G$. It is now easily seen that the first equation in~\eqref{eq:ODE_F_G} is satisfied. To check the validity of the second equation in~\eqref{eq:ODE_F_G}, we start from~\eqref{eq:ODE_Upsilon} and use~\eqref{eq:formula_Upsilondot_FG} as well as the definition of $F$ and $G$ to compute
\begin{align*}
x\ddot\Upsilon(x)\,&=\,\left[-\frac{p}{p-1}+\frac{p+1}{2(p-1)}x\right]\frac{\dot\Upsilon(x)}{1-x}-\frac{5-p}{4(p-1)}\frac{\Upsilon(x)}{1-x}
\\
&=\,2\pi\frac{5-p}{p-1}\left[\frac{1}{(p-1)x}\left(-2p+(p+1)x\right)\left(G(x)-F(x)\right)-F(x)\right]
\\
&=\,2\pi\frac{5-p}{p-1}\left[-\frac{2p}{p-1}\frac{1-x}{x}\left(G(x)-F(x)\right)-G(x)\right]\,.
\end{align*}
This allows to write $\ddot\Upsilon$ in terms of $F$ and $G$. It follows that 
\begin{align*}
\frac{dG(x)}{dx}\,&=\,
\frac{G(x)}{1-x}+\frac{1}{8\pi}\frac{p+3}{5-p}\frac{1}{1-x}\dot\Upsilon(x)+\frac{1}{4\pi}\frac{p-1}{5-p}\frac{x}{1-x}\ddot\Upsilon(x)
\\
&=\,\frac{G(x)}{1-x}+\frac{p+3}{2(p-1)}\frac{1}{x}\left(G(x)-F(x)\right)+\frac{1}{2}\left[-\frac{2p}{p-1}\frac{1}{x}(G(x)-F(x))-\frac{G(x)}{1-x}\right]
\\
&=\,\frac{G(x)}{2(1-x)}+\frac{3-p}{2(p-1)}\frac{1}{x}\left(G(x)-F(x)\right)\,.
\end{align*}
Rearranging these terms one gets precisely the second equation in~\eqref{eq:ODE_F_G}.

\subsection{Proof of Proposition~\ref{pro:behaviour_near_one}}
\label{app:ASYM_proof}

We first check the validity of 
\begin{align*}
\lim_{r\to R_\Lambda}\Phi(r)\left(1-\frac{\Lambda}{3}r^2\right)\,&=\,\frac{R_\Lambda}{16{\pi}}\,\, \frac{\Gamma(\tfrac{1}{2})\, \Gamma(c_p)}{\Gamma(a_p+\frac{3}{2})\, \Gamma(b_p+\frac{3}{2})}\,, 
\\
\lim_{r\to R_\Lambda}\Psi(r) \,\, \sqrt{1-\frac{\Lambda}{3}r^2}\,&=\,\frac{R_\Lambda}{8{\pi}}\,\, \frac{\Gamma(\tfrac12) \, \Gamma(c_p)}{\Gamma(a_p+1)\, \Gamma(b_p+1)} \,,
\end{align*}
then we discuss the asyptotic behavior of the two limits, as $p \to 1^+$.

\smallskip

The first limit follows immediately from the definition of $\Phi$, recalling that 
\begin{equation*}
\Upsilon(1)\,=\,\frac{\Gamma(\frac32) \, \Gamma(c_p)}{\Gamma(a_p+\frac{3}{2})\Gamma(b_p+\frac{3}{2})}\,,
\end{equation*}
as we have computed in Subsection~\ref{app:Upsilon}.

\smallskip

Concerning the limit for $\Psi$, we recall that 
\[
\left( 1-(\Lambda/{3}) \,r^2  \right) \, \Psi(r)\,=\,\frac{r}{8\pi} \,  \left[\Upsilon\!\bigl(\tfrac{\Lambda}{3}r^2\bigr) \, +2 \!\left(\frac{p-1}{5-p}\right)({\Lambda}/{3}) \, r^2\,\, \dot\Upsilon\!\bigl(\tfrac{\Lambda}{3}r^2\bigr)\right]\,  \, .
\]
and we use again the substitution $x=(\Lambda/3)r^2$ to study the right hand side. We are going to exploit the following transformation formula for the hypergeometric functions:
\begin{multline*}
{}_2F_1(a,b,c;x)\,=\,\frac{\Gamma(c)\Gamma(c-a-b)}{\Gamma(c-a)\Gamma(c-b)} \, {}_2F_1(a,b,a+b-c+1;1-x)
\\
+\frac{\Gamma(c)\Gamma(a+b-c)}{\Gamma(a)\Gamma(b)}(1-x)^{c-a-b} \, {}_2F_1(c-a,c-b,c-a-b+1;1-x)\,.
\end{multline*}
Applying this transformation formula to the two hypergeometric functions that appear in
\[
\Upsilon(x)+2\frac{p-1}{5-p}x\dot\Upsilon(x)={}_2F_1(a_p,b_p,c_p;x)-\frac{p-1}{2p}x {}_2F_1(a_p+1,b_p+1,c_p+1;x) \, ,
\]
this latter quantity can be computed as
\begin{multline*}
\frac{1}{2}\frac{\Gamma(c_p)\Gamma(\frac{1}{2})}{\Gamma(a_p+1)\Gamma(b_p+1)}\left[{}_2F_1(a_p,b_p,-\frac{1}{2};1-x)-x\,{}_2F_1(a_p+1,b_p+1,\frac{1}{2};1-x)\right]
\\
+\frac{\Gamma(c_p)\Gamma(-\frac{3}{2})}{\Gamma(a_p)\Gamma(b_p)}\sqrt{1-x}\left[(1-x){}_2F_1(a_p+\frac{3}{2},b_p+\frac{3}{2},\frac{5}{2};1-x)\right.
\\
\left. -3\frac{p-1}{5-p}x\, {}_2F_1(a_p+\frac{3}{2},b_p+\frac{3}{2},\frac{3}{2};1-x)\right]\,.
\end{multline*}
Since ${}_2F_1(a,b,c;y)=1+\mathcal{O}(y)$ as $y\to 0$, we easily deduce that the leading term of the above quantity as $x\to 1$ is given by the last hypergeometric function. Using $\Gamma(-3/2)=(4/3)\Gamma(1/2)$, we  obtain
\[
\left[\Upsilon\!\bigl(\tfrac{\Lambda}{3}r^2\bigr) \, +2 \!\left(\frac{p-1}{5-p}\right)({\Lambda}/{3}) \, r^2\,\, \dot\Upsilon\!\bigl(\tfrac{\Lambda}{3}r^2\bigr)\right]=-4 \left(\frac{p-1}{5-p}\right)\frac{\Gamma(\tfrac12)\Gamma(c_p)}{\Gamma(a_p)\Gamma(b_p)}\sqrt{1-\frac{\Lambda}{3}r^2} \, +\, \mathcal{O}(1-\frac{\Lambda}{3}r^2)\,,
\]
as $r\to R_\Lambda$. Substituting in the expression for $\Psi$ and using the fact that
\[
\Gamma(a_p+1)\Gamma(b_p+1)\,=\,a_p b_p\Gamma(a_p)\Gamma(b_p)\,=\,-\frac{5-p}{4(p-1)}\Gamma(a_p)\Gamma(b_p)\,,
\]
we finally obtain the desired limit.

\medskip

In the second part of this subsection, we are going to discuss the asymptotic behavior of the limit that we have just computed, when the exponent $p$ approaches $1$. First of all, we compute the following expansions as $p\to 1^+$:
\begin{align*}
a_p\,&=\,\frac{1}{p-1}+\frac{1}{2}-\frac{3}{4}(p-1)+\frac{9}{8}(p-1)^2+\mathcal{O}\left((p-1)^3\right)\,,
\\
b_p\,&=\,-1+\frac{3}{4}(p-1)-\frac{9}{8}(p-1)^2+\mathcal{O}\left((p-1)^3\right)\,.
\end{align*}
We also have the following expansions for the Gamma function:
\begin{align*}
\Gamma(-1+x)\,&=\,-\frac{1}{x}+\mathcal{O}(1)\,,\quad\hbox{as }x\to 0^+\,,
\\
\Gamma(x)\,&=\,\sqrt{\frac{2\pi}{x}}e^{x\log x-x}\left(1+\mathcal{O}(x^{-1})\right)\,,\quad \hbox{as }x\to+\infty\,.
\end{align*}
As a consequence, when $p\to 1$, we get
\begin{align*}
\Gamma(a_p)\,&=\,
\sqrt{2\pi}\,e^{-\frac{\log (p-1)}{p-1}-\frac{1}{p-1}+\mathcal{O}((p-1)\log (p-1))}\left(1+\mathcal{O}(p-1)\right)
\\
\Gamma(b_p)\,&=\,-\frac{4}{3}\frac{1}{p-1}+\mathcal{O}(1)\,,
\\
\Gamma\left(c_p\right)\,&=\,
\sqrt{\frac{2\pi}{p-1}}\,e^{-\frac{\log (p-1)}{p-1}-\frac{1}{p-1}}\left(1+\mathcal{O}(p-1)\right)
\\
\Gamma\left(a_p+\frac{1}{2}\right)\,&=\,
\sqrt{\frac{2\pi}{p-1}}\,e^{-\frac{\log (p-1)}{p-1}-\frac{1}{p-1}+\mathcal{O}((p-1)\log (p-1))}\left(1+\mathcal{O}(p-1)\right)
\\
\Gamma\left(b_p+\frac{1}{2}\right)\,&=\,-2\sqrt{\pi}+\mathcal{O}(p-1)\,.
\end{align*}
With the help of the following algebraic identities
\begin{align*}
\Gamma(a_p+1) =a_p\Gamma(a_p) \,, \qquad &\Gamma(b_p+1)=b_p\Gamma(b_p)
\,  \\
\Gamma(a_p+3/2 ) =(a_p+1/2)\Gamma(a_p+1/2)
\,, \qquad&
\Gamma(b_p+3/2)=(b_p+1/2)\Gamma(b_p+1/2) \, ,
\end{align*}
we finally get
\begin{align*}
\frac{\Gamma(\tfrac32)\,\Gamma(c_p)}{\Gamma(a_p+\tfrac32)\,\Gamma(b_p+\tfrac32)}
&\, = \, 
\frac{p-1}{2}
\, + \, \mathcal{O}((p-1)^{2}) \\
\frac{\Gamma(\tfrac12)\,\Gamma(c_p)}{\Gamma(a_p+1)\,\Gamma(b_p+1)}
&\, =\, 
\frac{3\sqrt{\pi}}{4}\,(p-1)^{3/2}
\, +\, \mathcal{O}((p-1)^{5/2})
\end{align*}
as $p\to 1^+$. The thesis follows at once. 

\subsection{Proof of Proposition~\ref{pro:behaviour_near_one_2}}\label{app:ASYM2_proof}

We first discuss the limit of the function $\Upsilon$.
One can check that $a_p<c_p$ for all $p$ and that $-1<b_p<-1/2$, hence, using the ratio test, we find out that the series
\[
{}_2F_1\left(a_p,b_p,c_p;x\right)\,=\,\sum_{k=0}^{+\infty}\frac{(a_p)_k (b_p)_k}{(c_p)_k}\frac{x^k}{k!}
\]
converges absolutely and uniformly in $[0,x_0]$, for all $x_0<1$. This convergence is uniform also for $p$ close to $1$: for $k\geq1$, $0<(a_p)_k/(c_p)_k\leq1$ and $|(b_p)_k|/k!\leq1$, so the terms are dominated by $x_0^k$. We can therefore pass to the limit term by term in the series.
The term for $k=0$ is equal to $1$, whereas the term for $k=1$ converges to $-x$, since
\[
\frac{a_p b_p}{c_p}=-\frac{5-p}{4p}\to -1\,.
\]
For $k\geq 2$, the product given by the Pochhammer symbol $(b_p)_k$ contains the factor $b_p+1$, which goes to zero as $p\to 1$. Since $a_p/c_p$ stays bounded as $p\to 1$ (in fact, $a_p$ and $c_p$ are both positive and we have already observed $a_p<c_p$), we conclude that every term with $k\geq 2$ converges to zero. We have thus proved that $\Upsilon(x)$ converges to $1-x$ as $p\to 1^+$, for all $x\in[0,1)$.

Similarly, one proves that
\[
\dot\Upsilon(x)\,=\,-\frac{5-p}{4p}\,{}_2F_1(a_p+1,b_p+1,c_p+1;x)
\]
converges to $-1$ for all $x\in[0,1)$; here $0<b_p+1<1$ gives the same uniform majorant, and all terms with $k\geq1$ tend to zero.

Recalling the definition~\eqref{eq:Phi_Psi_app_A}, we conclude easily that $\Phi(r)$ and $\Psi(r)$ both converge to $r/8\pi$ as $p\to 1$, as wished.

\section{Regularity and qualitative properties of 
\texorpdfstring{$p$}{p}-harmonic functions}
\label{app:pharmonic}

In this appendix we collect, for the reader’s convenience, some classical results concerning $p$-harmonic functions that are used throughout the paper. We begin by recalling the notion of weak solution for Dirichlet problems involving the $p$-Laplace operator on a Riemannian manifold, and we then list a few standard qualitative properties: Harnack-type inequalities, the local behavior of $p$-Green's functions near the pole, and ${\mathscr C}^{1,\beta}$ interior/boundary estimates. We also include the definition of a $p$-Inverse Mean Curvature Flow (in the sense used in the paper) and its relation with the $p$-Green's function.

\subsection{Weak formulation and basic regularity}

Let $(M,g)$ be a $n$-dimensional Riemannian manifold with smooth boundary $\pa M$, let $p>1$ and $\alpha>0$, and let $v\in\mathscr C^{1,\alpha}(\pa M)$. We say that a function 
$u\in W_{loc}^{1,p}(M\setminus\pa M)$ 
is a solution to the Dirichlet problem  
\begin{equation*}
\begin{dcases}
\Delta_p u\,=\,0 & \hbox{in } M\setminus\pa M,
\\
\quad \, u \,  = \, v & \hbox{on } \pa M,
\end{dcases}
\end{equation*}
if it fulfills the weak formulation
\[
\int_M \left\langle |\na u|^{p-2}  \na u  \, \big| \,  \na \varphi \right\rangle \, d\mu_g \, = \, 0 \, ,
\]
for every $\varphi\in\mathscr C^{\infty}_c(M\setminus\pa M)$, and satisfies the boundary condition in the sense of traces.

By the interior regularity estimates \cite{DiBenedetto1,Tolksdorf1984}
and the boundary regularity estimates
\cite{Lieb88}, such a solution $u$ turns out to be $\mathscr C^{1,\beta}(M)$, for some $0<\beta<1$.
Some of the key results of \cite{DiBenedetto1,Tolksdorf1984,Lieb88}
are recalled below in a form tailored to our setting.

\subsection*{\texorpdfstring{$p$}{p}-Green's functions and \texorpdfstring{$p$}{p}-Inverse Mean Curvature Flow.}

Fix $x\in M\setminus\pa M$. A function
$u\in W_{loc}^{1,p}\big(M\setminus(\pa M\cup\{x\})\big)$ is a solution to the problem
\begin{equation*}
\begin{dcases}
\Delta_p u\,=\,-4\pi\delta_x & \hbox{in } M\setminus\pa M,
\\
\quad \, u \,  = \, v & \hbox{on } \pa M,
\end{dcases}
\end{equation*}
if it fulfills 
\[
\int_M \left\langle |\na u|^{p-2}  \na u  \, \big| \,  \na \varphi \right\rangle \, d\mu_g \, = \, 4\pi\varphi(x) \, ,
\]
for every $\varphi\in\mathscr C^{\infty}_c\big(M\setminus(\pa M\cup\{x\})\big)$,
and satisfies the boundary condition in the sense of traces. Regularity results analogous to the homogeneous case hold away from the pole $x$. In the special case $v\equiv 0$ we will refer to $u$ as the (Dirichlet) \emph{$p$-Green's function} with pole $x$, with the above normalization constant $4\pi$. This convention is compatible (up to the multiplicative factor $4\pi$) with the standard definition in the literature, where one prescribes vanishing boundary values and requires the distributional identity $-\mathrm{div}(|\na u|^{p-2}\na u)=\delta_x$ (see, e.g., \cite{Kura_1999} and the references therein).

In the paper we also consider a notion of $p$-Inverse Mean Curvature Flow encoded by a potential function $w$. We say that $w$ is a \emph{$p$-Inverse Mean Curvature Flow} (with pole $x$) if it solves, in the weak sense, the following boundary blow-up problem:
\begin{equation}
\label{eq:p-imcf_app}
\begin{dcases}
\mathrm{div}\big(e^{-w}|\nabla w|^{p-2}\nabla w\big) \,=\, 4\pi (p-1)^{p-1} \delta_x & \text{in } M\,,
\\
\quad w(y)\longrightarrow +\infty & \text{as } y\to\partial M.
\end{dcases}
\end{equation}
More precisely, $w\in W^{1,p}_{loc}(M\setminus\{x\})$ is required to satisfy
\[
\int_M \left\langle e^{-w}|\na w|^{p-2}\na w \, \big| \, \na\varphi \right\rangle d\mu_g
\,=\, 4\pi (p-1)^{p-1}\,\varphi(x),
\]
for every $\varphi\in\mathscr C^\infty_c(M\setminus\{x\})$, and the boundary condition in \eqref{eq:p-imcf_app} is understood as a boundary blow-up requirement. It is convenient to record that \eqref{eq:p-imcf_app} is equivalent, via the change of variables
\[
u \,:=\, e^{-\frac{w}{p-1}},
\]
to the $p$-Green's equation with zero boundary data in the above normalization. Indeed, on $M\setminus\{x\}$ one has the identity
\[
e^{-w}|\na w|^{p-2}\na w \,=\,-(p-1)^{p-1}|\na u|^{p-2}\na u,
\]
and therefore \eqref{eq:p-imcf_app} becomes
\[
\Delta_p u \,=\, -4\pi\,\delta_x
\qquad\text{in } M,
\]
while the boundary blow-up condition $w\to +\infty$ as $y\to \pa M$ corresponds to $u\to 0$ as $y\to\pa M$.

\subsection{Harnack inequality}

A first fundamental qualitative property of nonnegative $p$-harmonic functions is the Harnack inequality. The next statement is an application of~\cite[Theorem 1.1]{Tru_1967} in a Riemannian manifold setting.

\begin{theorem}
[Harnack Inequality for $p$-harmonic functions, \cite{Tru_1967}]
\label{thm_Har}
Let $(M,g)$ be a $n$-dimensional Riemannian manifold
and 
let $(\mathcal U,y^1,\hdots,y^n)$ be a coordinate
chart. let $0<\alpha< A$ be such that
\[
\alpha|\xi|^2
\,\leq\, g_{ij}(y)\xi^i\xi^j
\,\leq\,
A|\xi|^2,
\qquad y\in\mathcal U,\quad\xi\in\R^n.
\]
Then, for every
$1<p<+\infty$, every nonnegative
$p$-harmonic function $u$ in $\mathcal U$, and every geodesic ball $B_{2\rho}\subset\mathcal U$, it holds
\[
\sup_{B_{\rho}} u
\,\leq\,
C
\inf_{B_{\rho}} u,
\]
for some
positive constant $C$
depending only on $n$, $p$, 
$\alpha$, and $A$.
\end{theorem}

\subsection{Asymptotics of the \texorpdfstring{$p$}{p}-Green's functions near the pole}

We next recall a description of the singular behavior of a $p$-Green's function close to its pole. This provides the precise order of growth (up to an arbitrarily small multiplicative error) in local coordinates, under uniform ellipticity bounds for the metric coefficients.

\begin{theorem}
[\cite{Serrin_1965,MRS}]
\label{fund_growth}
Let $(M,g)$ be a $n$-dimensional Riemannian manifold, $(\mathcal U,y^1,\hdots,y^n)$ be a coordinate chart, and
let $0<\alpha< A$ be such that
\[
\alpha|\xi|^2
\,\leq\, g_{ij}(y)\xi^i\xi^j
\,\leq\,
 A|\xi|^2,
 \qquad y\in\mathcal U,\quad\xi\in\R^n.
\]
For $x\in\mathcal U$ and $1<p<n$,
let $u_x$ be a $p$-Green's function with pole at $x$.
Then, for every $\eta>0$, there exists $r_{\eta}(p,n,\alpha,A)>0$ such that
\[
\Big(\frac{p-1}{n-p}-\eta\Big)|y-x|^{-\frac{n-p}{p-1}}
\,\leq\,
u(y)
\,\leq\,
\Big(\frac{p-1}{n-p}+\eta\Big)|y-x|^{-\frac{n-p}{p-1}},
\qquad\mbox{for}
\quad
0<|y-x|\leq r_{\eta}.
\]
\end{theorem}

\subsection{Interior and boundary \texorpdfstring{${\mathscr C}^{1,\beta}$}{C¹ᵝ} estimates}

Finally, we record a convenient form of the interior and boundary ${\mathscr C}^{1,\beta}$ estimates from \cite{Tolksdorf1984} and \cite{Lieb88}, adapted to the $p$-Laplace operator on a Riemannian manifold and specialized to homogeneous boundary data in the boundary chart setting.

\begin{theorem}
[Interior and boundary ${\mathscr C}^{1,\beta}$ estimates, \cite{Tolksdorf1984,Lieb88}]
\label{thm_est_int_estest}
Let $(M,g)$ be a Riemannian manifold with smooth boundary $\pa M$ and let $1<p<n$.
\begin{itemize}
\item[(i)]
Let $u$ be a $p$-harmonic function in an interior coordinate chart
$(\mathcal U,y^1,\hdots,y^n)$
and let 
$B_{2R}:=\big\{(y^1)^2+\hdots+(y^n)^2<4R^2\big\}$ be a coordinate ball of radius $2R>0$.
Then, there exists a positive constant $C=C\big(p,g_{ij},\pa_k g_{ij}, \|u\|_{L^{\infty}(B_{2R})}\big)$ and a constant $0<\beta<1$
such that
\[
\|u\|_{{\mathscr C}^{1,\beta}(B_R)}
\,\leq\,
C,
\]
where $B_R$ is a coordinate ball of radius $R$.
\smallskip
\item[(ii)]
Let $u$ be a $p$-harmonic function in a boundary coordinate chart
$(\mathcal U,y^1,\hdots,y^n)$
such that $y^n$ vanishes on $\pa M\cap\mathcal U$ and $y^n$ is positive in $(M\setminus\pa M)\cap\mathcal U$,
and let 
$B_{2R}^+:=\big\{(y^1)^2+\hdots+(y^n)^2<4R^2,\,y^n>0\big\}$ 
be a coordinate upper half ball of radius $2R>0$.
Also, assume that
\[
u=0\qquad\mbox{on}\quad\pa (B_{2R}^+)\cap\pa M.
\]
Then, there exists a positive constant $C=C\big(p,g_{ij},\pa_k g_{ij}, \|u\|_{L^{\infty}(B_{2R}^+)}\big)$ and a constant $0<\beta<1$
such that
\[
\|u\|_{{\mathscr C}^{1,\beta}\left(\,\overline{B_R^+}\,\right)}
\,\leq\,
C.
\]
\end{itemize}
\end{theorem}

\section{Weak IMCF with outer obstacle emanating from a pole}\label{app:pole-imcf}

In this appendix, we record the existence and uniqueness statement for the weak inverse mean curvature flow with outer obstacle emanating from a point.  The result is the pole-normalized counterpart of Xu's existence theorem for the initial value problem with outer obstacle \cite[Theorem 6.1]{Xu}.  We use it in the paper with the notation introduced in Subsection 1.4, namely
\[
\Omega_{x,t}^{(1)}=\{w_x^{(1)}<t\},
\qquad
\Sigma_{x,t}^{(1)}=\{w_x^{(1)}=t\},
\]
whenever the latter level set is sufficiently regular.

Throughout this appendix, $(M,g)$ is a compact connected three-dimensional Riemannian manifold with smooth connected boundary $\partial M$. We denote by $\nu$ the outer unit normal to $\partial M$.  The minimality of $\partial M$, assumed in 
Theorems \ref{thm_eich}
and \ref{thm:RPI-1harm}, is not used in the construction below. 

\medskip

\noindent \textbf{Ambient-perimeter convention.} \textit{
As in Xu's formulation, $M$ is tacitly realized as a smooth precompact domain in a boundaryless Riemannian manifold.
Accordingly, every perimeter
is computed in that ambient manifold. In particular, if a competitor meets $\partial M$,
the part of its reduced boundary contained in $\partial M$ contributes to its perimeter;
this is the convention defining the outer-obstacle problem in
\cite[Definitions~3.1 and~3.2]{Xu}.
}

\subsection{Statement of the pole problem}

We first 
introduce the weak formulation.  Since the weak inverse mean curvature flow may contain jump regions, the vector field $\nabla w/|\nabla w|$ is replaced by a calibration.
If $U\subset M\setminus\partial M$ is open, a vector field $\xi\in L^\infty(U;TU)$ is said to calibrate a weak solution $w\in \Lip_\loc(U)$ if
\begin{equation}\label{eq:calibration-appendix}
|\xi|\le 1,\qquad
\langle \nabla w,\xi\rangle=|\nabla w|\quad\hbox{a.e. in }U,
\qquad
\divg \xi=|\nabla w|,
\end{equation}
in the sense of distributions. This is Xu's calibrated formulation
\cite[Definition~2.11]{Xu}. At every differentiability point at which $\nabla w\neq0$,
\eqref{eq:calibration-appendix} forces
$\xi=\nabla w/|\nabla w|$.

We shall use the normal trace in the precise sense of
\cite[Subsection~3.4, equations~(3.20)--(3.21)]{Xu}. Namely, if $\mathcal C_a$ is a
collar of $\partial M$ and
$Y\in L^\infty(\mathcal C_a;TM)$ satisfies
$\divg Y\in L^1(\mathcal C_a)$, then its normal trace
$[Y\cdot\nu]\in L^\infty(\partial M)$ is well defined and satisfies the corresponding
Gauss-Green formula.

\begin{definition}[Pole-normalized obstacle solution]
\label{def:pole-normalized-solution}
Let $x\in M\setminus\partial M$.  A pair $(w,\xi)$ is called a pole-normalized weak solution of the inverse mean curvature flow with outer obstacle $\pa M$ and pole $x$ if the following conditions are satisfied.
\begin{enumerate}[label=\textnormal{(\roman*)}]
\item $w\in \Lip_\loc\big(M\setminus(\partial M\cup\{x\})\big)\cap C^0_\loc(M\setminus\{x\})$. Moreover, for some collar $\mathcal C_a$ of $\partial M$,
$|\nabla w|\in L^1(\mathcal C_a)$.
\item The pair $(w,\xi)$ satisfies the calibration 
conditions
\eqref{eq:calibration-appendix} on $M\setminus(\partial M\cup\{x\})$.
\item The normal trace of $\xi$ on $\partial M$ satisfies
\begin{equation}\label{eq:obstacle-trace}
\xi\cdot\nu=1\qquad\hbox{a.e. on }\partial M .
\end{equation}
\item 
The following logarithmic normalization at the pole holds:
\begin{equation}
\label{eq:pole-normalization-appendix}
\lim_{y\to x}\bigl(w(y)-2\log d_g(x,y)\bigr)=0.
\end{equation}
\end{enumerate}
The sublevel sets 
of $w$ will be denoted by
\[
\Omega_{x,t}=\{w<t\}.
\]
\end{definition}
By \cite[Lemma~3.19]{Xu}, condition \eqref{eq:obstacle-trace} implies that $w$
respects $\partial M$ as an outer obstacle in Xu's variational sense. Condition \eqref{eq:pole-normalization-appendix} fixes the additive constant. In particular, by
this normalization the calibrated field $e^{-w}\xi$ has Euclidean flux $4\pi$ at the pole, as made precise in the following theorem.

\begin{theorem}
[Weak IMCF with outer obstacle coming out of a pole]\label{thm:pole-existence}
Let $(M,g)$ be as above and let $x\in M\setminus\partial M$. Then there exists a pole-normalized weak solution $(w_x^{(1)},\xi_x^{(1)})$ in the sense of Definition~\ref{def:pole-normalized-solution}. More precisely, for every sufficiently
small $\rho>0$, let $v_{x,\rho}$ be the unique weak solution supplied by
\cite[Theorem~6.1]{Xu} for the initial set $B_\rho(x)$ and the outer obstacle
$\partial M$, and set
\begin{equation}\label{eq:def-w-rho}
w_{x,\rho}=v_{x,\rho}+2\log\rho
\qquad\text{on }M\setminus B_\rho(x).
\end{equation}
Then
\begin{equation}\label{eq:local-uniform-convergence-pole}
w_{x,\rho}\longrightarrow w_x^{(1)}
\qquad\text{locally uniformly in }M\setminus\{x\}
\end{equation}
as $\rho\to0^+$, independently of the sequence along which the limit is taken. There is
a calibration $\xi_x^{(1)}$ for which
\begin{equation}\label{eq:dirac-with-boundary}
 -\int_M e^{-w_x^{(1)}}
\langle\xi_x^{(1)},\nabla\varphi\rangle\dmu
+\int_{\partial M}e^{-w_x^{(1)}}\varphi\dsigma
 =4\pi\varphi(x)
\end{equation}
for every $\varphi\in C^1(M)$. In particular,
\begin{equation}
\label{eq:dirac-interior}
\divg \bigl(e^{-w_x^{(1)}}\xi_x^{(1)}\bigr)
=4\pi\delta_x
\qquad\text{in }\mathcal D'(M \setminus \pa M).
\end{equation}
The function $w_x^{(1)}$ is the unique pole-normalized weak obstacle solution. Any
calibration is determined a.e. on the set $\{|\nabla w_x^{(1)}|>0\}$; no uniqueness is
asserted on $\{|\nabla w_x^{(1)}|=0\}$.
\end{theorem}

The rest of the appendix is devoted to the proof of this result. We keep the pole $x$ fixed and suppress it from the notation whenever this does not create ambiguity.

\subsection{The approximating obstacle problems}

Choose $\rho_0>0$ such that $B_{4\rho_0}(x)$ is compactly contained in $ M \setminus \pa M$ and
$B_{4\rho_0}(x)$ is a geodesic ball. After decreasing $\rho_0$, every
$B_\rho(x)$, $0<\rho<\rho_0$, is strictly outward minimizing in $M$. Indeed, the
geodesic spheres in $B_{2\rho_0}(x)$ form a strictly mean-convex foliation, as it is proven in the following lemma.

\begin{lemma}
[Strict outward minimality of small geodesic balls]
\label{lem:small-balls-strictly-outward-minimizing}
Let $(M,g)$ be a compact Riemannian three-manifold with smooth boundary,
and let $x\in M\setminus\partial M$. Then there exists $\rho_x>0$ such
that $B_\rho(x)$ is strictly outward minimizing in $M$ for every
$0<\rho<\rho_x$, with respect to the ambient-perimeter convention
adopted above.
\end{lemma}

\begin{proof}
This follows from the bounded-domain localization of the blow-up
argument in \cite[Theorem~5.7]{MondinoTempletonBrowne}, see also
\cite[Remark~7.9]{MondinoTempletonBrowne}. Indeed, applying that
argument with vanishing graph function $w=0$ shows that every
perimeter minimizer among the finite-perimeter sets
\[
E\subset M,
\qquad B_\rho(x)\subset E,
\]
coincides almost everywhere with $B_\rho(x)$, provided $\rho$ is
sufficiently small. More precisely, the final first-variation
argument in \cite[proof of Theorem~5.7, Step~4]
{MondinoTempletonBrowne} shows that equality of the perimeters is
possible only when the graph describing the minimizing hull over
$\partial B_\rho(x)$ vanishes identically. Hence
\[
{\rm Per}(E)= {\rm Per}(B_\rho(x))
\quad\Longrightarrow\quad
|E\setminus B_\rho(x)|=0,
\]
which is precisely the strict outward-minimizing property. Compare
also~\cite[Remark~2.7]{FogagnoloMazzieri}.
\end{proof}

\smallskip

For $0<\rho<\rho_0$, let $v_\rho$ be the solution given by
\cite[Theorem~6.1]{Xu}. By the initial-value condition
\cite[Definition~3.6]{Xu} and the continuity in \cite[Theorem~6.1(i)]{Xu},
$v_\rho=0$ on $\partial B_\rho(x)$. Hence
\begin{equation}\label{eq:inner-boundary-value}
w_\rho=2\log\rho
\qquad\text{on }\partial B_\rho(x).
\end{equation}
Let $\xi_\rho$ be the calibration furnished by Xu's construction. In
$M\setminus\overline{B_\rho(x)}$,
\begin{equation}
\label{eq:homogeneous-calibrated}
|\xi_\rho|\leq1,
\qquad
\langle\nabla w_\rho,\xi_\rho\rangle=|\nabla w_\rho|,
\qquad
\divg\xi_\rho=|\nabla w_\rho|.
\end{equation}
Consequently,
\begin{equation}\label{eq:weighted-divergence-zero}
\divg \bigl(e^{-w_\rho}\xi_\rho\bigr)=0
\qquad\text{in }M\setminus\overline{B_\rho(x)}.
\end{equation}
Estimate \cite[Theorem~6.1(ii), equation~(6.3)]{Xu}, together with Xu's normal-trace
formula \cite[equation~(3.20)]{Xu}, gives
\begin{equation}\label{eq:outer-trace-rho}
[\xi_\rho\cdot\nu]=1
\qquad\calH^2\text{-a.e. on }\partial M.
\end{equation}

We also need the trace on the initial sphere. Since $B_\rho(x)$ is strictly outward
minimizing, Xu's minimizing-hull results
\cite[Theorems~3.12 and~3.13]{Xu} show that its initial minimizing hull is
$B_\rho(x)$ itself. Hence the exact interior area law
\cite[Lemma~2.6(iv)]{Xu} applies until the first contact with $\partial M$:
if $E_{\rho,t}=B_\rho(x)\cup\{y\in M\setminus B_\rho(x):v_\rho(y)<t\}$
and $t>0$ is sufficiently small, then
\begin{equation}
\label{eq:approx-area-law}
\Per(E_{\rho,t})=e^t \, |\partial B_\rho(x)|.
\end{equation}
For a.e. such $t$, the calibration equals the measure-theoretic outer unit normal of
$E_{\rho,t}$ on its reduced boundary by \cite[Remark~2.12]{Xu}. The Gauss--Green
formula for \eqref{eq:homogeneous-calibrated}, applied to
$E_{\rho,t}\setminus\overline{B_\rho(x)}$, and the coarea formula therefore give
\begin{align*}
\int_{\partial B_\rho(x)}[\xi_\rho\cdot\eta_\rho]\,d\sigma
&=\Per(E_{\rho,t})-
\int_{E_{\rho,t}\setminus B_\rho(x)}|\nabla v_\rho|\dmu\\
&=e^t \, |\partial B_\rho(x)|
-\int_0^t e^s \, |\partial B_\rho(x)|
\,ds\\
&= |\partial B_\rho(x)|,
\end{align*}
where $\eta_\rho$ is the outer unit normal to $B_\rho(x)$. Since
$|[\xi_\rho\cdot\eta_\rho]|\leq1$, equality of the integrals implies
\begin{equation}
\label{eq:inner-trace-rho}
[\xi_\rho\cdot\eta_\rho]=1
\qquad\calH^2
\text{-a.e. on }
\partial B_\rho(x).
\end{equation}
Equations \eqref{eq:weighted-divergence-zero},
\eqref{eq:outer-trace-rho}, and \eqref{eq:inner-trace-rho} are the weak form of
\begin{equation}\label{eq:rho-problem-formal}
\begin{cases}
\displaystyle
\divg\!\left(e^{-w_\rho}\frac{\nabla w_\rho}{|\nabla w_\rho|}\right)=0
&\text{in }M\setminus B_\rho(x),\\[0.6em]
w_\rho=2\log\rho
&\text{on }\partial B_\rho(x),\\[0.4em]
\displaystyle
\frac{\nabla w_\rho}{|\nabla w_\rho|}\longrightarrow\nu
&\text{at }\partial M,
\end{cases}
\end{equation}
with $\nabla w_\rho/|\nabla w_\rho|$ replaced by $\xi_\rho$ on jump regions.

\subsection{Uniform pole normalization}

The next lemma provides a local estimate that 
will be used to obtain the logarithmic normalization \eqref{eq:pole-normalization-appendix}. Let us set
\[
r(y)=d_g(x,y).
\]

\begin{lemma}[Local barriers near the pole]\label{lem:pole-barriers-appendix}
There exist constants $r_0\in(0,\rho_0)$ and $A>0$, depending only on the geometry of $g$ in $B_{2r_0}(x)$, such that for every $0<\rho<r_0/2$ and every $y\in B_{r_0}(x)\setminus B_\rho(x)$ one has
\begin{equation}\label{eq:barrier-estimate}
2\log r(y)-A\bigl(r(y)^2-\rho^2\bigr)
\le w_\rho(y)\le
2\log r(y)+A\bigl(r(y)^2-\rho^2\bigr).
\end{equation}
\end{lemma}

\begin{proof}
In geodesic polar coordinates around $x$, the mean curvature $H_r$ of the geodesic sphere $\partial B_r(x)$, with respect to the outer normal $\partial_r$, satisfies
\begin{equation}\label{eq:mean-curvature-small-sphere}
H_r=\frac2r+O(r)
\qquad\hbox{as }r\to0.
\end{equation}
Hence, there exist $A>0$ and $r_0>0$ such that
\[
\frac2r-2Ar\le H_r\le \frac2r+2Ar,
\qquad\hbox{for }0<r<r_0.
\]
Define
\[
\psi_-(r)=2\log r-Ar^2,\qquad
\psi_+(r)=2\log r+Ar^2.
\]
After decreasing $r_0$, both functions are strictly increasing in $(0,r_0)$ and 
\[
|\nabla\psi_\pm| \, = \, \frac2r \pm 2Ar \,, \qquad\hbox{for }0<r<r_0.
\]
In particular, their level-set normal is $\partial_r$, and the previous inequalities give
\begin{equation}
\label{eq:radial-sub-super}
\divg\left(\frac{\nabla\psi_-}{|\nabla\psi_-|}\right)
\ge |\nabla\psi_-|,
\qquad
\divg\left(\frac{\nabla\psi_+}{|\nabla\psi_+|}\right)
\le |\nabla\psi_+| \, ,
\qquad\text{in }B_{r_0}(x)\setminus\{x\}.  
\end{equation}
Thus $\psi_-$ and $\psi_+$ are, respectively, a smooth subsolution and a smooth supersolution of the level-set IMCF equation in the punctured ball. 
Let now $V_0$ be Xu's obstacle solution with initial set $B_{r_0}(x)$, normalized by
$V_0=0$ on $\partial B_{r_0}(x)$, and let $\Xi_0$ be its calibration. Define
\begin{equation}\label{eq:pasted-lower-upper}
\underline v_\rho=
\begin{cases}
\psi_-(r)-\psi_-(\rho),&\rho<r<r_0,\\
\psi_-(r_0)-\psi_-(\rho)+V_0,&M\setminus B_{r_0}(x),
\end{cases}
\end{equation}
and
\begin{equation}\label{eq:pasted-upper}
\overline v_\rho=
\begin{cases}
\psi_+(r)-\psi_+(\rho),&\rho<r<r_0,\\
\psi_+(r_0)-\psi_+(\rho)+V_0,&M\setminus B_{r_0}(x).
\end{cases}
\end{equation}
We claim that that $\underline v_\rho$ is a subsolution and
$\overline v_\rho$ is a supersolution for the initial set $B_\rho(x)$ and the outer
obstacle $\partial M$. First observe that the trace computation leading to \eqref{eq:inner-trace-rho}, applied at radius
$r_0$, gives $[\Xi_0\cdot\eta_{r_0}]=1$ on $\partial B_{r_0}(x)$. Consequently, the
piecewise field
\[
\Xi=\begin{cases}
\partial_r,&B_{r_0}(x)\setminus B_\rho(x),\\
\Xi_0,&M\setminus B_{r_0}(x),
\end{cases}
\]
has matching normal components across $\partial B_{r_0}(x)$, so its distributional
divergence has no interface measure. Combining \eqref{eq:radial-sub-super} with the
calibration identities for $(V_0,\Xi_0)$ yields
\[
\divg\Xi\geq|\nabla\underline v_\rho|,
\qquad
\divg\Xi\leq|\nabla\overline v_\rho|,
\]
with equality between $\Xi$ and the corresponding gradient directions and with outer
trace one on $\partial M$. 
The variational characterization
\cite[Theorem~3.18]{Xu} therefore shows that $\underline v_\rho$ is a subsolution and
$\overline v_\rho$ is a supersolution for the initial set $B_\rho(x)$ and the outer
obstacle $\partial M$.

The maximality comparison \cite[Corollary~3.23]{Xu} gives
$\underline v_\rho\leq v_\rho$. For the opposite comparison, apply Xu's weak maximum
principle \cite[Theorem~3.21 and Remark~3.22(ii)]{Xu} to the supersolution
$\overline v_\rho+\delta$ and the subsolution $v_\rho$, and then let $\delta\downarrow0$.
The positive shift makes the comparison set disjoint from the artificial inner boundary.
Thus
\[
\underline v_\rho\leq v_\rho\leq\overline v_\rho
\qquad\text{on}\quad M\setminus B_\rho(x).
\]
Restricting to $B_{r_0}(x)\setminus B_\rho(x)$ and adding $2\log\rho$ proves
\eqref{eq:barrier-estimate}.
\end{proof}

Letting $\rho\to0$ in \eqref{eq:barrier-estimate} will give the required logarithmic asymptotics.  We also need bounds away from the pole.

\begin{corollary}[Bounds on compact subsets away from the pole]
\label{cor:compact-bounds-away-pole}
For every compact set $K\subset M\setminus\{x\}$ there exists a constant $C_K$ such that
\[
|w_\rho|\le C_K\qquad\hbox{on }K
\]
for every sufficiently small $\rho$.
\end{corollary}

\begin{proof}
The assertion is immediate from Lemma~\ref{lem:pole-barriers-appendix} if $K\subset B_{r_0}(x)\setminus\{x\}$.  In the complementary region $M\setminus B_{r_0}(x)$ the comparison functions constructed in the proof of Lemma~\ref{lem:pole-barriers-appendix} give
\[
2\log r_0-Ar_0^2+V_0-O(\rho^2)
\le w_\rho\le
2\log r_0+Ar_0^2+V_0+O(\rho^2),
\]
and $V_0$ is bounded on $M\setminus B_{r_0}(x)$ by Xu's theorem.
\end{proof}

\subsection{Compactness and passage to the limit}

We now isolate the compactness statement that will be used in the limiting procedure.  This is also the only point in the construction where one needs the quantitative estimates in Xu's existence theory. 

\begin{lemma}[Uniform estimates and compactness]
\label{lem:local-compactness-pole}
Let $\gamma \in(0,1)$ be the dimensional exponent in
\cite[Theorem~6.1 and Proposition~6.5]{Xu}.
Then, the following assertions hold.
\begin{enumerate}[label=\textnormal{(\roman*)}]
\smallskip
\item \emph{Uniform interior and boundary estimates.}  
For every compact set
\[
K\subset M\setminus(\partial M\cup\{x\}),
\]
there exists a constant $C_K>0$, independent of $\rho$, such that
\begin{equation}\label{eq:uniform-interior-gradient-rho}
\|\nabla w_\rho\|_{L^\infty(K)}\le C_K
\end{equation}
for all sufficiently small $\rho>0$.  Moreover, there exist a collar neighborhood
\[
\mathcal C_a=\{y\in M: d_\partial(y)<a\},
\qquad d_\partial(y)=d_g(y,\partial M),
\]
and a constant $C_\partial>0$, both independent of $\rho$, such that, writing
$\nu_\partial=-\nabla d_\partial$ in $\mathcal C_a$, one has
\begin{equation}
\label{eq:uniform-boundary-gradient-calibration-rho}
|\nabla w_\rho|(y)\le C_\partial d_\partial(y)^{\gamma-1},
\qquad
1-\langle \xi_\rho,\nu_\partial\rangle(y)
\le C_\partial d_\partial(y)^\gamma
\end{equation}
for a.e. $y\in\mathcal C_a\setminus B_\rho(x)$.  In particular, after possibly decreasing $a$,
\begin{equation}\label{eq:uniform-boundary-holder-rho}
|w_\rho(y)-w_\rho(z)|\le C_\partial d_g(y,z)^\gamma
\qquad\hbox{for all }y,z\in \overline{\mathcal C}_{a/2}.
\end{equation}
Consequently, $\{w_\rho\}_{\rho>0}$ is locally equicontinuous in $M\setminus\{x\}$: it is locally uniformly Lipschitz in the interior and locally uniformly H\"older continuous up to the obstacle, away from the pole.

\smallskip
\item 
\emph{Subsequential convergence.}  For every sequence $\rho_j\to0^+$ there exist a subsequence, not relabeled, a function
\[
w\in \Lip_{\loc}\bigl(M\setminus(\partial M\cup\{x\})\bigr)
\cap C^0_{\loc}(M\setminus\{x\}),
\]
and a vector field $\xi\in L^\infty_{\loc}(M\setminus\{x\};TM)$ such that
\begin{equation}\label{eq:compactness-subsequence}
\qquad\quad    w_{\rho_j}\longrightarrow w
\qquad\hbox{locally uniformly in }M\setminus\{x\},
\quad \hbox{and} \quad
\xi_{\rho_j}\overset{*}{\rightharpoonup}\xi
\qquad\hbox{in }L^\infty_{\loc}(M\setminus\{x\}).
\end{equation}
The pair $(w,\xi)$ is calibrated in $M\setminus(\partial M\cup\{x\})$, namely it satisfies \eqref{eq:calibration-appendix} there.  In addition, $\xi$ satisfies the boundary estimate
\begin{equation}\label{eq:limit-boundary-calibration-estimate}
1-\langle \xi,\nu_\partial\rangle
\le C_\partial d_\partial^\gamma
\qquad\hbox{a.e. in }\mathcal C_a,
\end{equation}
so that $\xi\cdot\nu=1$ on $\partial M$ in the normal trace sense.  Hence, $w$ respects the outer obstacle $\partial M$.
\end{enumerate}
\end{lemma}

\begin{proof}
Since $w_\rho-v_\rho$ is constant, all gradient and calibration estimates for $v_\rho$
apply unchanged to $w_\rho$. Fix a compact subset $K\subset M\setminus(\partial M\cup\{x\})$ and set
\[
d_K=\min\{d_g(K,x),d_g(K,\partial M)\}>0.
\]
Let $\sigma(\,\cdot\,;M,g)$ be the quantity defined in
\cite[Definition~2.14]{Xu}, and put
\[
\sigma_K=\inf_{y\in K}\sigma(y;M,g)>0,\qquad
r_K=\frac14\min\{d_K,\sigma_K\}.
\]
If $\rho<d_K/4$, then
$B_{r_K}(y)\cap\partial B_\rho(x)=\emptyset$ for every $y\in K$. In
\cite[Theorem~6.1(i), equation~(6.1)]{Xu}, the term involving the positive mean
curvature of the initial boundary is therefore absent, and
\[
|\nabla w_\rho|(y)=|\nabla v_\rho|(y)
\leq\frac{C(3)}{r_K}
\qquad\text{for a.e. }y\in K.
\]
This proves \eqref{eq:uniform-interior-gradient-rho}.

Choose $a>0$ so small that the signed distance to $\partial M$ is smooth in
$\mathcal C_{4a}$ and
$\mathcal C_{4a}\cap B_{4\rho_0}(x)=\emptyset$. We may also choose, once and for all,
a radius $r_I>4a$ satisfying the geometric requirements in
\cite[Subsection~6.1]{Xu} and
\[
 \dist(B_\rho(x),\partial M)>3r_I
\]
for every sufficiently small $\rho$. With this fixed choice, the proof of
\cite[Proposition~6.5, equations~(6.15)--(6.16)]{Xu} gives the same collar radius and
the same constant for the entire family. These estimates are exactly
\eqref{eq:uniform-boundary-gradient-calibration-rho}. The proof of
\cite[Proposition~6.6]{Xu}, applied with the same $r_I$, integrates the gradient bound
in boundary normal coordinates and yields
\eqref{eq:uniform-boundary-holder-rho}, again with a constant independent of
$\rho$. Notice also that
\[
\sup_\rho\int_{\mathcal C_a}|\nabla w_\rho|\dmu
\leq C_\partial\int_{\mathcal C_a}d_\partial^{\gamma-1}\dmu
<\infty.
\]

Let now $\rho_j\to0^+$. Corollary~\ref{cor:compact-bounds-away-pole},
\eqref{eq:uniform-interior-gradient-rho}, and
\eqref{eq:uniform-boundary-holder-rho} give local equiboundedness and local
equicontinuity in $M\setminus\{x\}$. Arzel\`a--Ascoli and a diagonal argument produce
the locally uniform convergence in \eqref{eq:compactness-subsequence}. Since
$|\xi_{\rho_j}|\leq1$, Banach-Alaoglu gives the weak-star convergence there.

On each compact subset of $M\setminus (\pa M \cup \{x\})$, Xu's compactness theorem for
calibrated solutions \cite[Theorem~2.13]{Xu} applies and yields
\eqref{eq:calibration-appendix} for $(w,\xi)$. For every $\delta>0$, weak compactness on
$\mathcal C_a\cap\{d_\partial>\delta\}$ passes the first estimate in
\eqref{eq:uniform-boundary-gradient-calibration-rho} to $w$. Letting
$\delta\downarrow0$ and using monotone convergence gives
$|\nabla w|\in L^1(\mathcal C_a)$. Passing to the weak-star limit in the second
estimate in
\eqref{eq:uniform-boundary-gradient-calibration-rho} 
gives
\[
1-\langle\xi,\nu_\partial\rangle
\leq C_\partial d_\partial^\gamma
\qquad\text{a.e. in }\mathcal C_a.
\]
Hence
\[
|\xi-\nu_\partial|^2
\leq2\bigl(1-\langle\xi,\nu_\partial\rangle\bigr)
\leq2C_\partial d_\partial^\gamma.
\]
The normal-trace formula \cite[equation~(3.20)]{Xu} now gives
$[\xi\cdot\nu]=1$ on $\partial M$. Equivalently, one may first invoke the quantitative
criterion \cite[Corollary~3.20]{Xu} and then the Gauss--Green characterization in
\cite[Subsection~3.4]{Xu}.
\end{proof}

For any subsequential limit $w$, Lemma~\ref{lem:pole-barriers-appendix} gives
\begin{equation}\label{eq:limit-pole-asymptotics}
 |w(y)-2\log d_g(x,y)|\leq A\,d_g(x,y)^2
 \qquad(0<d_g(x,y)<r_0).
\end{equation}
Thus \eqref{eq:pole-normalization-appendix} holds.
This proves the existence of a pole-normalized weak obstacle solution, up to the identification of the Dirac mass, which is carried out next.

\subsection{The Dirac mass and the coefficient \texorpdfstring{$4\pi$}{4π}}

In this section, our aim is to establish identity~\eqref{eq:dirac-with-boundary}.
Note first that,
by Lemma~\ref{lem:pole-barriers-appendix}
and by $|\xi_\rho| \leq 1$, one has that
\begin{equation}
\label{eq:per_dopo}
\left|e^{-w_\rho}\xi_\rho\right|
\leq C r^{-2},
\qquad\mbox{in\  }B_{r_0}(x)\setminus B_\rho(x), 
\end{equation}
where we recall that 
$r=d_g(x,\cdot)$ and the constant $C$ is independent on $\rho$. 
Let $\varphi\in C^1(M)$ be a test function.  Applying the divergence theorem to \eqref{eq:weighted-divergence-zero} on $M\setminus B_\rho(x)$ and using the traces \eqref{eq:outer-trace-rho} and \eqref{eq:inner-trace-rho}, we obtain
\begin{align}
\label{eq:green-rho}
-\int_{M\setminus \overline{B_\rho}(x)} e^{-w_\rho}\langle\xi_\rho,\nabla\varphi\rangle\dmu
+\int_{\partial M} e^{-w_\rho}\varphi\dsigma
&=\int_{\partial B_\rho(x)} e^{-2\log\rho}\varphi\dsigma \\
&=\rho^{-2}\int_{\partial B_\rho(x)}\varphi\dsigma . \nonumber
\end{align}
Observe that, since
\[
|\partial B_\rho(x)|=4\pi\rho^2+O(\rho^4),
\qquad
\varphi=\varphi(x)+O(\rho)
\quad\hbox{on }\partial B_\rho(x),
\]
the right-hand side of \eqref{eq:green-rho} converges to $4\pi\varphi(x)$, as $\rho \to 0$.
To treat the left-hand side, we write, for any fixed $s>0$ and for every $\rho$ smaller than $s$
\[
-\int_{M \setminus \overline{B_\rho}(x)} \!\!\!\!\!\!\!e^{-w_\rho}\langle\xi_\rho,\nabla\varphi\rangle\dmu = -\int_{M\setminus \overline{B_s}(x)}\!\!\!\!\!\!\! e^{-w_\rho}\langle\xi_\rho,\nabla\varphi\rangle\dmu -\int_{ B_s(x) \setminus \overline{B_\rho}(x)}\!\!\!\!\!\!\! e^{-w_\rho}\langle\xi_\rho,\nabla\varphi\rangle\dmu.
\]
Hence, with the help of the estimate~\eqref{eq:per_dopo},we get 
\begin{align*}
\left| -\int_{M\setminus \overline{B_s}(x)} \!\!\!\!\!\!\!e^{-w_\rho}\langle\xi_\rho,\nabla\varphi\rangle\dmu  +\int_{\partial M} e^{-w_\rho}\varphi\dsigma -  \rho^{-2}\int_{\partial B_\rho(x)}\!\!\!\!\!\!\!\varphi\dsigma \right| &= \left| \int_{ B_s(x) \setminus \overline{B_\rho}(x)}\!\!\!\!\!\!\! e^{-w_\rho}\langle\xi_\rho,\nabla\varphi\rangle\dmu  \right| \\
\leq \, C \, \|\nabla\varphi\|_{L^\infty}\int_{B_s(x)} \!\!\!\!\!\!\!r^{-2} \, &\leq \, C' \, s \, , 
\end{align*}
for some constants $C$ and $C'$, independent of $\rho$.
The locally uniform convergence of $w_\rho$ and the weak star convergence of $\xi_\rho$ imply
\begin{align*}
& -\int_{M\setminus \overline{B_s}(x)} \!\!\!\!\!\!\!e^{-w_\rho}\langle\xi_\rho,\nabla\varphi\rangle\dmu
 +\int_{\partial M} e^{-w_\rho}\varphi\dsigma \\
&\hspace{3cm}\longrightarrow
\int_{M\setminus \overline{B_s}(x)} \!\!\!\!\!\!\!e^{-w}\langle\xi,\nabla\varphi\rangle\dmu
 +\int_{\partial M} 
e^{-w}\varphi\dsigma,\quad \hbox{as $\rho \to 0$.}
\end{align*}
Letting first $\rho\to0$, gives
\begin{align*}
\left| -\int_{M\setminus \overline{B_s}(x)} \!\!\!\!\!\!\!e^{-w}\langle\xi,\nabla\varphi\rangle\dmu  +\int_{\partial M} e^{-w}\varphi\dsigma  \, -  \,  4 \pi \varphi(x) \right| \leq \, C' \, s \, .
\end{align*}
Leaving $s\to0$ gives \eqref{eq:dirac-with-boundary}.  If $\varphi$ is compactly supported in $M\setminus\partial M$, the boundary term vanishes, and \eqref{eq:dirac-interior} follows.

\subsection{Uniqueness} 

We shall use the following localized form of Xu's weak maximum principle. Suppose that
an open domain has a genuine obstacle part $\Gamma$ and an additional artificial
boundary part $\Gamma_0$. If $u$ is a solution and $v$ is a subsolution near the
comparison set, both satisfy the outer-obstacle condition on $\Gamma$, and
\[
\overline{\{u<v\}}\cap\Gamma_0=\emptyset,
\]
then the conclusion $u\geq v$ follows exactly as in
\cite[Theorem~3.21]{Xu}, provided $u$ is bounded below on a compact neighborhood of the
comparison set. Indeed, that compact neighborhood can be chosen disjoint from
$\Gamma_0$; all competitors in Xu's proof, specifically those used in
(3.24)--(3.28), are supported there, so no boundary term on $\Gamma_0$ occurs. This is a
localization of the proof, not an additional boundary-value theorem.

\begin{proposition}
[Comparison at the pole]
\label{prop:pole-comparison}
Let $(w,\xi)$ and $(z,\zeta)$ be two pole-normalized weak obstacle solutions with pole $x$.  Then $w=z$ in $M\setminus\{x\}$.
\end{proposition}

\begin{proof}
Fix $\eps>0$. The pole normalization gives $s>0$ such that
\begin{equation}
\label{eq:inner-comparison}
w+\eps>z
\qquad\text{in }B_s(x)\setminus\{x\}.
\end{equation}
By \cite[Lemma~3.19]{Xu}, both functions solve the variational outer-obstacle problem.
Adding a constant preserves that problem. The comparison set
\[
A_\eps=\{w+\eps<z\}
\]
is disjoint from $B_s(x)$ by \eqref{eq:inner-comparison}; it may meet the genuine
obstacle $\partial M$, which is allowed in Xu's formulation. On a compact neighborhood
of $A_\eps$, $w+\eps$ is bounded below. The localized maximum principle described above,
with $\Gamma=\partial M$ and $\Gamma_0=\partial B_s(x)$, therefore gives
\[
w+\eps\geq z
\qquad\text{in }M\setminus B_s(x).
\]
Together with \eqref{eq:inner-comparison}, this holds throughout
$M\setminus\{x\}$. Letting $\eps\downarrow0$ gives $w\geq z$. Interchanging $w$ and
$z$ gives the reverse inequality.
\end{proof}

All subsequential limits in Lemma~\ref{lem:local-compactness-pole} therefore
coincide. Hence the whole family $w_\rho$ converges locally uniformly to a unique function
$w_x^{(1)}$, and Theorem~\ref{thm:pole-existence} follows.

\subsection{Area normalization before first contact}

We conclude by recording the normalization of the perimeter of the sublevel sets before the flow touches the obstacle.  This is the form used in Subsections~\ref{sub:pto1} and~\ref{sub:1harm}, as well as in Section~\ref{sec:harm_1}. 
Let
\begin{equation}
\label{eq:first-contact-time}
T_*(x)=\min_{\partial M} w_x^{(1)},
\qquad
T^*(x)=\max_{\partial M} w_x^{(1)} .
\end{equation}
The continuity up to the obstacle, away from the pole, implies that both quantities are finite.  If $t<T_*(x)$, then $\overline\Omega_{x,t}^{(1)}\subset M\setminus\partial M$ and the obstacle plays no role up to time $t$. 
For a.e. such $t$, $\Omega_{x,t}^{(1)}$ has finite perimeter and
$\xi_x^{(1)}$ agrees with its measure-theoretic outer unit normal on
$\partial^*\Omega_{x,t}^{(1)}$ by \cite[Remark~2.12]{Xu}. Put
$X=e^{-w_x^{(1)}}\xi_x^{(1)}$. For a.e. sufficiently small $s>0$, apply the
Gauss-Green formula to
$\Omega_{x,t}^{(1)}\setminus\overline{B_s(x)}$. Since
$\divg X=0$ away from $x$, one obtains
\begin{equation}
\label{eq:flux-annulus-level}
e^{-t}\Per\bigl(\Omega_{x,t}^{(1)}\bigr)
=\int_{\partial B_s(x)}[X\cdot\eta_s]\,d\sigma.
\end{equation}
The right-hand side equals $4\pi$: this follows from
\eqref{eq:dirac-interior}, tested with cutoffs that are identically one on
$B_s(x)$, or equivalently from \eqref{eq:dirac-with-boundary}. 
Therefore
\begin{equation}
\label{eq:area-growth-pole}
\Per\bigl(\Omega_{x,t}^{(1)}\bigr)=4\pi e^t
\qquad\text{for a.e. }t<T_*(x).
\end{equation}
Whenever $\Sigma_{x,t}^{(1)}=\partial\Omega_{x,t}^{(1)}$ is a smooth level surface,
\begin{equation}\label{eq:smooth-area-growth-pole}
|\Sigma_{x,t}^{(1)}|=4\pi e^t.
\end{equation}

\bigskip
\noindent {\textbf{Computational Tools and AI Disclosure.}} {\textit{
Wolfram Mathematica was used in Appendix A to obtain an explicit and tractable power series representation of the solutions to system (A.1); the resulting computation has been made manually verifiable, at the price of a reasonable computational effort. 
During the final revision of this work,
ChatGPT 5.5 Pro was used to verify 
some routine passages from material supplied by the authors
in the proof of Theorem C.2, which represents a slight variant of the main existence and uniqueness theorem of a recently developed theory of K. Xu.
All mathematical ideas, results, and proofs are the authors' own; all AI-assisted material was reviewed, revised where necessary, and independently verified by the authors, who take full responsibility for the content of the manuscript.
}}

\bigskip
\begin{ackn}

The authors are particularly grateful to the organizers of the conferences and workshops at which this work was presented and discussed at different stages of its development: Alessandro Carlotto and Tristan Rivi\`ere, for \emph{CHANGE -- CHallenges in ANalysis and GEometry} (ETH Z\"urich, 2024); Francesco Bei, Luca Martinazzi, and Paolo Piazza, for \emph{Curvature and Geometric Analysis in Rome} (Sapienza Universit\`a di Roma, 2024); Andrea Malchiodi and Alessandra Pluda, for \emph{Topics in Geometric Analysis} (Centro di Ricerca Matematica Ennio De Giorgi, Pisa, 2025); Robert McCann, Nicola Gigli, and Mathias Braun, for \emph{Optimal transport and metric geometry across structures} (Bernoulli Center, EPFL, Lausanne, 2025); Esther Cabezas-Rivas, for \emph{Valentia Geometrica 26: Workshop on Geometric Analysis} (Universitat de Val\`encia, 2026); Carla Cederbaum and Rodrigo Avalos, for the \emph{43. S\"uddeutsches Kolloquium \"uber Differentialgeometrie} (T\"ubingen, 2026); and Klaus Kr\"oncke, Markus Wolff, and Louis Yudowitz, for \emph{Geometric Analysis: Parabolic and Elliptic Methods} (KTH Stockholm, 2026). They also thank the participants in these meetings for the stimulating questions and discussions that contributed to the evolution of the paper.
The authors would like to thank L.~Benatti, A.~Pisante, A.~Pluda, and M.~Pozzetta, for useful discussions and for their interest in this work. The authors are members of the Gruppo Nazionale per l’Analisi Matematica, la Probabilità e le loro Applicazioni (GNAMPA), which is part of the Istituto Nazionale di Alta Matematica (INdAM). V.~A. and L.~M. gratefully acknowledge the support by the MUR PRIN-2022JJ8KER grant “Contemporary perspectives on geometry and gravity”.
S.~B. has been partially funded by the GNAMPA project ``Flussi geometrici in spazi metrici: criteri di esistenza e applicazioni geometriche'', codice  CUP \#E5324001950001\#,
and by the PRIN Project 2022E9CF89 “Geometric Evolution
Problems and Shape Optimization (GEPSO)”, PNRR Italia Domani, financed by European
Union via the Program NextGenerationEU.
\end{ackn}


\end{document}